\documentclass[a4paper,11pt]{article}
\usepackage{amsmath,amsthm,amssymb}
\usepackage[mathscr]{eucal}
\usepackage[bookmarks=false]{hyperref}

\numberwithin{equation}{section}

{\theoremstyle{definition}\newtheorem{definition}{Definition}[section]

\newtheorem{remark}[definition]{Remark}
}

\newtheorem{proposition}[definition]{Proposition}
\newtheorem{lemma}[definition]{Lemma}
\newtheorem{theorem}[definition]{Theorem}
\newtheorem{corollary}[definition]{Corollary}

\newcommand{\M}{\operatorname{M}}
\newcommand{\C}{\mathbb{C}}
\newcommand{\F}{\mathbb{F}}
\newcommand{\cR}{\mathcal{R}}
\newcommand{\actson}{\curvearrowright}
\newcommand{\SL}{\operatorname{SL}}
\newcommand{\rL}{\mathord{\text{\rm L}}}
\newcommand{\Aut}{\operatorname{Aut}}

\newcommand{\N}{\mathbb{N}}
\newcommand{\T}{\mathbb{T}}
\newcommand{\Z}{\mathbb{Z}}
\newcommand{\cF}{\mathcal{F}}
\newcommand{\cA}{\mathcal{A}}

\newcommand{\id}{\mathord{\operatorname{id}}}
\newcommand{\si}{\sigma}
\newcommand{\cE}{\mathcal{E}}
\newcommand{\recht}{\rightarrow}
\newcommand{\cU}{\mathcal{U}}
\newcommand{\vphi}{\varphi}

\newcommand{\R}{\mathbb{R}}
\newcommand{\al}{\alpha}
\newcommand{\eps}{\varepsilon}
\newcommand{\Tr}{\operatorname{Tr}}
\newcommand{\ovt}{\mathbin{\overline{\otimes}}}
\newcommand{\B}{\operatorname{B}}
\newcommand{\om}{\omega}
\newcommand{\cP}{\mathcal{P}}
\newcommand{\cZ}{\mathcal{Z}}

\newcommand{\Ker}{\operatorname{Ker}}
\newcommand{\cK}{\mathcal{K}}

\newcommand{\cH}{\mathcal{H}}
\newcommand{\cJ}{\mathcal{J}}
\newcommand{\ot}{\otimes}
\newcommand{\cL}{\mathcal{L}}

\newcommand{\Ad}{\operatorname{Ad}}
\newcommand{\cG}{\mathcal{G}}
\newcommand{\cM}{\mathcal{M}}
\newcommand{\dpr}{^{\prime\prime}}

\newcommand{\vphitil}{\widetilde{\vphi}}

\newcommand{\lspan}{\operatorname{span}}
\newcommand{\Mtil}{\widetilde{M}}

\newcommand{\D}{\operatorname{D}}
\newcommand{\Stab}{\operatorname{Stab}}

\newcommand{\cN}{\mathcal{N}}

\newcommand{\cS}{\mathcal{S}}

\newcommand{\Om}{\Omega}
\newcommand{\supp}{\operatorname{supp}}

\newcommand{\cC}{\mathcal{C}}

\newcommand{\cO}{\mathcal{O}}
\newcommand{\cNtil}{\widetilde{\mathcal{N}}}
\newcommand{\cHtil}{\widetilde{\mathcal{H}}}
\newcommand{\otalg}{\otimes_{\text{\rm alg}}}
\newcommand{\otmin}{\otimes_{\text{\rm min}}}

\newcommand{\ubar}{\overline{u}}
\newcommand{\abar}{\overline{a}}
\newcommand{\Gammah}{\widehat{\Gamma}}
\newcommand{\Ktil}{\widetilde{K}}
\newcommand{\cJtil}{\widetilde{\mathcal{J}}}
\newcommand{\cT}{\mathcal{T}}
\newcommand{\vbar}{\overline{v}}
\newcommand{\psitil}{\widetilde{\psi}}
\newcommand{\cKtil}{\widetilde{\mathcal{K}}}
\newcommand{\qtil}{\widetilde{q}}
\newcommand{\Omtil}{\widetilde{\Om}}
\newcommand{\SO}{\operatorname{SO}}
\newcommand{\SU}{\operatorname{SU}}
\newcommand{\m}{\mathord{\text{\rm m}}}
\newcommand{\Vtil}{\widetilde{V}}
\newcommand{\Wtil}{\widetilde{W}}
\newcommand{\Btil}{\widetilde{B}}
\newcommand{\Atil}{\widetilde{A}}
\newcommand{\Ptil}{\widetilde{P}}

\newcommand{\op}{^\text{\rm op}}
\newcommand{\cb}{_\text{\rm cb}}
\newcommand{\nor}{_\text{\rm n}}
\newcommand{\sing}{_\text{\rm s}}
\newcommand{\cCs}{\mathcal{C}_{\text{\rm\tiny s}}}
\newcommand{\cCgms}{\mathcal{C}_{\text{\rm\tiny gms}}}
\newcommand{\Sfactor}{\mathcal{S}_{\text{\rm\tiny factor}}}
\newcommand{\Seqrel}{\mathcal{S}_{\text{\rm\tiny eqrel}}}

\newcommand{\bim}[3]{\mathord{\raisebox{-0.4ex}[0ex][0ex]{\scriptsize $#1$}{#2}\hspace{-0.05ex}\raisebox{-0.4ex}[0ex][0ex]{\scriptsize $#3$}}}

\begin{document}

\begin{center}
{\boldmath\LARGE\bf Unique Cartan decomposition for II$_1$ factors\vspace{0.5ex}\\
arising from arbitrary actions of free groups}

\bigskip

{\sc by Sorin Popa\footnote{Mathematics Department, UCLA, CA 90095-1555 (United States), popa@math.ucla.edu\\
Supported in part by NSF Grant DMS-1101718} and Stefaan Vaes\footnote{KU~Leuven, Department of Mathematics, Leuven (Belgium), stefaan.vaes@wis.kuleuven.be \\
    Supported by ERC Starting Grant VNALG-200749, Research
    Programme G.0639.11 of the Research Foundation --
    Flanders (FWO) and KU~Leuven BOF research grant OT/08/032.}}
    
%
\end{center}

\begin{abstract}\noindent
We prove that for any free ergodic probability measure preserving
action $\F_n \actson (X,\mu)$ of a free group on $n$ generators $\F_n$, $2\leq n \leq \infty$, the associated group measure space II$_1$ factor
$\rL^\infty(X) \rtimes \F_n$ has $\rL^\infty(X)$ as its unique Cartan subalgebra, up to unitary conjugacy. We deduce that group measure space II$_1$ factors arising from
actions of free groups with different number of generators are never isomorphic.
We actually prove unique Cartan decomposition results  for
II$_1$ factors arising from arbitrary actions of a much larger family of groups, including all free products of amenable groups and their direct products.
\end{abstract}

\section{Introduction and main results}

A crossed product type construction due to Murray and von Neumann \cite{MvN36} associates to any free ergodic probability measure preserving (pmp) action $\Gamma \actson (X,\mu)$ of a countable group $\Gamma$,
a II$_1$ factor denoted $ \rL^\infty(X) \rtimes \Gamma$ and
called the {\it group measure space} algebra  of $\Gamma \curvearrowright X$. A more general, groupoid-version of this
construction associates a II$_1$ factor $\rL \cR$ to any countable ergodic pmp equivalence relations $\cR$ on $(X,\mu)$
(\cite{FM75}).  The two algebras coincide when $\cR$ is given by the orbits of the free ergodic action $\Gamma \curvearrowright X$,
showing that group actions having the same orbits give the same II$_1$ factor. Moreover,
both $\rL^\infty(X)\rtimes \Gamma$ and $\rL \cR$ contain $\rL^\infty(X)$ as a
{\it Cartan subalgebra}, i.e. a maximal abelian $^*$-subalgebra whose normalizer generates
the II$_1$ factor, while by \cite{FM75}
two countable ergodic pmp
equivalence relations $\cR_1, \cR_2$ are isomorphic iff there exists an isomorphism of the associated II$_1$  factors taking
the corresponding Cartan subalgebras one onto the other.

The classification of the
algebras $\rL^\infty(X)\rtimes \Gamma$, $\rL \cR$ in terms of their building data, $\Gamma \curvearrowright X$, $\cR$,
is a notoriously hard problem which, over the years, has led to a fruitful interplay between
operator algebras and functional analysis, group theory (geometric, measured, etc), representation theory, Lie group theory,
ergodic theory, etc.

The dichotomy
amenable-nonamenable is particularly strong in this framework: by a celebrated theorem of Connes \cite{Co75}, all II$_1$
factors $\rL^\infty(X)\rtimes \Gamma$, $\rL \cR$ with $\Gamma, \cR$ amenable
are isomorphic (in fact, by \cite{CFW81}, there is just one amenable equivalence relation $\cR$!); but
nonamenable group actions ``tend to be'' recognizable from the
isomorphism class of their associated algebra. In fact, the prevailing point of view in recent years
has been to approach the nonamenable case of this classification problem
as a rigidity paradigm,
seeking to prove that an isomorphism of group measure space II$_1$ factors forces the corresponding building data
(e.g., $\Gamma$,  $\cR$) to share some common properties, or even coincide.

There has been intense activity in this direction over the last decade, with the emergence of new tools
of investigation and the discovery of many surprising rigidity results. But one of the most intriguing
questions in this area, asking
whether an isomorphism $\rL^\infty(X)\rtimes \F_n \simeq \rL^\infty(Y)\rtimes \F_m$,
arising from two \emph{arbitrary}  free ergodic pmp actions $\F_n \curvearrowright X$,
$\F_m \curvearrowright Y$ of the free groups with $n$ and respectively $m$ generators,
forces $n=m$, has remained open. There was supporting evidence for this conjecture from results
in \cite{Po01} and \cite{OP07}, showing that this is indeed the case if the two actions
are either HT or compact. But this was not known for other actions, such as the Bernoulli actions
$\F_n \curvearrowright [0,1]^{\F_n}$.

We solve this problem here, in the affirmative.
More precisely, we prove that any group measure space II$_1$ factor
$M=\rL^\infty(X)\rtimes \F_n$, arising from an arbitrary free ergodic pmp
action $\F_n \curvearrowright X$,  ``remembers'' the associated equivalence relation $\cR_{\F_n}$. We do this
by showing that $M$ has unique Cartan subalgebra,
up to conjugacy by a unitary in $M$. This in turn reduces the problem to whether equivalence relations arising
from free ergodic pmp actions of free groups with different number of generators are always non-isomorphic, which does hold true by
a well known result  in \cite{Ga99},  \cite{Ga01}. Note that our result gives an answer to the wreath product version
of the famous free group factor problem: if $\rL(\Z \wr \F_n) \simeq \rL(\Z \wr \F_m)$
then $n=m$.
In fact, by combining our theorem with the work  in \cite{Bo09a,Bo09b}, we obtain a
complete classification of the amplifications of  II$_1$ factors arising
from Bernoulli actions of free groups, $(\rL^\infty([0,1]^{\F_n}) \rtimes \F_n)^t$,
for which we show that the number $(n-1)/t$ is a complete invariant.

Note that our result provides the first groups $\Gamma$  with the property that \emph{any}
group measure space II$_1$ factor $\rL^\infty(X)\rtimes \Gamma$, arising from
an \emph{arbitrary} free ergodic pmp $\Gamma$-action, has unique Cartan subalgebra,
up to unitary conjugacy, a class of groups that we call $\cC$-rigid. Indeed, the results in
\cite{OP07}, which were the first to provide a class of factors with unique Cartan  decomposition up to unitary conjugacy,
only covered group measure space II$_1$ factors arising from \emph{profinite} actions of $\F_n$.

We in fact prove $\cC$-rigidity for much larger classes of groups $\Gamma$ than the free groups. For instance,
we show that any weakly amenable group $\Gamma$ with nonzero first $\ell^2$-Betti number, $\beta^{(2)}_1(\Gamma)>0$,
is $\cC$-rigid. We conjecture that in fact any $\Gamma$ with at least one nonzero $\ell^2$-Betti number,
$\beta^{(2)}_n(\Gamma) > 0$, is $\cC$-rigid. Note that if this conjecture would be true then,
since the $\ell^2$-Betti numbers of groups are invariant under orbit equivalence (cf. \cite{Ga01}),
it would follow that $\beta_n^{(2)}(\Gamma)$ are isomorphism
invariants for arbitrary group measure space II$_1$ factors $\rL^\infty(X)\rtimes \Gamma$.

There is further supporting evidence for the above conjecture. For instance, in \cite{PV09} we proved that a fairly large class of free product groups $\Gamma = \Gamma_1 * \Gamma_2$, including all those where $\Gamma_1$ is an infinite property (T) group and $\Gamma_2$ is nontrivial, has the property that
$\rL^\infty(X) \rtimes \Gamma$ has a unique \emph{group measure space Cartan\footnote{A maximal abelian subalgebra $A$ of a II$_1$ factor $M$ is called a group measure space Cartan subalgebra if $M$ can be decomposed as a crossed product $M = A \rtimes \Lambda$. Not all Cartan subalgebras in II$_1$ factors are of this form.}} subalgebra for any $\Gamma$-action. We call groups $\Gamma$ with this property $\cCgms$-rigid. More generally, it was established in \cite{CP10} that all
groups that have at the same time a nonvanishing first $\ell^2$-Betti number and a nonamenable subgroup with the relative property (T), are $\cCgms$-rigid (see also the expository paper \cite{Va10b}).
Very recently it was shown in \cite{Io11} that $\rL^\infty(X) \rtimes \Gamma$ has a unique group measure space Cartan subalgebra if $\beta_1^{(2)}(\Gamma) > 0$ and $\Gamma \actson (X,\mu)$ is a \emph{rigid} (in the sense of \cite{Po01}) free ergodic pmp action.

One should point out that the unique Cartan decomposition results for
profinite actions of \cite{OP07,OP08} have been generalized in \cite{CS11,CSU11} to
show that group measure space II$_1$ factors $\rL^\infty(X) \rtimes \Gamma$ arising from \emph{profinite} free ergodic pmp
actions of any hyperbolic group or direct product of hyperbolic groups, have a unique Cartan subalgebra up to unitary conjugacy. In the follow-up paper [PV12], the main innovations of our article (Sections 4 and 5) are combined with the methods of [CS11,CSU11] to prove that any product of hyperbolic groups is $\cC$-rigid. So, the uniqueness of the Cartan subalgebra of $\rL^\infty(X) \rtimes \Gamma$ holds without assuming the profiniteness of the action $\Gamma \actson (X,\mu)$.

While a characterization of all $\cC$-rigid groups seems even difficult to guess, it would be very interesting to find other sufficient conditions
for this property to hold. As for
necessary conditions, let us point out that in \cite{CJ81} it was shown that any direct product $\Gamma = H \times G$
between a non-amenable group $G$ and a certain type of locally finite, infinite, non-commutative group $H$, is not $\cC$-rigid.
Another class of groups that are not $\cC$-rigid was found in \cite{OP08} and it consists of
certain semidirect products $\Gamma=H \rtimes G$, with $H$ abelian, notably $\Gamma=\Z^2 \rtimes \SL(2, \Z)$.
More generally, it was shown in \cite[Section 5.5]{PV09} that a semidirect product $\Gamma = H \rtimes G$ with $H$ infinite abelian, is never $\cC$-rigid.
We believe that in fact groups $\Gamma$ with an infinite amenable normal subgroup are never $\cC$-rigid. Since by \cite{CG85} (see also \cite[Theorem 7.2.(2)]{Lu02}) all $\ell^2$-Betti numbers of such groups $\Gamma$ vanish, this is compatible with the conjecture that all groups with at least one non-zero $\ell^2$-Betti number are $\cC$-rigid, as formulated above. On the other hand, it would be interesting to find examples of non $\cC$-rigid groups that admit no infinite amenable quasi-normal subgroup.

To state our results in more details, we first need some terminology.

\begin{definition}\label{def.weaklyamenable}
A \emph{Herz-Schur multiplier} on a countable group $\Gamma$ is a function $f : \Gamma \recht \C$ such that the corresponding map $u_g \mapsto f(g) u_g$ extends to a normal completely bounded map $m_f : \rL(\Gamma) \recht \rL(\Gamma)$. In that case we write $\|f\|\cb := \|m_f\|\cb$. A countable group $\Gamma$ is called \emph{weakly amenable} (see \cite{CH88}) if it admits a sequence of finitely supported Herz-Schur multipliers $f_n : \Gamma \recht \C$ that tend to $1$ pointwise and that satisfy $\limsup_n \|f_n\|\cb < \infty$. If $(f_n)$ can be chosen in such a way that $\limsup_n \|f_n\|\cb = 1$, we say that $\Gamma$ has the \emph{complete metric approximation property (CMAP),} see \cite{Ha78}.

Let $\Gamma$ be a countable group and $\eta : \Gamma \recht \cO(K_\R)$ an orthogonal representation. A \emph{$1$-cocycle for $\Gamma$ into the orthogonal representation $\eta$} is a map $c : \Gamma \recht K_\R$ satisfying $c(gh) = c(g) + \eta_g c(h)$ for all $g,h \in \Gamma$. We say that $c$ is \emph{proper} if $\|c(g)\| \recht \infty$ whenever $g \recht \infty$.

Following \cite[Definition 1.1]{Be89}, a unitary representation $\eta : \Gamma \recht \cU(K)$ is called \emph{amenable} if $\B(K)$ admits an $(\Ad \eta_g)_{g \in \Gamma}$-invariant state. A unitary representation $\eta : \Gamma \recht \cU(K)$ is called \emph{mixing} if for all $\xi,\xi' \in K$, we have that $\langle \eta_g\xi,\xi'\rangle \recht 0$ whenever $g \recht \infty$, i.e.\ when the matrix coefficients of $\eta$ tend to zero at infinity.
\end{definition}

\begin{theorem}\label{thm.Crigidgroups}
For all of the following groups $\Gamma$, all group measure space II$_1$ factors $M := \rL^\infty(X) \rtimes \Gamma$ with respect to arbitrary free ergodic pmp actions $\Gamma \actson (X,\mu)$ have $\rL^\infty(X)$ as their unique Cartan subalgebra up to unitary conjugacy.
\begin{enumerate}
\item All weakly amenable groups $\Gamma$ with $\beta_1^{(2)}(\Gamma) > 0$. More generally, all weakly amenable groups $\Gamma$ that admit an unbounded $1$-cocycle into a mixing nonamenable representation.
\item All weakly amenable groups $\Gamma$ that admit a proper $1$-cocycle into a nonamenable representation.
\end{enumerate}
Actually a more general statement holds: whenever $A \subset M$ is a maximal abelian subalgebra whose normalizer is a finite index subfactor of $M$, we must have that $A$ is unitarily conjugate to $\rL^\infty(X)$.
\end{theorem}

\begin{remark}
Theorem \ref{thm.Crigidgroups} covers a rather large family of groups. In \cite[Definition 1]{OP08} a countable group $\Gamma$ is said to have the property $(HH)^+$ if $\Gamma$ has the CMAP and if $\Gamma$ admits a proper $1$-cocycle into a nonamenable representation. Obviously all groups with the property $(HH)^+$ belong to the second family of Theorem \ref{thm.Crigidgroups}. By \cite[Theorem 2.3]{OP08} the class $(HH)^+$ contains all lattices in $\SL(2,\R)$, $\SL(2,\C)$, $\SO(n,1)$ with $n \geq 2$, and $\SU(n,1)$. Furthermore the class $(HH)^+$ contains the free groups $\F_n$, $2 \leq n \leq \infty$, and contains all free products $\Lambda_1 * \Lambda_2$ of amenable groups $\Lambda_1,\Lambda_2$ with $|\Lambda_1| \geq 2$ and $|\Lambda_2| \geq 3$. Also the class $(HH)^+$ is stable under free products and direct products.
\end{remark}

\begin{definition}
We say that a countable group $\Gamma$ is $\cC$-rigid (Cartan-rigid) if for every free ergodic pmp action $\Gamma \actson (X,\mu)$, the II$_1$ factor $\rL^\infty(X) \rtimes \Gamma$ has $\rL^\infty(X)$ as its unique Cartan subalgebra up to unitary conjugacy.

In view of \cite[Proposition 4.12]{OP07} we say that a countable group $\Gamma$ is $\cCs$-rigid\footnote{The notation $\cCs$-rigid can be read as ``strongly Cartan-rigid'', but also as ``stably Cartan-rigid'' because of the stability results in \cite[Proposition 4.12]{OP07}.} if for every free ergodic pmp action $\Gamma \actson (X,\mu)$, the II$_1$ factor $M = \rL^\infty(X) \rtimes \Gamma$ has the following property~: every maximal abelian subalgebra $A \subset M$ whose normalizer $\cN_M(A)\dpr$ is a finite index subfactor of $M$, is unitarily conjugate to $\rL^\infty(X)$.
\end{definition}

As already mentioned above, Theorem \ref{thm.Crigidgroups} has some immediate consequences in the classification of free group
measure space II$_1$
factors. Recall that if $M$ is a II$_1$ factor and $s > 0$, then $M^s$ denotes the Murray-von Neumann \emph{amplification} of $M$ by $s$.

\begin{theorem}\label{thm.class}
\begin{enumerate}
\item If $n \neq m$ and $\F_n \actson (X,\mu)$, $\F_m \actson (Y,\eta)$ are arbitrary free ergodic pmp actions, then
$$\rL^\infty(X) \rtimes \F_n \not\cong \rL^\infty(Y) \rtimes \F_m \; .$$

\item If $(X_0,\mu_0)$ and $(Y_0,\eta_0)$ are nontrivial standard probability spaces and if $2 \leq n,m \leq \infty$, $s,t > 0$, we have
$$\Bigl(\rL^\infty\bigl( X_0^{\F_n}\bigr) \rtimes \F_n \Bigr)^s \cong \Bigl(\rL^\infty\bigl( Y_0^{\F_m}\bigr) \rtimes \F_m \Bigr)^t \quad\text{if and only if}\quad \frac{n-1}{s} = \frac{m-1}{t} \; .$$
In particular for the wreath product groups $\Z \wr \F_n = \Z^{(\F_n)} \rtimes \F_n$ we get that $\rL(\Z \wr \F_n)^s \cong \rL(\Z \wr \F_m)^t$ if and only if $(n-1)/s = (m-1)/t$.

\item If $\cR_1$ is a treeable ergodic pmp equivalence relation and if $\rL \cR_1 \cong \rL \cR_2$ for some other pmp equivalence relation $\cR_2$, then $\cR_1 \cong \cR_2$.
\end{enumerate}
\end{theorem}

Theorem \ref{thm.Crigidgroups} also has a number of consequences for the \emph{fundamental group} of group measure space II$_1$ factors. Recall that the fundamental group $\cF(M)$ of a II$_1$ factor $M$ is the group of positive real numbers $s > 0$ such that $M^s \cong M$. In \cite{PV08b} we introduced the invariants $\Sfactor(\Gamma)$ and $\Seqrel(\Gamma)$ of a countable group $\Gamma$, as being the set of subgroups of $\R_+$ that can arise as the fundamental group of a group measure space II$_1$ factor $\rL^\infty(X) \rtimes \Gamma$, resp.\ an orbit equivalence relation $\cR(\Gamma \actson X)$, for some free ergodic pmp action of $\Gamma$. In \cite{PV08a} we proved that $\Sfactor(\F_\infty)$ and $\Seqrel(\F_\infty)$ are huge. They for instance contain subgroups of $\R_+$ that can have any Hausdorff dimension between $0$ and $1$. On the other hand from \cite[Th\'{e}or\`{e}me 6.3]{Ga01} we know that $\Seqrel(\F_n) = \{\{1\}\}$ for all $2 \leq n < \infty$. Whenever $\Gamma$ is a $\cC$-rigid group we have $\Sfactor(\Gamma) = \Seqrel(\Gamma)$. So it follows from Theorem \ref{thm.Crigidgroups} that also $\Sfactor(\F_n) = \{\{1\}\}$ for all $2 \leq n < \infty$, confirming our conjecture in \cite{PV08b}.

Throughout this article we call $(M,\tau)$ a tracial von Neumann algebra if $M$ is a von Neumann algebra equipped with a faithful normal tracial state $\tau$.

Following \cite{Oz03} a tracial von Neumann algebra $(M,\tau)$ is called \emph{solid} if the relative commutant $A' \cap M$ of any diffuse von Neumann subalgebra $A \subset M$ is amenable. It is shown in \cite{Oz03} that the group von Neumann algebras $\rL \Gamma$ of
any   hyperbolic groups is solid. Then in \cite{OP07},  $(M,\tau)$ is called \emph{strongly solid} if even
the normalizer of any diffuse amenable subalgebra of $M$ is still amenable, and
it is shown that the free group factors $\rL \F_n$ are strongly solid. It has
been recently
proved in \cite{CS11} that in fact all group von Neumann algebras $\rL \Gamma$ of arbitrary hyperbolic groups are strongly solid.

Crossed products $B \rtimes \Gamma$ are of course typically not strongly solid, but we establish the following relative strong solidity property: for certain groups $\Gamma$ we prove the dichotomy that an amenable subalgebra $A$ of an arbitrary crossed product $B \rtimes \Gamma$ with $B$ amenable either embeds into $B$ (in the sense of intertwining-by-bimodules, see Definition \ref{def.intertwine}), or has an amenable normalizer. More generally one can replace ``amenability'' by ``amenability relative to $B$'' in the sense of Definition \ref{def.rel-amen}, resulting in the following statement.

\begin{theorem}\label{thm.main}
Let $\Gamma$ be a weakly amenable group that admits a proper $1$-cocycle into an orthogonal representation that is weakly contained in the regular representation. Let $\Gamma \overset{\si}{\actson} (B,\tau)$ be any trace preserving action on a tracial von Neumann algebra $(B,\tau)$. Denote $M = B \rtimes \Gamma$ and let $A \subset M$ be a von Neumann subalgebra such that $A$ is amenable relative to $B$.

Either $A \prec_M B$ or the normalizer $P := \cN_M(A)\dpr$ is amenable relative to $B$.
\end{theorem}

Note that Theorem \ref{thm.main} immediately implies that for all II$_1$ factors $B$ and all $2 \leq n \leq +\infty$, the tensor product $B \ovt \rL \F_n$ has no Cartan subalgebra, thus improving \cite[Corollary 2]{OP07} which required $B$ to have the complete metric approximation property.

If $\Gamma = \Gamma_1 \times \cdots \times \Gamma_n$ is a direct product of $n \geq 2$ nonamenable groups, Theorem \ref{thm.main} does not hold since, for instance, the relative commutant of a subalgebra of $\rL(\Gamma_1)$ contains $\rL(\Gamma_2)$. Nevertheless we obtain the following precise description of what exactly can happen. The notion of strong intertwining $A \prec^f_M Q$ is explained in Definition \ref{def.intertwine}.

\begin{theorem}\label{thm.main-products}
Let $\Gamma = \Gamma_1 \times \cdots \times \Gamma_n$ be a direct product of weakly amenable groups such that every $\Gamma_i$ admits a proper $1$-cocycle into an orthogonal representation that is weakly contained in the regular representation of $\Gamma_i$. Let $\Gamma \overset{\si}{\actson} (B,\tau)$ be any trace preserving action on a tracial von Neumann algebra $(B,\tau)$. Denote $M = B \rtimes \Gamma$ and let $A \subset M$ be a von Neumann subalgebra that is amenable relative to $B$. Denote by $P := \cN_M(A)\dpr$ the normalizer of $A$ inside $M$.

Then there exist projections $p_0,\ldots,p_n \in \cZ(P)$, some of which might be zero, such that $p_0 \vee \cdots \vee p_n = 1$ and
\begin{itemize}
\item $P p_0$ is amenable relative to $B$,
\item for every $i=1,\ldots,n$ we have $A p_i \prec^f_M B \rtimes \Gammah_i$ where $\Gammah_i$ is the product of all $\Gamma_j$, $j \neq i$.
\end{itemize}
\end{theorem}

Note that each $\Gamma$ covered by Theorem \ref{thm.main-products} with the factors $\Gamma_i$ being nonamenable, also belongs to the second family of Theorem \ref{thm.Crigidgroups} and hence is $\cC$-rigid and $\cCs$-rigid.

We obtain the following similar result for crossed products $B \rtimes \Gamma$ by arbitrary actions of weakly amenable free products $\Gamma = \Lambda_1 * \Lambda_2$. Note that these groups belong to the first family in Theorem \ref{thm.Crigidgroups} and hence also are $\cC$-rigid and $\cCs$-rigid.

\begin{theorem}\label{thm.freeproducts}
Let $\Gamma = \Lambda_1 * \Lambda_2$ be any weakly amenable free product group (e.g.\ the free product of two groups with the CMAP). Let $\Gamma \overset{\si}{\actson} (B,\tau)$ be any trace preserving action on a tracial von Neumann algebra $(B,\tau)$. Denote $M = B \rtimes \Gamma$ and let $A \subset M$ be a von Neumann subalgebra that is amenable relative to $B$. Denote by $P := \cN_M(A)\dpr$ the normalizer of $A$ inside $M$.

Then there exist projections $q,p_0,p_1,p_2 \in \cZ(P)$, some of which might be zero, such that $q \vee p_0 \vee p_1 \vee p_2 = 1$ and
\begin{itemize}
\item $A q \prec^f_M B$,
\item $P p_0$ is amenable relative to $B$,
\item $P p_i \prec^f_M B \rtimes \Lambda_i$ for all $i = 1,2$.
\end{itemize}
\end{theorem}

All the results above will follow from a key technical theorem that we state as Theorem \ref{thm.source} in Section \ref{sec.source}.

As a consequence of the above uniqueness theorems for Cartan subalgebras, we obtain several W$^*$-superrigidity results. Recall that a free ergodic pmp action $\Gamma \actson (X,\mu)$ is called W$^*$-superrigid if the group measure space II$_1$ factor $\rL^\infty(X) \rtimes \Gamma$ ``remembers'' the group action $\Gamma \actson (X,\mu)$~: any other group measure space construction yielding an isomorphic II$_1$ factor must come from an isomorphic group and a conjugate action (see Section \ref{sec.superrigid} for precise definitions). In \cite{Pe09} the existence of virtually W$^*$-superrigid group actions was proven.
In \cite{PV09} we obtained the first concrete W$^*$-superrigidity theorem,
for Bernoulli actions of a large class of amalgamated free product groups. In \cite{Io10} it was shown   that  Bernoulli actions of icc property (T) groups are W$^*$-superrigid. In the present paper, a combination of our unique Cartan decomposition
theorem \ref{thm.Crigidgroups} and the OE superrigidity theorems in \cite{Po05,Po06b} will allow us to deduce  the following result (see also Theorem \ref{thm.superrigid-quotient-bernoulli} and Remark \ref{rem.superrigid} thereafter).

\begin{theorem}
Let $\Lambda, \Gamma_1, \Gamma_2$ be weakly amenable icc groups that admit a proper $1$-cocycle into a nonamenable representation.
\begin{itemize}
\item Put $\Gamma = \Gamma_1 \times \Gamma_2$. All free actions of $\Gamma$ that arise as a quotient of the Bernoulli action $\Gamma \actson [0,1]^{\Gamma}$ are W$^*$-superrigid.
\item Consider $\Lambda \times \Lambda \actson \Lambda$ by left-right multiplication. All free actions of $\Lambda \times \Lambda$ that arise as a quotient of the generalized Bernoulli action $\Lambda \times \Lambda \actson [0,1]^\Lambda$ are W$^*$-superrigid.
\end{itemize}
\end{theorem}

We finally deduce a strong rigidity theorem for crossed products by outer actions. Recall that an action $(\al_g)_{g \in \Gamma}$ by automorphisms of a factor $R$ is called \emph{outer} if no $\al_g$, $g \in e$, is an inner automorphism $\Ad u$, $u \in \cU(R)$. Two outer actions $\al : \Gamma \actson P$ and $\beta : \Lambda \actson Q$ are called \emph{cocycle conjugate} if there exists an isomorphism $\pi : P \recht Q$, an isomorphism $\delta : \Gamma \recht \Lambda$ and a map $w : \Gamma \recht \cU(P)$ such that
$$\pi(w_g \al_g(x) w_g^*) = \beta_{\delta(g)}(\pi(x)) \quad\text{and}\quad w_{gh} = w_g \, \al_g(w_h) \quad\text{for all}\;\; g,h \in \Gamma, x \in P \; .$$

\begin{theorem}\label{thm.rigid-outer}
If $\Gamma,\Lambda$ are icc groups in one of the families of Theorem \ref{thm.Crigidgroups} and if $\Gamma \actson R$ and $\Lambda \actson R$ are outer actions on the hyperfinite II$_1$ factor $R$ such that $R \rtimes \Gamma \cong R \rtimes \Lambda$, then $\Gamma \cong \Lambda$ and the actions $\Gamma \actson R$, $\Lambda \actson R$ are cocycle conjugate.
\end{theorem}

\subsection*{Comments on the proofs}

In order to explain the main ideas of the paper, we outline the proof of the following special case of Theorem \ref{thm.main}. Assume that $\Gamma$ is a group with the CMAP and with a proper $1$-cocycle into the infinite multiple $\ell^2_\R(\Gamma)^{\oplus \infty}$ of the regular representation. Note that the free groups $\Gamma = \F_n$ satisfy these properties. Assume that $\Gamma \actson (B,\tau)$ is an arbitrary trace preserving action on the tracial von Neumann algebra $(B,\tau)$ and put $M = B \rtimes \Gamma$. Let $A \subset M$ be a von Neumann subalgebra that we assume, in this rough sketch, to be plainly amenable. Put $P := \cN_M(A)\dpr$. We want to prove that either $A \prec_M B$ or that $P$ is amenable relative to $B$.

\vspace{1ex}

{\bf Step 1~: reduction to the trivial action.} As we will see in Lemma \ref{lem.reduction}, we may assume that $\Gamma \actson (B,\tau)$ is the trivial action. To make this reduction from arbitrary actions to the trivial action, we use the comultiplication trick. So denote by $\Delta : M \recht M \ovt \rL(\Gamma)$ the normal $*$-homomorphism defined by $\Delta(b u_g) = bu_g \ot u_g$ for all $b \in B$ and $g \in \Gamma$. We view $M \ovt \rL(\Gamma)$ as the crossed product of $\Gamma$ acting trivially on $M$. We consider $\Delta(A) \subset M \ovt \rL(\Gamma)$. As we will see, it is rather straightforward to prove
\begin{itemize}
\item that $A \prec_M B$ if and only if $\Delta(A) \prec_{M \ovt \rL(\Gamma)} M \ot 1$,
\item and that $P$ is amenable relative to $B$ if and only if $\Delta(P)$ is amenable relative to $M \ot 1$.
\end{itemize}
So the result for arbitrary actions is an immediate consequence of the result for the trivial action.

From now on we will assume that $\Gamma \actson B$ is the trivial action. Hence $M$ equals the tensor product $M = B \ovt \rL(\Gamma)$.

\vspace{1ex}

{\bf\boldmath Step 2~: weak compactness relative to $B$.} The most important novelty of this paper is the proof that the action $\cN_M(A) \actson A$ satisfies a relative version w.r.t.\ $B$ of the weak compactness property of \cite[Definition 3.1]{OP07}. For this we only use the CMAP of $\Gamma$. So take a sequence of finitely supported Herz-Schur multipliers $f_n : \Gamma \recht \C$ that tend to $1$ pointwise and that satisfy $\limsup_n \|f_n\|\cb = 1$. Denote by $\vphi_n : M \recht M$ the associated completely bounded maps given by $\vphi_n(b \ot u_g) = f_n(g) b \ot u_g$ for all $b \in B$ and $g \in \Gamma$. The formula
$$\mu_n : M \otmin P\op \recht \C : \mu_n(x \ot y\op) := \tau(\vphi_n(x) E_A(y)) \quad\text{for all}\;\; x \in M, y \in P,$$
provides a sequence of continuous functionals on the C$^*$-algebra $M \otmin P\op$ satisfying
\begin{itemize}
\item $\limsup_n \|\mu_n\| = 1$,
\item $\lim_n \|\mu_n \circ \Ad(u \ot \ubar) - \mu_n \| = 0$ for all $u \in \cN_M(A)$, where $\ubar = (u\op)^*$.
\end{itemize}
Since moreover $\mu_n(1) \recht 1$, it follows that $\|\mu_n - \om_n\| \recht 0$, where $\om_n$ denotes the state on $M \otmin P\op$ defined as $\om_n = \|\mu_n\|^{-1} |\mu_n|$.

A crucial point in the continuation of the argument will be to construct a von Neumann algebra completion $\cN$ of $M \otmin P\op$ with the following two properties~:
\begin{itemize}
\item the states $\om_n$ are normal on $\cN$,
\item the von Neumann algebra $\cN$ splits as a tensor product $\cN = N \ovt \rL(\Gamma)$, with the natural copy of $\rL(\Gamma)$ inside $M \subset \cN$ corresponding to the copy of $\rL(\Gamma)$ inside $N \ovt \rL(\Gamma)$.
\end{itemize}
Choosing a standard representation of $N$ on the Hilbert space $H$, it follows that $\cN$ is standardly represented on $H \ot \ell^2(\Gamma)$. The states $\om_n$ are then implemented by canonical positive vectors $\xi_n \in H \ot \ell^2(\Gamma)$. These vectors $(\xi_n)$ inherit the almost invariance properties of $(\om_n)$.

\vspace{1ex}

{\bf\boldmath Step 3~: applying a malleable deformation $(\al_t)_{t \in \R}$ to the vectors $(\xi_n)$.} The group $\Gamma$ admits a proper $1$-cocycle $c : \Gamma \recht \ell^2_\R(\Gamma)^{\oplus \infty}$ into an infinite multiple of the regular representation. Associated with $c$ is a $1$-parameter family $(\psi_t)_{t > 0}$ of unital completely positive maps on $\cN$ given by
$$\psi_t(x \ot u_g) = \exp(-t \|c(g)\|^2) \, (x \ot u_g) \quad\text{for all}\;\; x \in N, g \in \Gamma \; .$$
By \cite{Si10} the $1$-parameter family $(\psi_t)_{t > 0}$ dilates as a \emph{malleable deformation} $(\al_t)_{t \in \R}$ by automorphisms of a larger von Neumann algebra $\cNtil \supset \cN$. This construction comes with a conditional expectation $E : \cNtil \recht \cN$ such that
$$\psi_{t^2/2}(x) = E(\al_t(x)) \quad\text{for all}\;\; x \in \cN, t \in \R \; .$$
The dichotomy in the conclusion of the theorem then arises as follows.
\begin{itemize}
\item Either the deformation $(\al_t)$ significantly moves the vectors $(\xi_n)$. Since these vectors $(\xi_n)$ have a certain almost invariance property under all $u \in \cN_M(A)$, this will lead to the amenability of $P$ relative to $B$.
\item Or the deformation $(\al_t)$ does not significantly move the vectors $(\xi_n)$. By the properness of the $1$-cocycle $c$, this will lead to $A \prec_M B$.
\end{itemize}

\section{Preliminaries}

To make this article as self-contained as possible we have chosen to include a rather extensive section with preliminaries.

\subsection{Terminology}

As we said above we call $(M,\tau)$ a \emph{tracial von Neumann algebra} if $M$ is a von Neumann algebra equipped with a faithful normal tracial state $\tau$.

Whenever $M$ is a von Neumann algebra and $A \subset M$ is a von Neumann subalgebra, we denote by $\cN_M(A)$ the group of unitaries $u \in \cU(M)$ that satisfy $u A u^* = A$. We call the von Neumann algebra $\cN_M(A)\dpr$ the \emph{normalizer} of $A$ inside $M$. We say that $A \subset M$ is \emph{regular} if its normalizer equals $M$. A \emph{Cartan subalgebra} of a II$_1$ factor $M$ is a maximal abelian, regular von Neumann subalgebra.

If $(M,\tau)$ and $(Q,\tau)$ are tracial von Neumann algebras, we call \emph{right $Q$-module} any Hilbert space equipped with a normal $*$-anti-representation of $Q$. We call \emph{$M$-$Q$-bimodule} any Hilbert space equipped with a normal $*$-representation of $M$ and a normal $*$-anti-representation of $Q$ with commuting ranges. We usually simple write $x \cdot \xi \cdot y$ to denote the left and right module actions of $x \in M$, $y \in Q$ on the vector $\xi$.

If $\cN$ is a von Neumann algebra and $M \subset \cN$ is a von Neumann subalgebra, a functional $\Om$ on $\cN$ is called \emph{$M$-central} if $\Om(S x) = \Om(x S)$ for all $S \in \cN$ and all $x \in M$.

A tracial von Neumann algebra $(M,\tau)$ is called \emph{amenable} if there exists an $M$-central state on $\B(\rL^2(M))$ whose restriction to $M$ equals $\tau$. We refer to Section \ref{sec.rel-amen} for more background on amenability.

\subsection{Intertwining by bimodules}\label{sec.intertwine}

We recall from \cite[Theorem 2.1 and Corollary 2.3]{Po03} the theory of \emph{intertwining-by-bimodules,} summarized in the following definition.

\begin{definition}\label{def.intertwine}
Let $(M,\tau)$ be a tracial von Neumann algebra and $P,Q \subset M$ possibly non-unital von Neumann subalgebras. We write $P \prec_M Q$, and say that $P$ embeds into $Q$ inside $M$, when one of the following equivalent conditions is satisfied.
\begin{itemize}
\item There exist projections $p \in P$, $q \in Q$, a normal $*$-homomorphism $\vphi : pPp \recht qQq$ and a nonzero partial isometry $v \in pMq$ such that $x v = v \vphi(x)$ for all $x \in pPp$.
\item It is impossible to find a net of unitaries $u_n \in \cU(P)$ satisfying $\|E_Q(x u_n y^*)\|_2 \recht 0$ for all $x,y \in 1_Q M 1_P$.
\end{itemize}
We write $P \prec^f_M Q$ if $P p \prec_M Q$ for every projection $p \in P' \cap 1_P M 1_P$.
\end{definition}

\subsection{Basic construction, Jones index, Connes tensor product}\label{sec.connesproduct}

Let $(Q,\tau)$ be a tracial von Neumann algebra and $\cK_Q$ a right Hilbert $Q$-module. Then the von Neumann algebra $\cN := \B(\cK) \cap (Q\op)'$ carries a canonical semifinite faithful normal trace $\Tr$ that can be characterized as follows. First recall that a vector $\xi \in \cK$ is called right bounded if there exists a $\kappa \geq 0$ such that $\|\xi x\| \leq \kappa \|x\|_2$ for all $x \in Q$. When $\xi \in \cK$ is right bounded we denote by $L_\xi \in \B(\rL^2(Q),\cK)$ the operator defined as $L_\xi x = \xi x$ for all $x \in Q$. For all right bounded vectors $\xi,\eta \in \cK$ we have that $L_\xi L_\eta^* \in \cN$, while $L_\eta^* L_\xi \in Q$. The right bounded vectors form a dense subspace of $\cK$ and the corresponding elements $L_\xi L_\eta^* \in \cN$ span a dense $*$-subalgebra of $\cN$. The trace $\Tr$ on $\cN$ can be characterized by the formula
$$\Tr(L_\xi L_\eta^*) = \tau(L_\eta^* L_\xi) \quad\text{for all right bounded vectors $\xi,\eta \in \cK$.}$$
When $Q \subset (M,\tau)$ is a von Neumann subalgebra, we denote by $e_Q$ the orthogonal projection of $\rL^2(M)$ onto $\rL^2(Q)$. Jones' \emph{basic construction} $\langle M,e_Q \rangle$ is the von Neumann algebra generated by $M$ and $e_Q$ on the Hilbert space $\rL^2(M)$. We have that $\langle M,e_Q \rangle = \B(\rL^2(M)) \cap (Q\op)'$. So, applying the above construction to the right $Q$-module $\rL^2(M)_Q$, we recover the usual semifinite faithful normal trace $\Tr$ on $\langle M,e_Q \rangle$ characterized by
$$\Tr(x e_Q y) = \tau(xy) \quad\text{for all}\;\; x,y \in M \; .$$
The number $\Tr(1)$ is called the \emph{Jones index} of $Q \subset M$ and is denoted by $[M:Q]$.

We also recall the Connes tensor product of bimodules. Assume that $\bim{M}{\cK}{Q}$ and $\bim{Q}{\cH}{P}$ are bimodules between tracial von Neumann algebras $M$, $Q$ and $P$. Denote by $\cK_0 \subset \cK$ the subspace of right $Q$-bounded vectors in $\cK$. The separation/completion of $\cK_0 \otalg \cH$ with respect to the scalar product
$$\langle \xi \ot_Q \eta , \xi' \ot_Q \eta' \rangle := \langle (L_{\xi'}^* L_\xi) \eta, \eta' \rangle$$
together with the bimodule action
$$x \cdot (\xi \ot_Q \eta) \cdot y := x\xi \ot_Q \eta y$$
yields an $M$-$P$-bimodule that is denoted by $\cK \ot_Q \cH$.

If $\bim{M}{\cK}{Q}$ is an $M$-$Q$-bimodule between the tracial von Neumann algebras $(M,\tau)$ and $(Q,\tau)$, we denote by $\bim{Q}{\overline{\cK}}{M}$ the contragredient bimodule on the adjoint Hilbert space $\overline{\cK}$ of $\cK$ with bimodule action
$$x \cdot \overline{\xi} \cdot y := \overline{y^* \xi x^*} \quad\text{for all}\;\; \xi \in \cK, x \in Q, y \in M \; .$$

Assume that $\bim{M}{\cK}{Q}$ is an $M$-$Q$-bimodule between the tracial von Neumann algebras $(M,\tau)$ and $(Q,\tau)$. Denote as above $\cN := \B(\cK) \cap (Q\op)'$, equipped with its canonical semifinite normal faithful trace $\Tr$ as explained above. Denote by $\cK_0 \subset \cK$ the subspace of right $Q$-bounded vectors. One checks that the formula
$$\cK_0 \otalg \overline{\cK_0} \recht \rL^2(\cN,\Tr) : \xi \ot_Q \overline{\eta} \mapsto L_\xi L_\eta^*$$
extends to an $M$-$M$-bimodular unitary operator of $\cK \ot_Q \overline{\cK}$ onto $\rL^2(\cN,\Tr)$.

Finally assume that $M = B \rtimes \Gamma$ is the crossed product of a countable group $\Gamma$ with a trace preserving action $\Gamma \actson (B,\tau)$. Whenever $\rho : \Gamma \recht \cU(K)$ is a unitary representation, we consider the $M$-$M$-bimodule $\bim{M}{\cK^\rho}{M}$ on the Hilbert space $\cK^\rho = K \ot \rL^2(M)$ with bimodule action
\begin{equation}\label{eq.Krho}
(b u_g) \cdot (\xi \ot x) \cdot y = \rho_g\xi \ot b u_g x y \quad\text{for all}\;\; b \in B, g \in \Gamma, \xi \in K, x,y \in M \; .
\end{equation}
If $\rho$ and $\eta$ are unitary representations, one has
$$\bim{M}{(\cK^\rho \ot_M \cK^\eta)}{M} \cong \bim{M}{\cK^{\rho \ot \eta}}{M}$$
as $M$-$M$-bimodules.

\subsection{Weak containment of representations and bimodules}\label{sec.weakcontainment}

If $\rho : \Gamma \recht \cU(K)$ and $\pi : \Gamma \recht \cU(H)$ are unitary representations of a countable group $\Gamma$, one says that $\rho$ is weakly contained in $\pi$ if $\|\rho(a)\| \leq \|\pi(a)\|$ for all $a \in \C \Gamma$. Similarly if $\bim{M}{\cK}{Q}$ and $\bim{M}{\cH}{Q}$ are bimodules between tracial von Neumann algebras $(M,\tau)$ and $(Q,\tau)$ we say that $\cK$ is weakly contained in $\cH$ if $\|\pi_\cK(x)\| \leq \|\pi_\cH(x)\|$ for all $x \in M \otalg Q\op$, where we denote by $\pi_\cK$, resp.\ $\pi_\cH$, the obvious $*$-representation associated with the bimodule structure.

Weak containment of bimodules is well behaved w.r.t.\ the Connes tensor product. If $\bim{M}{\cK}{Q}$ is weakly contained in $\bim{M}{\cH}{Q}$, then $\cK \ot_Q \cL$ is weakly contained in $\cH \ot_Q \cL$ for all $Q$-$P$-bimodules $\cL$. A similar statement holds for weak containment in the second variable.

If $M = B \rtimes \Gamma$ is a crossed product von Neumann algebra by a trace preserving action $\Gamma \actson (B,\tau)$ and if $\rho : \Gamma \recht \cU(K)$ and $\pi : \Gamma \recht \cU(H)$ are unitary representations, then $\rho$ is weakly contained in $\pi$ if and only if the $M$-$M$-bimodule $\cK^\rho$ described in \eqref{eq.Krho} is weakly contained in the $M$-$M$-bimodule $\cK^\pi$.

\subsection{Relative amenability of subalgebras and left amenability of bimodules}\label{sec.rel-amen}

A tracial von Neumann algebra $(M,\tau)$ is called \emph{amenable} if there exists an $M$-central state on $\B(\rL^2(M))$ whose restriction to $M$ equals $\tau$.
Connes' fundamental theorem in \cite{Co75} says that a tracial von Neumann algebra $M$ is amenable if and only if $M$ is hyperfinite, i.e.\ $M$ admits an increasing net of finite dimensional von Neumann subalgebras whose union is weakly dense in $M$. Also $M$ is amenable if and only if the trivial bimodule $\bim{M}{\rL^2(M)}{M}$ is weakly contained in the coarse bimodule $\bim{M}{(\rL^2(M) \ot \rL^2(M))}{M}$.

\begin{definition}[Section 2.2 in \cite{OP07}]\label{def.rel-amen}
Let $(M,\tau)$ be a tracial von Neumann algebra and let $P \subset pMp$ and $Q \subset M$ be von Neumann subalgebras. We say that $P$ is \emph{amenable relative to $Q$,} if the von Neumann algebra $p \langle M, e_Q \rangle p$ admits a $P$-central positive functional whose restriction to $pMp$ coincides with $\tau$.
\end{definition}

Recall that the basic construction von Neumann algebra $\langle M,e_Q \rangle$
coincides with the commutant of $Q\op$ acting on $\rL^2(M)$. Replacing in the above definition $\langle M, e_Q \rangle=(Q\op)'\cap \B(\rL^2M)$ by $(Q\op)'\cap \B(\cK)$ for an arbitrary $M$-$Q$-bimodule $\cK$, we arrive at the following definition (cf.\ \cite[Theorem 2.2]{Si10}).

\begin{definition}\label{def.left-amenable}
Let $(M,\tau)$ and $(Q,\tau)$ be tracial von Neumann algebras and $P \subset M$ a von Neumann subalgebra. We say that an $M$-$Q$-bimodule $\bim{M}{\cK}{Q}$ is \emph{left $P$-amenable} if there exists a $P$-central state $\Om$ on $\B(\cK) \cap (Q\op)'$ whose restriction to $M$ equals $\tau$.
\end{definition}

So by definition, for $P \subset pMp$ and $Q \subset M$ we have that $P$ is amenable relative to $Q$ if and only if the $pMp$-$Q$-bimodule $\bim{pMp}{p\rL^2(M)}{Q}$ is left $P$-amenable. Even more specifically, recall from \cite[Definition 3.2.1]{Po86} and \cite[Definition 2.1]{AD93} that a von Neumann subalgebra $Q \subset M$ is called \emph{co-amenable} if the whole of $M$ is amenable relative to $Q$. So $Q \subset M$ is co-amenable if and only if the bimodule $\bim{M}{\rL^2(M)}{Q}$ is left $M$-amenable.

Next note that Definition \ref{def.left-amenable} generalizes the notion of left amenability of bimodules introduced in \cite{AD93}. More precisely, an $M$-$Q$-bimodule $\bim{M}{\cK}{Q}$ is left $M$-amenable in the sense of Definition \ref{def.left-amenable} if and only if $\bim{M}{\cK}{Q}$ is left amenable in the sense of \cite[Definition 2.1]{AD93}. This follows immediately from Proposition \ref{prop.char} below.

Finally left amenability of bimodules has its origin in the concept of an \emph{amenable representation,} see \cite{Be89}. To make this link explicit, assume that $M := B \rtimes \Gamma$ is the crossed product of a countable group by a trace preserving action $\Gamma \actson (B,\tau)$. Every unitary representation $\rho : \Gamma \recht \cU(K)$ gives rise to an $M$-$M$-bimodule $\cK^\rho$ given by \eqref{eq.Krho}. This $M$-$M$-bimodule $\cK^\rho$ is left $M$-amenable if and
only if $\rho$ is an amenable representation in the sense of \cite[Definition 1.1]{Be89}, i.e.\ if and only if $\B(K)$ admits an $(\Ad \rho_g)_{g \in \Gamma}$-invariant state (see e.g.\ \cite[Proposition 3.3]{AD93}).

The proof of the following proposition is almost identical to the proof of \cite[Theorem 2.1]{OP07}. Part of the proposition also appears in \cite[Theorem 2.2]{Si10}. We nevertheless provide full details for the convenience of the reader. We refer to sections \ref{sec.connesproduct} and \ref{sec.weakcontainment} for the relevant terminology on bimodules, tensor products and weak containment.

\begin{proposition}\label{prop.char}
Let $(M,\tau)$ and $(Q,\tau)$ be tracial von Neumann algebras and $P \subset M$ a von Neumann subalgebra. Let $\bim{M}{\cK}{Q}$ be an $M$-$Q$-bimodule and denote $\cN := \B(\cK) \cap (Q\op)'$ with its canonical semifinite trace $\Tr$ as in Section \ref{sec.connesproduct}. Define the contractive linear map
$$\cT : \rL^1(\cN,\Tr) \recht \rL^1(M,\tau) : \tau(\cT(S) x) = \Tr(S x) \quad\text{for all}\;\; S \in \cN , x \in M \; .$$
Then the following statements are equivalent.
\begin{enumerate}
\item The $M$-$Q$-bimodule $\bim{M}{\cK}{Q}$ is left $P$-amenable.
\item There exists a net $\xi_n \in \rL^2(\cN,\Tr)^+$ satisfying the following properties.
\begin{itemize}
\item $0 \leq \cT(\xi_n^2) \leq 1$ for all $n$ and $\lim_n \|\cT(\xi_n^2) - 1\|_1 = 0$.
\item For all $y \in P$ we have $\lim_n \|y \xi_n - \xi_n y\|_2 = 0$.
\end{itemize}
\item The $M$-$P$-bimodule $\bim{M}{\rL^2(M)}{P}$ is weakly contained in the $M$-$P$-bimodule $\cK \ot_Q \overline{\cK}$.
\item There exists a $Q$-$P$-bimodule $\bim{Q}{\cH}{P}$ such that $\bim{M}{\rL^2(M)}{P}$ is weakly contained in the $M$-$P$-bimodule $\cK \ot_Q \cH$.
\item There exists a tracial von Neumann algebra $(N,\tau)$ and a $Q$-$N$-bimodule $\bim{Q}{\cH}{N}$ such that the $M$-$N$-bimodule $\cK \ot_Q \cH$ is left $P$-amenable.
\end{enumerate}
\end{proposition}
\begin{proof}
Assume that condition 1 holds. Take a $P$-central state $\Om \in \cN^*$ whose restriction to $M$ equals $\tau$. Identifying $\cN_* = \rL^1(\cN,\Tr)$, we can take a net of positive elements $S_n \in \rL^1(\cN,\Tr)^+$ such that $\Tr(S_n) = 1$ for all $n$ and such that $S_n \recht \Om$ in the weak$^*$ topology on $\cN^*$. It follows that
$\cT(S_n) \recht 1$ in the weak topology on $\rL^1(M,\tau)$ and that for all $y \in P$ we have that $y S_n - S_n y \recht 0$ in the weak topology on $\rL^1(\cN,\Tr)$.
After a passage to convex combinations we have $\|\cT(S_n)-1\|_1 \recht 0$ and $\|y S_n - S_n y\|_1 \recht 0$ for all $y \in P$. We will further modify the net $(S_n)$ in such a way that $0 \leq \cT(S_n) \leq 1$ for all $n$. For this we need the following standard functional calculus manipulations.

For every $\eps > 0$ and every $n$ denote by $p_{\eps,n} \in M$ the spectral projection $p_{\eps,n} := \chi_{[0,1+\eps]}(\cT(S_n))$.
Since $\|1-\cT(S_n)\|_1 \recht 0$, one checks that for every fixed $\eps > 0$ we have
$$\|S_n^{1/2} p_{\eps,n} - S_n^{1/2}\|^2_2 = \Tr((1-p_{\eps,n}) S_n) = \tau((1-p_{\eps,n}) \cT(S_n)) \recht 0 \quad\text{as $n \recht \infty$.}$$
So, for every fixed $\eps > 0$, we have $\lim_n \|p_{\eps,n} S_n p_{\eps_n} - S_n \|_1 = 0$. Put $T_{\eps,n} := (1+\eps)^{-1} p_{\eps,n} S_n p_{\eps,n}$. Then, for every $\eps > 0$, we have
$$\limsup_n \|T_{\eps,n} - S_n \|_1 \leq \eps \quad\text{and}\quad 0 \leq \cT(T_{\eps,n}) \leq 1 \quad\text{for all $n$.}$$
Reorganizing the $T_{\eps,n}$ we find a net $T_i \in \rL^1(\cN,\Tr)^+$ such that $0 \leq \cT(T_i) \leq 1$ for all $i$, such that $\|\cT(T_i) - 1\|_1 \recht 0$ and $\|y T_i - T_i y \|_1 \recht 0$ for all $y \in P$.

Defining $\xi_i := T_i^{1/2}$, we obtain a net in $\rL^2(\cN,\Tr)^+$ which, thanks to the Powers-St{\o}rmer inequality satisfies condition~2 in the formulation of the proposition.

Next assume that $(\xi_n)$ is a net in $\rL^2(\cN,\Tr)^+$ satisfying condition~2. Recall from Section \ref{sec.connesproduct} that $\rL^2(\cN,\Tr)$ can be identified with $\cK \ot_Q \overline{\cK}$ as an $M$-$M$-bimodule. Viewing $\xi_n$ as a net of vectors $\cK \ot_Q \overline{\cK}$ we get that
$$\langle x \xi_n y , \xi_n \rangle \recht \tau(xy) \quad\text{for all}\;\; x \in M, y \in P \; .$$
Hence the $M$-$P$-bimodule $\bim{M}{\rL^2(M)}{P}$ is weakly contained in the $M$-$P$-bimodule $\cK \ot_Q \overline{\cK}$. So condition~3 holds.

It is trivial that condition~3 implies condition~4.

We next prove that condition~4 implies condition~1. Condition~4 yields a net $(\xi_n)$ in an infinite multiple of $\cK \ot_Q \cH$ satisfying
$$\langle x \xi_n , \xi_n\rangle \recht \tau(x) \quad\text{for all}\;\; x \in M \quad\text{and}\quad \|y \xi_n - \xi_n y\| \recht 0 \quad\text{for all}\;\; y \in P \; .$$
The formula $S (\xi \ot_Q \eta) = S\xi \ot_Q \eta$ provides a normal representation of $\cN$ on $\cK \ot_Q \cH$ that commutes with the right $P$-module action on $\cK \ot_Q \cH$.
Choosing a state $\Om \in \cN^*$ as a weak$^*$ limit point of the net of states $S \mapsto \langle S \xi_n,\xi_n\rangle$, we have found a $P$-central state $\Om$ on $\cN$ whose restriction to $M$ equals $\tau$. So condition~1 holds.

We finally prove the equivalence of conditions 1 and 5. One implication being trivial by taking $N=Q$ and $\cH = \rL^2(Q)$, assume that the $M$-$N$-bimodule $\cL := \cK \ot_Q \cH$ is left $P$-amenable. The formula $S(\xi \ot_Q \eta) = S\xi \ot_Q \eta$ provides a normal $*$-homomorphism
$$\Theta : \B(\cK) \cap (Q\op)' \recht \B(\cL) \cap (N\op)'$$
whose restriction to $M$ is the identity. Given a $P$-central state $\Om$ on $\B(\cL) \cap (N\op)'$ with $\Om_{|M} = \tau$, the composition $\Om \circ \Theta$ is a $P$-central state on $\B(\cK) \cap (Q\op)'$ whose restriction to $M$ equals $\tau$. So condition~1 holds and the proposition is proven.
\end{proof}

\begin{corollary}\label{cor.weakcontainment}
Let $(M,\tau)$ and $(Q,\tau)$ be tracial von Neumann algebras and $P \subset M$ a von Neumann subalgebra. Let $\bim{M}{\cK}{Q}$ and $\bim{M}{\cK'}{Q}$ be a $M$-$Q$-bimodules. If $\bim{M}{\cK}{Q}$ is left $P$-amenable and weakly contained in $\bim{M}{\cK'}{Q}$, then also $\bim{M}{\cK'}{Q}$ is left $P$-amenable.
\end{corollary}
\begin{proof}
Since weak containment of bimodules is transitive and preserved under the Connes tensor product of bimodules, this is a direct consequence of the characterization of left $P$-amenability by condition~3 in Proposition \ref{prop.char}.
\end{proof}

\begin{corollary}\label{cor.passage-rel-amen}
Let $(M,\tau)$ and $(Q,\tau)$ be tracial von Neumann algebras and $P_1,P_2 \subset M$ von Neumann subalgebras. Let $\bim{M}{\cK}{Q}$ be an $M$-$Q$-bimodule.

If $\bim{M}{\cK}{Q}$ is a left $P_1$-amenable $M$-$Q$-bimodule and if $P_2$ is amenable relative to $P_1$, then $\bim{M}{\cK}{Q}$ is also left $P_2$-amenable.

In particular, if $P_1 \subset P_2$ is an inclusion of finite index and if $\bim{M}{\cK}{Q}$ is a left $P_1$-amenable $M$-$Q$-bimodule, then $\bim{M}{\cK}{Q}$ is also left $P_2$-amenable.
\end{corollary}
\begin{proof}
By condition~3 in Proposition \ref{prop.char} we have that $\bim{M}{\rL^2(M)}{P_1}$ is weakly contained in $\cK \ot_Q \overline{\cK}$. Hence
$$\bim{M}{\bigl(\rL^2(M) \ot_{P_1} \rL^2(M)\bigr)}{M} \quad\text{is weakly contained in}\quad \bim{M}{\bigl( \cK \ot_Q \overline{\cK} \ot_{P_1} \cK \ot_Q \overline{\cK}\bigr)}{M} \; .$$
Since $P_2$ is amenable relative to $P_1$, we know from condition~3 in Proposition \ref{prop.char} that $\bim{M}{\rL^2(M)}{P_2}$ is weakly contained in $\bim{M}{\bigl(\rL^2(M) \ot_{P_1} \rL^2(M)\bigr)}{P_2}$. In combination with the previous line and writing $\bim{Q}{\cH}{P_2} := \bim{Q}{\bigl(\overline{\cK} \ot_{P_1} \cK \ot_Q \overline{\cK}\bigr)}{P_2}$ we conclude that
$$\bim{M}{\rL^2(M)}{P_2} \quad\text{is weakly contained in}\quad \bim{M}{(\cK \ot_Q \cH)}{P_2} \; .$$
Condition~4 in Proposition \ref{prop.char} implies that $\bim{M}{\cK}{Q}$ is left $P_2$-amenable.

If $P_1 \subset P_2$ has finite index, then $P_2$ is trivially amenable relative to $P_1$ and hence also the final statement is proven.
\end{proof}

We next prove a result where the amenability of $P$ relative to two subalgebras $Q_1$ and $Q_2$ implies the amenability of $P$ relative to $Q_1 \cap Q_2$.
Obviously such a result cannot hold if $Q_1$ and $Q_2$ are in a generic position where typically $Q_1 \cap Q_2 = \C1$. So recall that two von Neumann subalgebras of a tracial von Neumann algebra $(M,\tau)$ are said to form a \emph{commuting square} if $E_{Q_1} \circ E_{Q_2} = E_{Q_2} \circ E_{Q_1}$, where $E_{Q_i}$ denotes the unique trace preserving conditional expectation of $M$ onto $Q_i$. In that case $E_{Q_1} \circ E_{Q_2}$ is the unique trace preserving conditional expectation of $M$ onto $Q_1 \cap Q_2$.

\begin{proposition}\label{prop.intersection}
Let $(M,\tau)$ be a tracial von Neumann algebra with von Neumann subalgebras $Q_1,Q_2 \subset M$. Assume that $Q_1$ and $Q_2$ form a commuting square and that $Q_1$ is regular in $M$.

If a von Neumann subalgebra $P \subset pMp$ is amenable relative to both $Q_1$ and $Q_2$, then $P$ is amenable relative to $Q_1 \cap Q_2$.
\end{proposition}
\begin{proof}
We use the notation
$$\cT_i : \rL^1(\langle M,e_{Q_i} \rangle) \recht \rL^1(M) : \tau(\cT_i(S) x) = \Tr(S x) \quad\text{for all}\;\; S \in \rL^1(\langle M,e_{Q_i} \rangle) \; , \; x \in M \; .$$
Since $P$ is amenable relative to $Q_1$ and relative to $Q_2$, condition~2 in Proposition \ref{prop.char} provides nets $\mu_i \in p\rL^2(\langle M,e_{Q_1} \rangle)^+p$ and $\xi_j \in p\rL^2(\langle M,e_{Q_2} \rangle)^+p$ satisfying the following properties.
$$0 \leq \cT_1(\mu_i^2) \leq p \;\;\text{for all $i$,}\quad \|\cT_1(\mu_i^2) - p\|_1 \recht 0 \quad\text{and}\quad \|y \mu_i - \mu_i y \|_2 \recht 0 \;\;\text{for all}\;\; y \in P \; ,$$
and similarly for $(\xi_j)$.

Consider the $M$-$M$-bimodule
$$\cH := \rL^2(\langle M,e_{Q_1} \rangle) \ot_M \rL^2(\langle M,e_{Q_2} \rangle) \; .$$
We will prove below that $\cH$ admits a net of vectors $\eta_k \in p\cH p$ such that
\begin{equation}\label{eq.aimH}
\|y \eta_k - \eta_k y \| \recht 0 \quad\text{for all}\;\; y \in P \quad\text{and}\quad \langle x \eta_k,\eta_k \rangle \recht \tau(x) \quad\text{for all}\;\; x \in pMp \; .
\end{equation}
Note that for every $\mu \in \rL^2(\langle M,e_{Q_1} \rangle)$ and every $j$, the vector $\mu \ot_M \xi_j \in \cH$ is well defined and satisfies
\begin{equation}\label{eq.form1}
\|\mu \ot_M \xi_j \| = \langle \mu \cT_2(\xi_j^2) , \mu \rangle^{1/2} \leq \|\mu\|_2 \; .
\end{equation}
Similarly, for every $\xi \in \rL^2(\langle M,e_{Q_2} \rangle)$ and every $i$, the vector $\mu_i \ot_M \xi$ is well defined and satisfies
\begin{equation}\label{eq.form2}
\|\mu_i \ot_M \xi\| \leq \|\xi\|_2 \; .
\end{equation}
Fix finite subsets $\cF \subset P$, $\cG \subset pMp$ and fix $\eps > 0$. We will produce a vector $\eta \in p\cH p$ such that
\begin{align}
& \|y \eta - \eta y \| \leq 2 \eps \quad\text{for all}\;\; y \in  \cF \; ,\label{eq.star1}\\
& |\langle x \eta,\eta \rangle - \tau(x) | \leq 2 \eps \quad\text{for all}\;\; x \in \cG \; .\label{eq.star2}
\end{align}
Once these two statements are proven, we find a net $(\eta_k)$ in $\cH$ satisfying conditions \eqref{eq.aimH}.

First fix $i$ such that $\|y \mu_i - \mu_i y \|_2 \leq \eps$ for all $y \in \cF$ and $|\langle x \mu_i , \mu_i \rangle - \tau(x)| \leq \eps$ for all $x \in \cG$.

Since $0 \leq \cT_1(\mu_i^2) \leq p$, it follows that for every $x \in M$, the element $\cT_1(\mu_i x \mu_i) \in \rL^1(M)$ is bounded in the uniform norm and hence belongs to $pMp$. Put $\cG' := \{\cT_1(\mu_i x \mu_i) \mid x \in \cG\}$. Then fix $j$ such that $\|y \xi_j - \xi_j y \|_2 \leq \eps$ for all $y \in \cF$ and $|\langle x \xi_j , \xi_j \rangle - \tau(x)| \leq \eps$ for all $x \in \cG'$.

Put $\eta := \mu_i \ot_M \xi_j$. Note that $\eta \in p\cH p$. We now prove that $\eta$ satisfies \eqref{eq.star1} and \eqref{eq.star2}. Take $y \in \cF$. Since $\|y \mu_i - \mu_i y\|_2 \leq \eps$, it follows from \eqref{eq.form1} that $\|y \eta - \mu_i y \ot_M \xi_j\| \leq \eps$. Note that $\mu_i y \ot_M \xi_j = \mu_i \ot_M y \xi_j$. Since $\|y \xi_j - \xi_j y\|_2 \leq \eps$, it follows from \eqref{eq.form2} that $\| \mu_i \ot_M y \xi_j - \eta y \| \leq \eps$. So \eqref{eq.star1} holds.

To prove \eqref{eq.star2} take $x \in \cG$. Note that
$$\langle x \eta, \eta \rangle = \langle x\mu_i \ot_M \xi_j , \mu_i \ot_M \xi_j \rangle = \langle \cT_1(\mu_i x \mu_i) \xi_j , \xi_j \rangle \; .$$
Since $\cT_1(\mu_i x \mu_i) \in \cG'$ it follows from our choice of $j$ that
$$|\langle x \eta,\eta \rangle - \tau(\cT_1(\mu_i x \mu_i))| \leq \eps \; .$$
But $\tau(\cT_1(\mu_i x \mu_i)) = \Tr(\mu_i x \mu_i) = \langle x \mu_i, \mu_i \rangle$ and also $|\langle x \mu_i, \mu_i \rangle - \tau(x) | \leq \eps$. Hence also \eqref{eq.star2} follows.

So we have proven the existence of a net $(\eta_k)$ in $p\cH p$ satisfying the conditions \eqref{eq.aimH}. It follows that the bimodule $\bim{pMp}{\rL^2(pMp)}{P}$ is weakly contained in the bimodule $\bim{pMp}{(p \cH p)}{P}$.

We claim that the $M$-$M$-bimodule $\cH$ is contained in a multiple of $\bim{M}{\rL^2(\langle M,e_{Q} \rangle)}{M}$ with $Q = Q_1 \cap Q_2$. Whenever $u,v \in \cN_M(Q_1)$, denote by $\cH_{u,v} \subset \cH$ the closed linear span of the vectors $\{ x e_{Q_1} u \ot_M v e_{Q_2} y \mid x,y \in M\}$. Note that $\cH_{u,v}$ is an $M$-$M$-subbimodule of $\cH$. The commuting square condition together with the formula $\Ad (uv)^* \circ E_{Q_1} = E_{Q_1} \circ \Ad (uv)^*$ guarantees that the formula
$$x e_{Q_1} u \ot_M v e_{Q_2} y \mapsto x u v \ot_Q y$$
defines an $M$-$M$-bimodular unitary of $\cH_{u,v}$ onto $\rL^2(\langle M,e_Q \rangle )$. Since $Q_1$ is regular in $M$, the subbimodules $\{\cH_{u,v}  \mid u,v \in \cN_M(Q_1) \}$ span a dense subspace of $\cH$. It then follows that $\cH$ is indeed contained in a multiple of $\rL^2(\langle M,e_Q \rangle )$ and the claim is proven.

Using the claim it follows that the bimodule $\bim{pMp}{\rL^2(pMp)}{P}$ is weakly contained in the bimodule
$$\bim{pMp}{(p \rL^2(\langle M,e_{Q} \rangle) p)}{P} = \bim{pMp}{(p \rL^2(M) \ot_Q \rL^2(M) p)}{P} \; .$$
By condition~3 in Proposition \ref{prop.char} this means that $P$ is amenable relative to $Q$.
\end{proof}

We finally prove the following easy lemma. Its proof is almost identical to the proof of \cite[Lemma 3.6]{OP07}.

\begin{lemma}\label{lem.rel-amen-crossed}
Assume that $(M,\tau)$ is a tracial von Neumann algebra with von Neumann subalgebra $A \subset M$. Let $\Lambda < \cN_M(A)$ be a countable subgroup. Assume that $\Lambda$ is amenable. Then $(A \cup \Lambda)\dpr$ is amenable relative to $A$.
\end{lemma}

Note that the von Neumann algebra $(A \cup \Lambda)\dpr$ need not be a crossed product $A \rtimes \Lambda$. In the extreme (and uninteresting) case we might even have that $\Lambda \subset \cU(A)$.

\begin{proof}
Define
$$K:=\{ \Om \in \langle M,e_A \rangle^* \mid \Om \;\;\text{is an $A$-central state satisfying $\Om_{|M} = \tau$}\;\} \; .$$
Equipped with the weak$^*$ topology, $K$ is compact and convex. Also $K$ is nonempty since the state on $\langle M,e_A \rangle \subset \B(\rL^2(M))$ implemented by the vector $1 \in \rL^2(M)$, belongs to $K$.

The formula $\al_g(\Om) = g \cdot \Om \cdot g^*$ defines an action of $\Lambda$ on $K$ by weak$^*$ homeomorphisms. Since $\Lambda$ is amenable, this action has a fixed point $\Om \in K$. So $\Om$ is a state on $\langle M,e_A \rangle$ that is $x$-central for all $x \in \lspan \{a g \mid a \in A, g \in \Lambda\}$ and that satisfies $\Om_{|M} = \tau$. It remains to prove that $\Om$ is $(A \cup \Lambda)\dpr$-central. This follows immediately since $\lspan \{a g \mid a \in A, g \in \Lambda\}$ is $\|\,\cdot\,\|_2$-dense in $(A \cup \Lambda)\dpr$ and since the Cauchy-Schwarz inequality implies that for all $x,y \in M$ we have
$$\| x \cdot \Om - y \cdot \Om \| \leq \Om((x-y)^* (x-y))^{1/2} = \|x-y\|_2$$
and similarly $\|\Om \cdot x - \Om \cdot y\| \leq \|x-y\|_2$.
\end{proof}

\subsection{A lemma on non-normal states}

The following lemma is distilled from \cite[Corollary 2.3]{OP07} and \cite[Lemma 5]{Oz10}, with a very similar proof but a more generic formulation of the result.

\begin{lemma}\label{lem.make-Om}
Let $\cN$ be a von Neumann algebra and $M \subset \cN$ a von Neumann subalgebra. Let $\cG_1 \subset \cG_2 \subset \cU(\cN)$ be subgroups such that all $u \in \cG_2$ normalize $M$. Assume that $\tau$ is a faithful normal tracial state on $M$ that is $(\Ad u)_{u \in \cG_2}$-invariant.

Assume that for every nonzero $(\Ad u)_{u \in \cG_2}$-invariant projection $p \in M$, there exists a (typically non-normal) positive functional $\Psi$ on $\cN$ satisfying the following three properties.
\begin{enumerate}
\item $\Psi(vp) = \Psi(p)$ for all $v \in \cG_1$,
\item $\Psi \circ \Ad u = \Psi$ for all $u \in \cG_2$,
\item Either $\Psi_{|pMp}$ is normal and nonzero; or $\Psi_{|pMp}$ is faithful in the sense that $\Psi(q) > 0$ for all nonzero projections $q \in pMp$.
\end{enumerate}
Then there exists a state $\Om$ on $\cN$ such that $\Om(v) = 1$ for all $v \in \cG_1$, $\Om \circ \Ad u = \Om$ for all $u \in \cG_2$ and $\Om(x) = \tau(x)$ for all $x \in M$.
\end{lemma}
\begin{proof}
We first claim that for every nonzero $(\Ad u)_{u \in \cG_2}$-invariant projection $p \in M$, there exists a nonzero $(\Ad u)_{u \in \cG_2}$-invariant projection $p_0 \in pMp$ and a positive functional $\Psi_0$ on $p_0 \cN p_0$ such that
\begin{itemize}
\item $\Psi_0(vp_0) = \Psi_0(p_0)$ for all $v \in \cG_1$,
\item $\Psi_0 \circ \Ad u = \Psi_0$ for all $u \in \cG_2$,
\item The restriction of $\Psi_0$ to $p_0 M p_0$ is normal and faithful.
\end{itemize}

Given a nonzero $(\Ad u)_{u \in \cG_2}$-invariant projection $p \in M$, take a positive functional $\Psi$ on $\cN$ satisfying properties 1, 2 and 3 in the formulation of the lemma. First assume that $\Psi_{|pMp}$ is normal and nonzero. Since $\Psi_{|pMp}$ is $(\Ad u)_{u \in \cG_2}$-invariant, the support of the nonzero normal positive functional $\Psi_{|pMp}$ also is $(\Ad u)_{u \in \cG_2}$-invariant. We denote this support by $p_0$ and define $\Psi_0(S) := \Psi(p_0 S p_0)$. Note that $p_0$ is a nonzero projection in $pMp$ and that $\Psi(p-p_0) = 0$. Hence the Cauchy Schwarz inequality implies that $\Psi(v(p-p_0)) = 0$ for all $v \in \cG_1$. We conclude that $\Psi_0(v p_0) = \Psi_0(p_0)$ for all $v \in \cG_1$. The other conditions are obvious and we have shown the claim in the case where $\Psi_{|pMp}$ is normal and nonzero.

Next assume that $\Psi_{|pMp}$ is faithful. Replacing $\Psi$ by $\Psi(p \cdot p)$, properties 1 and 2 in the formulation of the lemma remain valid and $\Psi(S) = \Psi(Sp) = \Psi(p S)$ for all $S \in \cN$. Still $\Psi_{|pMp}$ is faithful. We prove now that the claim holds with $p_0 = p$.

We consider the bidual von Neumann algebras $M^{**}$ and $\cN^{**}$. We view $M$, resp.\ $\cN$, as weakly dense C$^*$-subalgebras of $M^{**}$, resp.\ $\cN^{**}$. We denote by $\theta : M^{**} \recht \cN^{**}$ the bidual of the inclusion $M \subset N$. Then $\theta$ is the unique normal $*$-homomorphism satisfying $\theta(x) = x$ for all $x \in M$. We denote by $\pi : M^{**} \recht M$ the unique normal $*$-homomorphism satisfying $\pi(x) = x$ for all $x \in M$. Define the central projection $z \in M^{**}$ as the support projection of $\pi$. Recall from \cite[Definition III.2.15]{Ta79} that for all $\om \in M^*$ we have that $\om = \om \cdot z + \om \cdot (1-z)$ corresponds to the unique decomposition of $\om$ as a sum of a normal and a singular functional on $M$.

Whenever $\al \in \Aut(M)$, we denote by $\al^{**}$ the bidual automorphism of $M^{**}$. Since $\al \circ \pi = \pi \circ \al^{**}$, it follows that $\al^{**}(z) = z$ for all $\al \in \Aut(M)$. For every $u \in \cG_2$, we define $\al_u \in \Aut(M)$ given by $\al_u(x) = u x u^*$ for all $x \in M$. Note that $u \theta(x) u^* = \theta(\al_u(x))$ for all $u \in \cG_2$ and all $x \in M$. Hence we get that $u \theta(x) u^* = \theta(\al_u^{**}(x))$ for all $x \in M^{**}$. It follows in particular that $u \theta(z) u^* = \theta(z)$ for all $u \in \cG_2$.

Define the positive functional $\Psi_0$ on $p \cN p$ by the formula $\Psi_0(S) = \Psi(\theta(z) S \theta(z))$. Note that the projection $\theta(z)$ commutes with $x = \theta(x)$ for all $x \in M$. So, since $\Psi(1-p) = 0$, also $\Psi_0(1-p) = 0$ and $\Psi_0(S) = \Psi_0(S p) = \Psi_0(pS)$ for all $S \in \cN$. As explained above, $\theta(z)$ also commutes with all $u \in \cG_2$. Since $\Psi \circ \Ad u = \Psi$ for all $u \in \cG_2$, also $\Psi_0 \circ \Ad u = \Psi_0$ for all $u \in \cG_2$.

Next take $v \in \cG_1$. Denote $d = 1 - (v+v^*)/2$. Note that $d$ is a positive element in $\cN$ and that $\Psi(d) = \Psi(d p) = 0$. Since $\theta(z)$ commutes with $v$, we also have that $\theta(z)$ commutes with $d$. Therefore, using the Cauchy Schwarz inequality
$$\Psi_0(d)^2 = |\Psi(\theta(z) d \theta(z))|^2 = |\Psi(\theta(z) d)|^2 \leq \Psi(\theta(z) d^{1/2} \theta(z)) \, \Psi(d) = 0 \; .$$
We conclude that $\Psi_0(v p) = \Psi_0(v) = \Psi_0(1) = \Psi_0(p)$ for all $v \in \cG_1$.

Denote by $\om$ the restriction of $\Psi$ to $p M p$. Denote by $\om = \om\nor+\om\sing$ the unique decomposition of $\om$ as the sum of a normal and a singular functional. As observed above the restriction of $\Psi_0$ to $pMp$ equals $\om\nor$. We know that $\om$ is faithful on $pMp$. It remains to show that $\om\nor$ is still faithful. Assume that $q \in pMp$ is a projection and that $\om\nor(q) = 0$. We have to prove that $q = 0$.
By \cite[Theorem III.3.8]{Ta79} we can take an increasing sequence of projections $p_k \in M$ such that $p_k \recht 1$ strongly and $\om\sing(p_k) = 0$ for all $k$.
Consider the projections $q \wedge p_k$ and note that $q \wedge p_k \recht q$ strongly. Indeed, since the projection $q - q \wedge p_k$ is equivalent with the projection $q \vee p_k - p_k$, we have
$$\tau(q-q \wedge p_k) = \tau(q \vee p_k) - \tau(p_k) \leq 1 - \tau(p_k) \recht 0 \; .$$
Since $q \wedge p_k \leq q$ and $\om\nor(q) = 0$, we have $\om\nor(q \wedge p_k) = 0$ for all $k$. Since $q \wedge p_k \leq p_k$ and $\om\sing(p_k) = 0$, we have $\om\sing(q \wedge p_k) = 0$ for all $k$. Hence, $\om(q \wedge p_k) = 0$ for all $k$. Since $\om$ is faithful on $p M p$, we conclude that $q \wedge p_k = 0$ for all $k$. Since $q \wedge p_k \recht q$ strongly, also $q = 0$. So we established the claim in the beginning of the proof.

Using Zorn's lemma take a maximal sequence $(p_n,\Psi_n)_{n \in \N}$ where the $p_n$ are mutually orthogonal $(\Ad u)_{u \in \cG_2}$-invariant nonzero projections in $M$ and the $\Psi_n$ are positive functionals on $p_n \cN p_n$ such that $\Psi_n(v p_n) = \Psi_n(p_n)$ for all $v \in \cG_1$, such that $\Psi_n \circ \Ad u = \Psi_n$ for all $u \in \cG_2$ and such that the restriction of $\Psi_n$ to $p_n M p_n$ is a faithful normal positive functional $\om_n$.

By the claim in the beginning of the proof and by the maximality of the family $(p_n,\Psi_n)$, it follows that $\sum_n p_n = 1$. Define the normal faithful $(\Ad u)_{u \in \cG_2}$-invariant state $\om$ on $M$ given by
$$\om(x) = \sum_{k=1}^\infty \frac{\tau(p_k)}{\om_k(p_k)} \om_k(p_k x p_k) \; .$$
Define the sequence of positive functionals $\Phi_n$ on $\cN$ given by the formula
$$\Phi_n(S) := \sum_{k=1}^{n} \frac{\tau(p_k)}{\Psi_k(p_k)} \Psi_k(p_k S p_k) \; .$$
Choose a state $\Phi$ on $\cN$ as a weak$^*$ limit point of the sequence $(\Phi_n)$. By construction we have that $\Phi(v) = \Phi(1)$ for all $v \in \cG_1$, that $\Phi \circ \Ad u = \Phi$ for all $u \in \cG_2$ and that $\Phi_{|M} = \om$.

Take $h \in \rL^1(M)^+$ such that $\om(x) = \tau(xh)$ for all $x \in M$. Note that the kernel of $h$ is trivial because $\om$ is a faithful normal state on $M$. Since both $\om$ and $\tau$ are $(\Ad u)_{u \in \cG_2}$-invariant, it follows that $h$ is $(\Ad u)_{u \in \cG_2}$-invariant.
Define $\Om \in \cN^*$ as any weak$^*$ limit point of the sequence of positive functionals
$$S \mapsto \Phi\bigl( (h+1/k)^{-1/2} S (h+1/k)^{-1/2} \bigr) \; .$$
By construction $\Om(x) = \tau(x)$ for all $x \in M$. Since both $\Phi$ and $(h+1/k)^{-1/2}$ are $(\Ad u)_{u \in \cG_2}$-invariant, also $\Om \circ \Ad u = \Om$ for all $u \in \cG_2$. Finally, take $v \in \cG_1$ and put $d := 1 - (v+v^*)/2$. Since $\cG_1 \subset \cG_2$, we see that $d$ commutes with $(h+1/k)^{-1/2}$ for all $k$. Using the Cauchy Schwarz inequality we get for every $k$ that
$$\Phi\bigl( (h+1/k)^{-1/2} d (h+1/k)^{-1/2} \bigr)^2 = \bigl|\Phi\bigl( (h+1/k)^{-1} d\bigl)\bigr|^2 \leq \Phi\bigl( (h+1/k)^{-1} d (h+1/k)^{-1} \bigr) \, \Phi(d) = 0 \; .$$
So also $\Om(d) = 0$ and hence $\Om(v) = 1$ for all $v \in \cG_1$.
\end{proof}

\section{Formulation of the key technical theorem}\label{sec.source}

If $c : \Gamma \recht K_\R$ is a $1$-cocycle into the orthogonal representation $\eta : \Gamma \recht \cO(K_\R)$, the function $g \mapsto \|c(g)\|^2$ is conditionally of negative type. By Schoenberg's theorem the formula
$$\psi_t : \Gamma \recht \R : \psi_t(g) := \exp(-t \|c(g)\|^2)$$
defines a $1$-parameter family $(\psi_t)_{t > 0}$ of functions of positive type on $\Gamma$.

Assume that $M = B \rtimes \Gamma$ is a crossed product of $\Gamma$ by a trace preserving action $\Gamma \actson (B,\tau)$. Associated with the $1$-cocycle $c : \Gamma \recht K_\R$, we get a $1$-parameter group $(\psi_t)_{t > 0}$ of unital completely positive normal trace preserving maps
\begin{equation}\label{eq.psit}
\psi_t : M \recht M : \psi_t(b u_g) = \exp(-t \|c(g)\|^2) \, b u_g \quad\text{for all}\;\; b \in B , g \in \Gamma \; .
\end{equation}

Recall from \eqref{eq.Krho} that we associated to every unitary representation $\eta : \Gamma \recht \cU(K)$ an $M$-$M$-bimodule $\cK^\eta$ defined by
\begin{equation}\label{eq.Krhobis}
\begin{split}
& \cK^\eta := K \ot \rL^2(M) \quad\text{and}\\
& (b u_g) \cdot (\xi \ot x) \cdot y = \eta_g\xi \ot b u_g x y \quad\text{for all}\;\; b \in B, g \in \Gamma, \xi \in K, x,y \in M \; .
\end{split}
\end{equation}
Whenever $K_\R$ is a real Hilbert space, we denote by $K$ its complexification. If $\eta : \Gamma \recht \cO(K_\R)$ is an orthogonal representation, we still denote by $\eta$ the corresponding unitary representation on $K$.

\begin{theorem}\label{thm.source}
Let $\Gamma$ be a weakly amenable group and $c: \Gamma \recht K_\R$ a $1$-cocycle into the orthogonal representation $\eta : \Gamma \recht \cO(K_\R)$.

Let $\Gamma \overset{\si}{\actson} (B,\tau)$ be any trace preserving action on a tracial von Neumann algebra $(B,\tau)$. Denote $M = B \rtimes \Gamma$. We consider the $M$-$M$-bimodule $\cK^\eta$ associated with the complexification of $\eta$ as in \eqref{eq.Krhobis}. We denote by $(\psi_t)_{t > 0}$ the $1$-parameter group of completely positive maps associated with $c : \Gamma \recht K_\R$ as in \eqref{eq.psit}.

Let $q \in M$ be a projection and $A \subset qMq$ any von Neumann subalgebra that is amenable relative to $B$. Denote by $P := \cN_{qMq}(A)\dpr$ its normalizer. Then at least one of the following statements holds.
\begin{itemize}
\item The $qMq$-$M$-bimodule $\bim{qMq}{(q \cK^\eta)}{M}$ is left $P$-amenable in the sense of Definition \ref{def.left-amenable};
\item or, there exist $t,\delta > 0$ such that $\|\psi_t(a)\|_2 \geq \delta$ for all $a \in \cU(A)$.
\end{itemize}
\end{theorem}

\section{\boldmath Proof of Theorem \ref{thm.source}: reduction to $\Gamma$ acting trivially}

\begin{lemma}\label{lem.reduction}
It suffices to prove Theorem \ref{thm.source} for the trivial action $\Gamma \actson (B,\tau)$ on \emph{arbitrary} tracial von Neumann algebras $(B,\tau)$.
\end{lemma}
\begin{proof}
Assume that Theorem \ref{thm.source} holds for the trivial action of $\Gamma$ on an arbitrary tracial von Neumann algebra. Let then $\Gamma \actson (B,\tau)$ be an any trace preserving action. Denote $M = B \rtimes \Gamma$ and let $A \subset qMq$ be a von Neumann subalgebra that is amenable relative to $B$. Denote by $P := \cN_{qMq}(A)\dpr$ the normalizer of $A$ inside $qMq$.
As in the formulation of Theorem \ref{thm.source} we consider the $M$-$M$-bimodule $\cK^\eta$ on the Hilbert space $\cK^\eta = K \ot \rL^2(M)$, and we consider the $1$-parameter group $(\psi_t)_{t > 0}$ of completely positive maps on $M$, associated with the $1$-cocycle $c : \Gamma \recht K_\R$.

Put $\cM := M \ovt \rL(\Gamma)$ and view $\cM$ as the crossed product of $M$ with the trivial action of $\Gamma$. Define
$$\Delta : M \recht \cM : \Delta(b u_g) = bu_g \ot u_g \quad\text{for all}\quad b \in B, g \in \Gamma \; .$$
Define $\qtil := \Delta(q)$, $\cA := \Delta(A)$ and $\cP := \cN_{\qtil\cM\qtil}(\cA)\dpr$. Note that $\Delta(P) \subset \cP$.

We prove that $\cA$ is amenable relative to $M \ot 1$. Since $A$ is amenable relative to $B$, it follows from Proposition \ref{prop.char}.3 that the bimodule $\bim{qMq}{\rL^2(qMq)}{A}$ is weakly contained in the bimodule $\bim{qMq}{(q \rL^2(M) \ot_B \rL^2(M)q)}{A}$. We take on the left the Connes tensor product with the bimodule $\bim{\qtil\cM\qtil}{\rL^2(\qtil \cM \qtil)}{\Delta(qMq)}$, in which the right module action of $x \in qMq$ is given by the right multiplication with $\Delta(x)$. It follows that the bimodule $\bim{\qtil\cM\qtil}{\rL^2(\qtil \cM \qtil)}{\Delta(A)}$ is weakly contained in the bimodule
$$\bim{\qtil \cM \qtil}{\cL}{A} := \bigl(\bim{\qtil\cM\qtil}{\rL^2(\qtil \cM \qtil)}{\Delta(qMq)}\bigr) \; \ot_{qMq} \; \bigl(\bim{qMq}{(q \rL^2(M) \ot_B \rL^2(M)q)}{A}\bigr) \; .$$
The following direct computation shows that the map $S \ot_{qMq} (x \ot_B y) \mapsto S \Delta(x) \ot_{M \ot 1} \Delta(y)$ extends to a bimodular isometry of $\bim{\qtil \cM \qtil}{\cL}{A}$ into the bimodule $\bim{\qtil \cM \qtil}{(\qtil \rL^2(\cM) \ot_{M \ot 1} \rL^2(\cM) \qtil)}{\Delta(A)}$. Indeed, for all $S,T \in \qtil \cM \qtil$, $x,a \in q M$ and $y,b \in M q$ we have
\begin{align*}
\langle S \ot_{qMq} (x \ot_B y), T \ot_{qMq} (a \ot_B b) \rangle &= \tau\bigl( (b^* \ot 1) E_{B \ot 1}\bigl( \Delta(a^*) T^* S \Delta(x)\bigr) (y \ot 1) \bigr) \\
&= \tau\bigl( (E_B(yb^*) \ot 1) \; \Delta(a^*) T^* S \Delta(x) \bigr) \\
&= \tau\bigl( E_{M \ot 1}(\Delta(yb^*)) \; \Delta(a^*) T^* S \Delta(x) \bigr) \\
&= \tau\bigl( \Delta(yb^*) \; E_{M \ot 1}\bigl(\Delta(a^*) T^* S \Delta(x)\bigr) \bigr) \\
&= \langle S \Delta(x) \ot_{M \ot 1} \Delta(y), T \Delta(a) \ot_{M \ot 1} \Delta(b) \rangle \; .
\end{align*}
So the bimodule $\bim{\qtil \cM \qtil}{(\qtil \rL^2(\cM) \ot_{M \ot 1} \rL^2(\cM) \qtil)}{\Delta(A)}$ weakly contains $\bim{\qtil\cM\qtil}{\rL^2(\qtil \cM \qtil)}{\Delta(A)}$. Proposition \ref{prop.char}.3 then says that $\Delta(A)$ is amenable relative to $M \ot 1$.

For the trivial crossed product $\cM$, we also consider the $\cM$-$\cM$-bimodule $\cKtil^\eta$ on the Hilbert space $\cKtil^\eta = K \ot \rL^2(\cM)$, and the $1$-parameter group of completely positive maps $(\psitil_t)_{t > 0}$ on $\cM$. Since we assumed that Theorem \ref{thm.source} holds for the trivial action and since we have proven above that $\cA$ is amenable relative to $M \ot 1$, at least one of the following statements is true.
\begin{itemize}
\item The $\qtil\cM\qtil$-$\cM$-bimodule $\qtil\cKtil^\eta$ is left $\cP$-amenable;
\item or, there exist $t,\delta > 0$ such that $\|\psitil_t(a)\|_2 \geq \delta$ for all $a \in \cU(\cA)$.
\end{itemize}
We prove now that these options lead respectively to the left $P$-amenability of $\bim{qMq}{(q\cK^\eta)}{M}$, or the inequality $\|\psi_t(a)\|_2 \geq \delta$ for all $a \in \cU(\cA)$. Once this is proven, also the lemma is proven.

First assume that the $\qtil\cM\qtil$-$\cM$-bimodule $\qtil\cKtil^\eta$ is left $\cP$-amenable. View $\cKtil^\eta$ as an $M$-$\cM$-bimodule using the left module action by $\Delta(x)$, $x \in M$. So a fortiori $\bim{qMq}{(\qtil\cKtil^\eta)}{\cM}$ is left $P$-amenable. Viewing $\rL^2(\cM)$ as an $M$-$\cM$-bimodule by using also here the left module action by $\Delta(x)$, $x \in M$, we observe that $\bim{M}{\cKtil^\eta}{\cM}$ is canonically isomorphic with $\bim{M}{(\cK^\eta \ot_M \rL^2(\cM))}{\cM}$. We conclude that the bimodule $\bim{qMq}{(q \cK^\eta \ot_M \rL^2(\cM))}{\cM}$ is left $P$-amenable. By condition~5 in Proposition \ref{prop.char} we get that also $\bim{qMq}{(q \cK^\eta)}{M}$ is left $P$-amenable.

Since $\psitil_t \circ \Delta = \Delta \circ \psi_t$, the inequality $\|\psitil_t(a)\|_2 \geq \delta$ for all $a \in \cU(\cA)$ immediately implies that $\|\psi_t(a)\|_2 \geq \delta$ for all $a \in \cU(A)$.
\end{proof}

\section{\boldmath Weak amenability produces almost invariant states}

We prove the following theorem, which will be the first step towards the proof of Theorem \ref{thm.source}. We use the notation $\ubar := (u\op)^*$.

\begin{theorem}\label{thm.omn}
Let $\Gamma$ be a weakly amenable group and $(B,\tau)$ any tracial von Neumann algebra. Write $M := B \ovt \rL(\Gamma)$ and assume that $A \subset M$ is a von Neumann subalgebra that is amenable relative to $B$. Denote its normalizer by $P := \cN_M(A)\dpr$. Define $N$ as the von Neumann algebra generated by $B$ and $P\op$ on the Hilbert space $\rL^2(M) \ot_A \rL^2(P)$. Put $\cN := N \ovt \rL(\Gamma)$ and define the tautological embeddings
$$\pi : M \recht \cN : \pi(b \ot u_g) = b \ot u_g \quad\text{and}\quad \theta : P\op \recht \cN : \theta(y\op) = y\op \ot 1$$
for all $b \in B, g \in \Gamma, y \in P$.

Then there exists a net of normal states $\om_i \in \cN_*$ satisfying the following properties.
\begin{itemize}
\item $\om_i(\pi(x)) \recht \tau(x)$ for all $x \in M$,
\item $\om_i(\pi(a) \theta(\abar)) \recht 1$ for all $a \in \cU(A)$,
\item $\|\om_i \circ \Ad(\pi(u) \theta(\ubar)) - \om_i\| \recht 0$ for all $u \in \cN_M(A)$.
\end{itemize}
\end{theorem}

\subsection{\boldmath Easy proof of Theorem \ref{thm.omn} when $\Gamma$ has CMAP}

In the case where $\Gamma$ has CMAP, the proof of Theorem \ref{thm.omn} is very similar to the proof of \cite[Theorem 3.5]{OP07}.

Fix a sequence $f_n : \Gamma \recht \C$ of finitely supported functions tending to $1$ pointwise and satisfying $\limsup_n \|f_n\|\cb = 1$. Denote by $\m_n : \rL(\Gamma) \recht \rL(\Gamma)$ the corresponding normal completely bounded maps given by $\m_n(u_g) = f_n(g) u_g$ for all $g \in \Gamma$. We also define $\vphi_n : M \recht M$ given by $\vphi_n = \id \ot \m_n$.

Define the von Neumann algebras $N$ and $\cN$, together with the embeddings $\pi : M \recht \cN$ and $\theta : P\op \recht \cN$ as in the formulation of Theorem \ref{thm.omn}. Note that $\pi(M)$ commutes with $\theta(P\op)$ and that together they generate $\cN$.

\begin{proof}[Proof of Theorem \ref{thm.omn} in the case where $\Gamma$ has CMAP]
Denote by $\bim{M}{\cK}{M}$ the $M$-$M$-bimodule $\cK := \rL^2(M) \ot_B \rL^2(M)$ and explicitly denote by $\lambda : M \recht \B(\cK)$ and $\rho : M\op \recht \B(\cK)$ the normal $*$-homomorphisms given by the left and the right bimodule action. Define the von Neumann algebra $\cS_A := \lambda(M) \vee \rho(A\op)$.

We claim that there exists a normal completely positive unital map $\cE : \cN \recht \cS_A$ satisfying
$$\cE(\pi(x) \theta(y\op)) = \lambda(x) \rho(E_A(y)\op) \quad\text{for all}\;\; x \in M, y \in P \; .$$
To prove this claim, recall that $\cN$ is defined as the von Neumann algebra acting on $(\rL^2(M) \ot_A \rL^2(P)) \ot \ell^2(\Gamma)$ generated by $\pi(M)$ and $\theta(P\op)$. The formula
\begin{align*}
V : \cK \recht (\rL^2(M) \ot_A \rL^2(P)) \ot \ell^2(\Gamma) : V ((b \ot u_g) \ot_B x) & = (bx \ot_A 1) \ot \delta_g \\ & \text{for all}\;\; b \in B, g \in \Gamma, x \in M \; ,
\end{align*}
yields a well defined isometry and $\cE$ can be defined by the formula $\cE(z) = V^* z V$ for all $z \in \cN$. This proves the claim.

We next claim that there exists a sequence of normal functionals $\mu_n^A \in (\cS_A)_*$ satisfying
$$\mu_n^A(\lambda(x) \rho(a\op)) = \tau(\vphi_n(x) a)\quad\text{for all}\;\; x \in M , a \in A \; .$$
This claim follows from a direct computation and the formula
$$\mu_n^A(T) = \sum_{g \in \supp f_n} \; f_n(g) \; \langle T \, (1 \ot_B (1 \ot u_g)) , ((1 \ot u_g) \ot_B 1) \rangle \quad\text{for all}\;\; T \in \cS_A \; ,$$
which is meaningful because $f_n$ is finitely supported.

We define $\gamma_n \in \cN_*$ by the formula $\gamma_n = \mu_n^A \circ \cE$ and put $\om_n := \|\gamma_n\|^{-1} |\gamma_n|$. We will prove that $\om_n \in \cN_*$ is a sequence of normal states that satisfies the conclusion of Theorem \ref{thm.omn}. Note that by definition
\begin{equation}\label{eq.defining-gamman}
\gamma_n(\pi(x)\theta(y\op)) = \tau(\vphi_n(x) E_A(y)) \quad\text{for all}\;\; x \in M, y \in P \; .
\end{equation}

For every $u \in \cN_M(A)$ the expression $\Ad(\lambda(u) \rho(\ubar))$ defines an automorphism of $\cS_A$.
We will prove the following two statements.
\begin{enumerate}
\item $\limsup_n \|\mu_n^A\| = 1$.
\item $\lim_n \|\mu_n^A \circ \Ad(\lambda(u) \rho(\ubar)) - \mu_n^A \| = 0$ for all $u \in \cN_M(A)$.
\end{enumerate}
Once these two statements are proven, we get that $\limsup_n \|\gamma_n\| = 1$ because $\gamma_n = \mu_n^A \circ \cE$. Since $\gamma_n(1) \recht 1$, it will follow that $\|\gamma_n - \om_n\| \recht 0$. Because $\cE \circ \Ad(\pi(u) \theta(\ubar)) = \Ad(\lambda(u) \rho(\ubar)) \circ \cE$ for all $u \in \cN_M(A)$, we also get that
$$\lim_n \|\gamma_n \circ \Ad(\pi(u) \theta(\ubar)) - \gamma_n \| = 0 \quad\text{for all}\;\; u \in \cN_M(A) \; .$$
Then the same holds for $\om_n$ instead of $\gamma_n$ and all the required properties of $\om_n$ are proven, or follow directly from \eqref{eq.defining-gamman} and the fact that $\|\gamma_n - \om_n\| \recht 0$.

It remains to prove statements 1 and 2 above. Define $S_A$ as the C*-algebra acting on $\cK$ generated $\lambda(M)$ and $\rho(A\op)$. Note that $S_A$ is a dense C$^*$-subalgebra of $\cS_A$. Since the norm of a normal functional coincides with the norm of its restriction to a dense C$^*$-subalgebra, we from now on consider $\mu_n^A$ as a continuous functional on $S_A$ and compute all norms inside $S_A^*$.

Whenever $Q \subset P$ is a von Neumann subalgebra, we define $S_Q$ as the C$^*$-algebra acting on $\cK$ generated by $\lambda(M)$ and $\rho(Q\op)$. As with $\mu_n^A$ above, the formula
$$\mu_n^Q : S_Q \recht \C : \mu_n^Q(\lambda(x) \rho(y\op)) = \tau(\vphi_n(x) y) \quad\text{for all}\;\; x \in M, y \in Q$$
defines a sequence of continuous functionals $\mu_n^Q$ on $S_Q$. We claim that if $Q$ is amenable relative to $B$, then $\limsup_n \|\mu_n^Q\| = 1$. The special case $Q = A$ then yields statement 1 above. To prove this claim, first observe that there is a sequence of completely bounded maps $\vphitil_n : S_Q \recht S_Q$ satisfying
$$
\vphitil_n(\lambda(x) \rho(y\op)) = \lambda(\vphi_n(x)) \rho(y\op) \quad\text{for all}\quad x \in M, y \in Q \quad\text{and}\quad \|\vphitil_n\|\cb = \|f_n\|\cb \; .
$$
To see this, it suffices to consider the unitary operator
$$U : \cK \recht \rL^2(B) \ot \ell^2(\Gamma) \ot \ell^2(\Gamma) : (b \ot u_g) \ot_B (c \ot u_h) \mapsto bc \ot \delta_g \ot \delta_h$$
for all $b,c \in B, g,h \in \Gamma$, which satisfies $U \lambda(b \ot u_g) U^* =b \ot u_g \ot 1$ for all $b \in B$, $g \in \Gamma$, and $U \rho(Q\op)U^* \subset \B(\rL^2(B)) \ovt 1 \ovt \B(\ell^2(\Gamma))$. We can then define
$\vphitil_n(z) = U^* (\id \ot \m_n \ot \id)(U z U^*) U$ for all $z \in S_Q$.

Since $Q$ is amenable relative to $B$, we know from point 3 in Proposition \ref{prop.char} that the bimodule $\bim{M}{\rL^2(M)}{Q}$ is weakly contained in the bimodule $\bim{M}{\cK}{Q}$.
Denoting by $\lambda_{\rL^2(M)}$ and $\rho_{\rL^2(M)}$ the left and right module action of $M$ and $M\op$ on $\rL^2(M)$, we then get a continuous $*$-homomorphism $\Theta : S_Q \recht \B(\rL^2(M))$ satisfying
$$\Theta(\lambda(x) \rho(y\op)) = \lambda_{\rL^2(M)}(x) \rho_{\rL^2(M)}(y\op) \quad\text{for all}\;\; x \in M, y \in Q \; .$$
Since
$$\mu_n^Q(z) = \langle \Theta(\vphitil_n(z)) \, 1 , 1 \rangle \quad\text{for all}\;\; z \in Q \; ,$$
the above claim follows and also statement 1 is proven.

To prove statement 2, fix $u \in \cN_M(A)$ and define $Q \subset P$ as the von Neumann algebra generated by $A$ and $u$. By Lemma \ref{lem.rel-amen-crossed}, $Q$ is amenable relative to $A$. Since $A$ is amenable relative to $B$, it then follows from Corollary \ref{cor.passage-rel-amen} that also $Q$ is amenable relative to $B$. Therefore we have $\limsup_n \|\mu_n^Q\| = 1$. The definition of $\mu_n^Q$ immediately gives us $\mu_n^Q(1) = \tau(\vphi_n(1)) \recht 1$ as well as
$$\mu_n^Q(\lambda(u) \rho(\ubar)) = \tau(\vphi_n(u) \, E_Q(u^*)) = \tau(\vphi_n(u) \, u^*) \recht 1$$
since $u \in Q$. Since $\limsup_n \|\mu_n^Q\| = 1$, it follows that
$$\bigl\| \, \mu_n^Q \circ \Ad(\lambda(u) \rho(\ubar)) - \mu_n^Q \, \bigr\| \recht 0 \; .$$
Restricting the functionals $\mu_n^Q \circ \Ad(\lambda(u) \rho(\ubar))$ and $\mu_n^Q$ to $S_A$, statement 2 follows.

As explained above the proof of statements 1 and 2 concludes the proof of the theorem.
\end{proof}

\subsection{\boldmath Proof of Theorem \ref{thm.omn} for arbitrary weakly amenable $\Gamma$}

For arbitrary weakly amenable groups $\Gamma$, our proof of Theorem \ref{thm.omn} follows very closely the proof of \cite[Theorem B]{Oz10}.
We start by the following adaptation of \cite[Lemma 6]{Oz10}.

\begin{lemma}\label{lem.adapted}
Let $M = B \ovt D$ be the tensor product of two tracial von Neumann algebras. Let $A \subset M$ be a von Neumann subalgebra that is amenable relative to $B$. Consider the $M$-$A$-bimodule $\cK := \rL^2(M) \ot_B \rL^2(M)$ and denote by $\lambda(x)$ and $\rho(a\op)$ the left and the right module action of $x \in M$, $a \in A$. Denote by $S_A$ the C$^*$-algebra generated by $\lambda(M)$ and $\rho(A\op)$.

We say that a normal completely bounded map $\psi : M \recht M$ is \emph{adapted} if there exists a $4$-tuple $(\pi,\cH,V,W)$ consisting of a $*$-representation $\pi$ of the C$^*$-algebra $S_A$ on a Hilbert space $\cH$ and bounded maps $V, W : \cN_M(A) \recht \cH$ such that
\begin{equation}\label{eq.crucial}
\tau(w^* \psi(x) v a) = \langle \pi(\lambda(x) \rho(a\op)) V(v), W(w) \rangle \quad\text{for all}\;\; x \in M, a \in A, v,w \in \cN_M(A) \; .
\end{equation}
Here we write $\|V\|_\infty := \sup \{\|V(v)\| \mid v \in \cN_M(A)\}$.
Following \cite[Discussion after Lemma 6]{Oz10} we define $\|\psi\|_A$ as the infimum of all $\|V\|_\infty \, \|W\|_\infty$, where $(\pi,\cH,V,W)$ ranges over all $4$-tuples satisfying \eqref{eq.crucial}.

\begin{enumerate}
\item If $\m : D \recht D$ is a normal completely bounded map, then $\id \ot \m : M \recht M$ is adapted and $\|\id \ot \m\|_A \leq \|\m\|\cb$.

\item If $\psi : M \recht M$ is an adapted normal completely bounded map and if $u_1,u_2 \in \cN_M(A)$ and $x_1,x_2 \in M$, then also the normal completely bounded map $x \mapsto u_1^* \psi(x_1^* x x_2) u_2$ is adapted.
\end{enumerate}
\end{lemma}
\begin{proof}
We start by proving the first statement. Assume that $\m : D \recht D$ is a normal completely bounded map. The formula
$$U : \cK \recht \rL^2(B) \ot \rL^2(D) \ot \rL^2(D) :  (b \ot d) \ot_B (b' \ot d') \mapsto bb' \ot d \ot d'$$
yields a unitary satisfying $U \lambda(b \ot d) U^* = b \ot d \ot 1$ for all $b \in B, d \in D$ and $U \rho(A\op) U^* \subset \B(\rL^2(B)) \ovt 1 \ovt \B(\rL^2(D))$. So the formula $\psitil(z) := U^* (\id \ot \m \ot \id)(U z U^*) U$ provides a normal completely bounded map $\psitil : S_A \recht S_A$ satisfying
$$\psitil(\lambda(x) \rho(a\op)) = \lambda((\id \ot \m)(x)) \rho(a\op) \quad\text{for all}\;\; x \in M , a \in A \; .$$
Note that $\|\psitil\|\cb = \|\m\|\cb$.

Since $A$ is amenable relative to $B$, we know from point 3 in Proposition \ref{prop.char} that the bimodule $\bim{M}{\rL^2(M)}{A}$ is weakly contained in the bimodule $\bim{M}{\cK}{A}$. So we have a continuous $*$-homomorphism $\Theta : S_A \recht \B(\rL^2(M))$ satisfying
$$\Theta(\lambda(x) \rho(a\op)) = \lambda_{\rL^2(M)}(x) \rho_{\rL^2(M)}(a\op) \quad\text{for all}\;\; x \in M, a \in A \; .$$
We now apply a Stinespring-type factorization theorem (see e.g.\ \cite[Theorem B.7]{BO08}) to the completely bounded map $\Theta \circ \psitil : S_A \recht \B(\rL^2(M))$. We find a $*$-representation $\pi : S_A \recht \B(\cH)$ of $S_A$ on a Hilbert space $\cH$ and bounded operators $V, W : \rL^2(M) \recht \cH$ such that
$$\Theta(\psitil(z)) = W^* \pi(z) V \quad\text{for all}\;\; z \in S_A \quad\text{and}\quad \|V\| \, \|W\| = \| \Theta \circ \psitil \|\cb \leq \|\psitil\|\cb = \|\m\|\cb \; .$$
Define $V,W : \cN_M(A) \recht \cH$ given by restricting $V$, $W$ to $\cN_M(A) \subset \rL^2(M)$. We have $\|V\|_\infty \, \|W\|_\infty \leq \|V\| \, \|W\| \leq \|\m\|\cb$. A direct computation yields that \eqref{eq.crucial} holds for $\psi = \id \ot \m$. So $\id \ot \m$ is adapted and
$$\|\id \ot \m\|_A \leq \|V\|_\infty \, \|W\|_\infty \leq \|\m\|\cb \; .$$

The proof of the second statement is straightforward. Assume that $(\pi,\cH,V,W)$ satisfies \eqref{eq.crucial} with respect to $\psi$. Define $\psitil(x) = u_1^* \psi(x_1^* x x_2) u_2$. Put $\Vtil(v) = \pi(\lambda(x_2))(V(u_2 v))$ and $\Wtil(w) = \pi(\lambda(x_1))(V(u_1 w))$. It is straightforward to check that \eqref{eq.crucial} holds for $(\pi,\cH,\Vtil,\Wtil)$ with respect to $\psitil$. So $\psitil$ is adapted.
\end{proof}

\begin{definition}\label{def.finite-rank}
Let $(M,\tau)$ be a tracial von Neumann algebra and $B \subset M$ a von Neumann subalgebra. We say that a linear map $\psi : M \recht M$ has \emph{finite rank relative to $B$} if $\psi$ can be written as a finite linear combination of the maps $\{\psi_{y,z,r,t} \mid y,z,r,t \in M\}$ where
$$\psi_{y,z,r,t} : M \recht M : x \mapsto y E_B(z x r) t$$
and where $E_B : M \recht B$ denotes the unique trace preserving conditional expectation.

We call a net of linear maps $\psi_i : M \recht M$ an \emph{approximate identity relative to $B$} if all the $\psi_i$ are completely bounded, of finite rank relative to $B$, and if they satisfy
$$\sup_i \|\psi_i\|\cb < \infty \quad\text{and}\quad \lim_i \|\psi_i(x) - x \|_2 = 0 \quad\text{for all}\;\; x \in M \; .$$
\end{definition}

The following proposition follows by a straightforward ``relativization to $B$'' of the proof of \cite[Proposition 7]{Oz10}. For completeness we nevertheless give a detailed proof.

\begin{proposition}\label{prop.approx-ad-invariant}
Let $\Gamma$ be a weakly amenable group and $(B,\tau)$ a tracial von Neumann algebra. Put $M = B \ovt \rL(\Gamma)$ and let $A \subset M$ be a von Neumann subalgebra that is amenable relative to $B$. Consider the $M$-$A$-bimodule $\cK := \rL^2(M) \ot_B \rL^2(M)$ and denote by $\lambda(x)$ and $\rho(a\op)$ the left and the right module action of $x \in M$, $a \in A$. Denote by $S_A$ the C$^*$-algebra generated by $\lambda(M)$ and $\rho(A\op)$.

Then $M$ admits an approximate identity relative to $B$, denoted by $\psi_i : M \recht M$, such that all the $\psi_i$ are adapted in the sense of Lemma \ref{lem.adapted} and such that the functionals $\mu_i \in S_A^*$ given by
$$\mu_i : S_A \recht \C : \mu_i(\lambda(x) \rho(a\op)) = \tau(\psi_i(x) a) \quad\text{for all}\;\; x \in M, a \in A$$
satisfy
\begin{itemize}
\item $\sup_i \|\mu_i\| < \infty$,
\item $\lim_i \|\mu_i \circ \Ad(\lambda(u) \rho(\ubar)) - \mu_i\| = 0$ for all $u \in \cN_M(A)$,
\item $\lim_i \|(\lambda(v) \rho(\vbar)) \cdot \mu_i - \mu_i\| = 0$ for all $v \in \cU(A)$, where the functional $(\lambda(v) \rho(\vbar)) \cdot \mu_i$ in $S_A^*$ is defined by the formula $((\lambda(v) \rho(\vbar)) \cdot \mu_i)(z) = \mu_i(z \lambda(v) \rho(\vbar))$ for all $z \in S_A$.
\end{itemize}
\end{proposition}

\begin{proof}
Whenever $\psi_i : M \recht M$ is a normal completely bounded map that is adapted in the sense of Lemma \ref{lem.adapted}, it follows from \eqref{eq.crucial} that the corresponding functional $\mu_i$ on $S_A$ is well defined and continuous, and satisfies $\|\mu_i\| \leq \|\psi_i\|_A$. Here, and in the rest of the proof, we use the notation $\|\psi_i\|_A$ introduced in Lemma \ref{lem.adapted}.

Since $\Gamma$ is weakly amenable we can take a sequence $f_n : \Gamma \recht \C$ of finitely supported functions that tend to $1$ pointwise and satisfy $\limsup_n \|f_n\|\cb < \infty$. Denote by $\m_n : \rL(\Gamma) \recht \rL(\Gamma)$ the corresponding completely bounded maps given by $\m_n(u_g) = f_n(g) u_g$ for all $g \in \Gamma$. Then $\id \ot \m_n : M \recht M$ forms an approximate identity relative to $B$. From Lemma \ref{lem.adapted}.1 we know that $\id \ot \m_n$ is adapted and that
$$\limsup_n \|\id \ot \m_n\|_A \leq \limsup_n \|\m_n\|\cb = \limsup_n \|f_n\|\cb < \infty \; .$$

Denote by $\kappa \geq 1$ the infimum of all the numbers $\limsup_i \|\psi_i\|_A$ where $(\psi_i)$ ranges over all \emph{adapted} approximate identities of $M$ relative to $B$. Because we have the adapted approximate identity relative to $B$ given as $(\id \ot \m_n)$, we know that $\kappa < \infty$.

Then $M$ admits an adapted approximate identity relative to $B$, denoted as $\psi_i : M \recht M$, and $4$-tuples $(\pi_i,\cH_i,V_i,W_i)$ satisfying \eqref{eq.crucial} with respect to $\psi_i$ and satisfying
$$\lim_i \|V_i\|_\infty = \sqrt{\kappa} = \lim_i \|W_i\|_\infty \; .$$

We will prove that the net $(\psi_i)$ satisfies the conclusion of the proposition.

First fix $u \in \cN_M(A)$ and define
$$\psi_i^u : M \recht M : \psi_i^u(x) = \psi_i(x u^*) u \; .$$
Note that every $\psi_i^u$ still has finite rank relative to $B$ in the sense of Definition \ref{def.finite-rank}. Hence $(\psi_i^u)$ and also $(\psi_i + \psi_i^u)/2$ are approximate identities of $M$ relative to $B$. Define $V_i^u(v) := \pi_i(\lambda(u))^* V_i(v)$ for all $v \in \cN_M(A)$. A direct computation shows that $(\pi_i,\cH_i,V_i^u,W_i)$ satisfies \eqref{eq.crucial} with respect to $\psi_i^u$. So the $4$-tuple $(\pi_i,\cH_i, (V_i + V_i^u)/2,W_i)$ also satisfies \eqref{eq.crucial} with respect to $(\psi_i + \psi_i^u)/2$. We conclude that $(\psi_i + \psi_i^u)/2$, and all its subnets, are adapted approximate identities relative to $B$. It follows that $\liminf_i \|(\psi_i + \psi_i^u)/2\|_A \geq \kappa$, which implies that
$$\kappa \leq \liminf_i \|(\psi_i + \psi_i^u)/2\|_A \leq \liminf_i \|(V_i + V_i^u)/2\|_\infty \, \|W_i\|_\infty = \sqrt{\kappa} \liminf_i \|(V_i + V_i^u)/2\|_\infty \; .$$
So we can choose $v_i \in \cN_M(A)$ such that
$$\liminf_i \Bigl\|\frac{V_i(v_i) + V_i^u(v_i)}{2}\Bigr\| \geq \sqrt{\kappa} \; .$$
Since $\|V_i(v_i)\| \leq \|V_i\|_\infty \recht \sqrt{\kappa}$ and also $\|V_i^u(v_i)\| \leq \|V_i^u\|_\infty = \|V_i\|_\infty \recht \sqrt{\kappa}$, the parallelogram law implies that $\|V_i(v_i) - V_i^u(v_i)\| \recht 0$.

Now define the functionals $\mu^u_i \in S_A^*$ that are associated with $\psi_i^u$ by the formula
$$\mu^u_i : S_A \recht \C : \mu_i(\lambda(x) \rho(a\op)) = \tau(\psi_i(xu^*)u a) \quad\text{for all}\;\; x \in M, a \in A \; .$$
One computes that for all $x \in M$, $a \in A$ and all $i$ we have
\begin{align*}
& \bigl(\mu^u_i \circ \Ad (\rho(\overline{v}_i))\bigr) (\lambda(x) \rho(a\op)) = \tau(\psi_i(x u^*) u v_i a v_i^*) = \langle \pi_i(\lambda(x) \rho(a\op)) V_i^u(v_i), W_i(v_i)\rangle \; ,\\
& \bigl(\mu_i \circ \Ad (\rho(\overline{v}_i))\bigr) (\lambda(x) \rho(a\op)) = \tau(\psi_i(x) v_i a v_i^*) =  \langle \pi_i(\lambda(x) \rho(a\op)) V_i(v_i) , W_i(v_i) \rangle \; .
\end{align*}
Hence,
$$\|\mu^u_i - \mu_i\| = \bigl\| (\mu^u_i - \mu_i) \circ \Ad(\rho(\overline{v}_i))\bigr\| \leq \|V_i^u(v_i) - V_i(v_i)\| \, \|W_i(v_i)\| \recht 0 \; .$$
Starting from the approximate identity relative to $B$ given by $\psi^u_i$, we can similarly consider the approximate identity relative to $B$ given by $^u(\psi^u_i) : x \mapsto u^* \psi^u_i(u x) = u^*\psi_i(uxu^*)u$. The net of functionals corresponding to $(^u(\psi^u_i))$ is precisely $\mu_i \circ \Ad (\lambda(u) \rho(\ubar))$. So, by symmetry
$$\lim_i \|\mu_i \circ \Ad(\lambda(u) \rho(\ubar))) - \mu_i^u \| = 0 \; .$$
Since we already showed that $\lim_i \|\mu^u_i - \mu_i\| =0$ we arrive at the required result  that $$\lim_i \|\mu_i \circ \Ad (\lambda(u) \rho(\ubar)) - \mu_i\|=0$$ for all $u \in \cN_M(A)$.

Finally, if $v \in \cU(A)$ we have $(\lambda(v^*) \rho(v\op)) \cdot \mu_i = \mu_i^v$. Since $v \in \cU(A)$ certainly normalizes $A$, we already showed that $\|\mu_i^v - \mu_i\| \recht 0$. Hence also
$\lim_i \|(\lambda(v^*) \rho(v\op))\cdot \mu_i - \mu_i\| = 0$ and the proposition is proven.
\end{proof}

Finally we are ready to prove Theorem \ref{thm.omn}

\begin{proof}[Proof of Theorem \ref{thm.omn}]
Take an adapted approximate identity $(\psi_i)$ of $M$ relative to $B$ satisfying the conclusion of Proposition \ref{prop.approx-ad-invariant}. This means that the continuous functionals
$$\mu_i : S_A \recht \C : \mu_i(\lambda(x) \rho(a\op)) = \tau(\psi_i(x) a) \quad\text{for all}\;\; x \in M, a \in A$$
satisfy $\sup_i \|\mu_i\| < \infty$, $\lim_i \|\mu_i \circ \Ad(\lambda(u) \rho(\ubar)) - \mu_i\| = 0$ for all $u \in \cN_M(A)$ and $\lim_i \|(\lambda(a) \rho(\abar)) \cdot \mu_i - \mu_i\| = 0$ for all $a \in \cU(A)$.

Define the von Neumann algebra $\cS_A := \lambda(M) \vee \rho(A\op)$ acting on the Hilbert space $\cK = \rL^2(M) \ot_B \rL^2(M)$. Observe that $S_A$ is a weakly dense C$^*$-subalgebra of $\cS_A$. We claim that the functionals $\mu_i \in S_A^*$ are normal on $\cS_A$. The $\psi_i$ have finite rank relative to $B$ in the sense of Definition \ref{def.finite-rank}. Using the notation introduced in Definition \ref{def.finite-rank}, in order to prove the claim, it suffices to construct for every $y,z,r,t \in M$ a normal functional $\mu_{y,z,r,t} \in (\cS_A)_*$ satisfying
$$
\mu_{y,z,r,t}(\lambda(x) \rho(a\op))=  \tau(\psi_{y,z,r,t}(x) a) \quad\text{for all}\;\; x \in M , a \in A \; .
$$
Since $\cK = \rL^2(M) \ot_B \rL^2(M)$, a straightforward computation yields that we can take $\mu_{y,z,r,t}$ of the form
$$\mu_{y,z,r,t}(T) = \langle T (r \ot_B t), z^* \ot_B y^* \rangle \quad\text{for all}\;\; T \in \cS_A \; .$$
This proves the claim on the normality of the functionals $\mu_i$.

We next claim that there exists a normal completely positive unital map $\cE : \cN \recht \cS_A$ satisfying
$$\cE : \cN \recht \cS_A : \cE(\pi(x) \theta(y\op)) = \lambda(x) \rho(E_A(y)\op) \quad\text{for all}\;\; x \in M, y \in P \; .$$
To prove this claim, recall that $\cN$ is defined as the von Neumann algebra acting on $(\rL^2(M) \ot_A \rL^2(P)) \ot \ell^2(\Gamma)$ generated by $\pi(M)$ and $\theta(P\op)$. The formula
$$V : \cK \recht (\rL^2(M) \ot_A \rL^2(P)) \ot \ell^2(\Gamma) : V ((b \ot u_g) \ot_B x) = (bx \ot_A 1) \ot \delta_g$$
yields a well defined isometry and $\cE$ can be defined by the formula $\cE(z) = V^* z V$ for all $z \in \cN$. This proves the claim.

Define the normal functionals $\gamma_i \in \cN_*$ by the formula $\gamma_i := \mu_i \circ \cE$. Note that
\begin{equation}\label{eq.defgammai}
\gamma_i(\pi(x) \theta(y\op))=  \tau(\psi_i(x) E_A(y)) \quad\text{for all}\;\; x \in M , y \in P \; .
\end{equation}
By the defining property \eqref{eq.defgammai} we have that $\gamma_i(\pi(x)) \recht \tau(x)$ for all $x \in M$. We also have $\|\gamma_i\| \leq \|\mu_i\|$ and hence $\sup_i \|\gamma_i\| < \infty$.

Since for all $u \in \cN_M(A)$ we have $\cE \circ \Ad(\pi(u)\theta(\ubar)) = \Ad(\lambda(u)\rho(\ubar)) \circ \cE$, we conclude that for all $u \in \cN_M(A)$ we have
$$\|\gamma_i \circ \Ad(\pi(u)\theta(\ubar)) - \gamma_i \| \leq \|\mu_i \circ \Ad (\lambda(u) \rho(\ubar)) - \mu_i \| \recht 0 \; .$$
A similar reasoning yields for all $a \in \cU(A)$ that
$$\|(\pi(a) \theta(\abar)) \cdot \gamma_i - \gamma_i \| \recht 0 \; .$$

Choose $\Theta \in \cN^*$ as a weak$^*$ limit point of the net $(\gamma_i)$. By construction
\begin{itemize}
\item $\Theta(\pi(x)) = \tau(x)$ for all $x \in M$,
\item $(\pi(a) \theta(\abar)) \cdot \Theta = \Theta$ for all $a \in \cU(A)$,
\item $\Theta \circ \Ad(\pi(u) \theta(\ubar)) = \Theta$ for all $u \in \cN_M(A)$.
\end{itemize}
Define the positive functional $\Psi \in \cN^*_+$ given by $\Psi := |\Theta|$. For all $u \in \cN_M(A)$ we have
$$|\Theta| \circ \Ad(\pi(u) \theta(\ubar)) = |\Theta \circ \Ad(\pi(u) \theta(\ubar))| = |\Theta| \; ,$$
meaning that $\Psi$ is $(\Ad(\pi(u) \theta(\ubar)))_{u \in \cN_M(A)}$-invariant.

For all $a \in \cU(A)$, we have
\begin{equation}\label{eq.stapje}
(\pi(a) \theta(\abar)) \cdot \Theta = \Theta \; .
\end{equation}
Take a partial isometry $V \in \cN^{**}$ such that $\Psi(x) = \Theta(V x)$ for all $x \in \cN$. Applying $V$ to the equality \eqref{eq.stapje}, we conclude that
$\Psi(\pi(a) \theta(\abar)) = \Psi(1)$ for all $a \in \cU(A)$.

We finally prove that the restriction of $\Psi$ to $\pi(M)$ is faithful. Let $p \in M$ be a nonzero projection. For every $x \in \cN$ we have $|\Theta(x)|^2 \leq \|\Theta\| \, \Psi(x^* x)$. So we get that
$$\tau(p)^2 = |\Theta(\pi(p))|^2 \leq \|\Theta\| \, \Psi(p) \; .$$
Hence $\Psi(p) > 0$.

Define the subgroups $\cG_1,\cG_2 \subset \cU(\cN)$ given by $\cG_1 := \{\pi(a)\theta(\abar) \mid a \in \cU(A)\}$ and $\cG_2 := \{\pi(u) \theta(\ubar) \mid u \in \cN_M(A)\}$. Observe that the unitaries in $\cG_2$ normalize $\pi(M)$ and implement on $\pi(M)$ an automorphism that is inner and hence preserves the trace $\tau$. Lemma \ref{lem.make-Om} provides us now with a state $\Om \in \cN^*_+$ such that
\begin{itemize}
\item $\Om(\pi(x)) = \tau(x)$ for all $x \in M$,
\item $\Om(\pi(a)\theta(\abar)) = 1$ for all $a \in \cU(A)$,
\item $\Om \circ \Ad(\pi(u) \theta(\ubar)) = \Om$ for all $u \in \cN_M(A)$.
\end{itemize}
Take a net of normal states $\om_i \in \cN_*$ such that $\om_i \recht \Om$ in the weak$^*$ topology. So $\om_i(\pi(x)) \recht \tau(x)$ for all $x \in M$ and $\om_i(\pi(a)\theta(\abar)) \recht 1$ for all $a \in \cU(A)$. Also, for all $u \in \cN_M(A)$ we have that
$$\om_i \circ \Ad(\pi(u) \theta(\ubar)) - \om_i \recht 0 \quad\text{weakly in $\cN_*$.}$$
After a passage to convex combinations, we find a net of normal states satisfying all the required conditions.
\end{proof}

\section{\boldmath Proof of Theorem \ref{thm.source}}\label{sec.proofsource}

By Lemma \ref{lem.reduction} it suffices to prove Theorem \ref{thm.source} for the trivial action of $\Gamma$ on $(B,\tau)$. Moreover, for notational convenience, we assume that the projection $q$ in the formulation of Theorem \ref{thm.source} equals $1$. In Remark \ref{rem.generalq} at the end of this section, we explain the necessary changes that are needed to deal with the general case. These changes are only cosmetic, but notationally cumbersome.

We fix a weakly amenable group $\Gamma$, a tracial von Neumann algebra $(B,\tau)$ and a $1$-cocycle $c : \Gamma \recht K_\R$ into the orthogonal representation $\eta : \Gamma \recht \cO(K_\R)$.
Write $M := B \ovt \rL(\Gamma)$ and fix a von Neumann subalgebra $A \subset M$ that is amenable relative to $B$. Denote by $P := \cN_M(A)\dpr$ its normalizer. We denote by $(u_g)_{g \in \Gamma}$ the canonical unitaries in $\rL(\Gamma)$.

As in Theorem \ref{thm.omn} we denote by $N$ the von Neumann algebra generated by $B$ and $P\op$ on the Hilbert space $\rL^2(M) \ot_A \rL^2(P)$. We always view $B$ and $P\op$ as commuting subalgebras of $N$ that together generate $N$. We fix a standard Hilbert space $H$ for $N$ and view $N$ as acting on $H$. This standard representation comes with the anti-unitary involution $J : H \recht H$.

We define $\cN := N \ovt \rL(\Gamma)$ and, as in Theorem \ref{thm.omn}, we consider the tautological embeddings
$$\pi : M \recht \cN : \pi(b \ot u_g) = b \ot u_g \quad\text{and}\quad \theta : P\op \recht \cN : \theta(y\op) = y\op \ot 1$$
for all $b \in B, g \in \Gamma, y \in P$. Clearly $\pi(M)$ commutes with $\theta(P\op)$ and together they generate $\cN$. Being the tensor product of $N$ and $\rL(\Gamma)$, the von Neumann algebra $\cN$ is standardly represented on $\cH := H \ot \ell^2(\Gamma)$ by the formula
$$(x \ot u_g) \cdot (\xi \ot \delta_h) = x \xi \ot \delta_{gh} \quad\text{for all}\;\; x \in N \; , \; g,h \in \Gamma \; , \; \xi \in H \; .$$
The corresponding anti-unitary involution $\cJ : \cH \recht \cH$ is given by $\cJ(\xi \ot \delta_g) = J \xi \ot \delta_{g^{-1}}$.

Take a net of normal states $\om_n \in \cN_*$ satisfying the conclusions of Theorem \ref{thm.omn}. Denote by $\xi_n \in \cH$ the canonical positive unit vectors that implement $\om_n$.
Whenever $u \in \cN_M(A)$ it follows from \cite[Theorem IX.1.2.(iii)]{Ta03} that the vector
$$\pi(u) \, \theta(\ubar) \, \cJ \pi(u) \, \theta(\ubar) \cJ \, \xi_n$$
is the canonical positive vector that implements $\om_n \circ \Ad(\pi(u^*)\theta(u\op))$. Using the Powers-St{\o}rmer inequality (see e.g.\ \cite[Theorem IX.1.2.(iv)]{Ta03}) the conclusion of Theorem \ref{thm.omn} can now be rewritten as follows in terms of the net $(\xi_n)$.
\begin{align}
& \langle \pi(x) \xi_n ,\xi_n \rangle = \om_n(\pi(x)) \recht \tau(x) \quad\text{for all}\;\; x \in M \; , \label{eq.onx}\\
& \| \pi(a) \theta(\abar) \xi_n - \xi_n\| \recht 0 \quad\text{for all}\;\; a \in \cU(A) \label{eq.ona}\; , \\
& \| \pi(u) \, \theta(\ubar) \, \cJ \pi(u) \, \theta(\ubar) \cJ \, \xi_n \; - \; \xi_n \| \recht 0 \quad\text{for all}\;\; u \in \cN_M(A) \; .\label{eq.onu}
\end{align}

To prove Theorem \ref{thm.source} we make use of the malleable deformation $(\al_t)$ of $\cN$ that was associated as follows in \cite{Si10} with the $1$-cocycle $c : \Gamma \recht K_\R$. We apply this malleable deformation $(\al_t)$ to the net $(\xi_n)$. With a proof that is very similar to \cite[Theorem 4.9]{OP07} we will reach the conclusion of Theorem \ref{thm.source}.

First apply the Gaussian construction to the real Hilbert space $K_\R$, yielding a tracial abelian von Neumann algebra $(D,\tau)$, generated by unitaries $\om(\xi)$, $\xi \in K_\R$, satisfying
$$\om(\xi + \xi') = \om(\xi) \, \om(\xi') \; , \quad \om(\xi)^* = \om(-\xi) \; , \quad \tau(\om(\xi)) = \exp(-\|\xi\|^2/2)$$
for all $\xi,\xi' \in K_\R$. The orthogonal representation $\eta : \Gamma \recht \cO(K_\R)$ yields a trace preserving action of $\Gamma$ on $D$, denoted by $(\si_g)_{g \in \Gamma}$, and given by $\si_g(\om(\xi)) = \om(\eta_g \xi)$ for all $g \in \Gamma$, $\xi \in K_\R$.

Denote $\cNtil := N \ovt (D \rtimes \Gamma)$ and view $\cN = N \ovt \rL(\Gamma)$ as a von Neumann subalgebra of $\cNtil$ in the natural way. We put $\Mtil := B \ovt (D \rtimes \Gamma)$ and extend the embedding $\pi : M \recht \cN$ to the still tautological embedding $\pi : \Mtil \recht \cNtil$ given by
$$\pi(b \ot du_g) = b \ot du_g \quad\text{for all}\quad b \in B \; , \; d \in D \; , \; g \in \Gamma \; .$$
We still have $\theta : P\op \recht \cNtil : \theta(y\op) = y\op \ot 1$ for all $y \in P$. We have that $\pi(\Mtil)$ commutes with $\theta(P\op)$ and together they generate $\cNtil$.

The $1$-cocycle $c : \Gamma \recht K_\R$ yields the malleable deformation $(\al_t)_{t \in \R}$ of \cite[Section 3]{Si10}, which is the one-parameter group of automorphisms of $\cNtil$ given by
\begin{equation}\label{eq.alphat}
\al_t( x \ot d u_g ) = x \ot d \om(t c(g)) u_g \quad\text{for all}\;\; x \in N \; , \;\; d \in D \; , \;\; g \in \Gamma \; , \;\; t \in \R \; .
\end{equation}
Note that $\al_t$ globally preserves the subalgebra $\pi(\Mtil) \subset \cNtil$. We also denote by $\al_t$ the corresponding deformation of $\Mtil$. Hence $\al_t \circ \pi = \pi \circ \al_t$. Repeating \eqref{eq.psit} we denote by
\begin{equation}\label{eq.psitbis}
\psi_t : M \recht M : \psi_t(b \ot u_g) = \exp(-t \|c(g)\|^2) (b \ot u_g) \quad\text{for all}\;\; b \in B, g \in \Gamma
\end{equation}
the $1$-parameter group of completely positive maps associated with the $1$-cocycle $c$. We note the crucial formula
$$\psi_{t^2/2}(x) = E_M(\al_t(x)) \quad\text{for all}\;\; x \in M , t \in \R \; .$$

Define $\cHtil := H \ot \rL^2(D) \ot \ell^2(\Gamma)$. Then $\cNtil$ is standardly represented on $\cHtil$ by
$$(x \ot d u_g) \cdot (\xi \ot d' \ot \delta_h) = x \xi \ot d \si_g(d') \ot \delta_{gh}$$
for all $x \in N$, $d, d' \in D$, $g,h \in \Gamma$ and $\xi \in H$. The corresponding anti-unitary involution $\cJtil : \cHtil \recht \cHtil$ is given by
$$\cJtil (\xi \ot d \ot \delta_g) = J \xi \ot \si_{g^{-1}}(d)^* \ot \delta_{g^{-1}}$$
for all $\xi \in H$, $d \in D$, $g \in \Gamma$.

For later use we record the following formulae.
\begin{equation}\label{eq.formrep}
\begin{split}
\pi(b \ot u_g) & \cdot (\xi \ot d \ot \delta_h) = b \xi \ot \si_g(d) \ot \delta_{gh} \; ,\\
\cJtil \pi(b \ot u_g) \cJtil  & \cdot (\xi \ot d \ot \delta_h) = JbJ \xi \ot d \ot \delta_{hg^{-1}} \; , \\
\theta(a\op) & \cdot (\xi \ot d \ot \delta_h) = a\op \xi \ot d \ot \delta_h \; , \\
\cJtil \theta(a\op)\cJtil & \cdot (\xi \ot d \ot \delta_h) = J a\op J \xi \ot d \ot \delta_h \; ,
\end{split}
\end{equation}
for all $b \in B$, $g,h \in \Gamma$, $d \in D$ and $\xi \in H$.

The canonical unitary implementation $(V_t)_{t \in \R}$ of the malleable deformation $(\al_t)_{t \in \R}$ of $\cNtil$ is given by
$$V_t (\xi \ot d \ot \delta_g) = \xi \ot d \, \om(t c(g)) \ot \delta_g$$
for all $\xi \in H$, $d \in D$, $g \in \Gamma$, and satisfies $\cJtil V_t = V_t \cJtil$ for all $t \in \R$.

Denote by $e : \cHtil \recht \cH$ the orthogonal projection, where we identified $\cH = H \ot \ell^2(\Gamma)$ with the subspace $H \ot \C 1 \ot \ell^2(\Gamma)$ of $\cHtil = H \ot \rL^2(D) \ot \ell^2(\Gamma)$. We write $e^\perp := 1 - e$.

We distinguish the following two cases which are each other's negation.

{\bf Case 1.} For every nonzero central projection $p \in \cZ(P)$ and for every $t > 0$ we have
$$\limsup_n \| e^\perp V_t \pi(p) \xi_n \| > \frac{\|p\|_2}{8} \; .$$

{\bf Case 2.} There exists a nonzero central projection $p \in \cZ(P)$ and a $t > 0$ such that
$$\limsup_n \| e^\perp V_t \pi(p) \xi_n \| \leq \frac{\|p\|_2}{8} \; .$$

Denote by $\gamma : \Gamma \recht \cU(\rL^2(D \ominus \C 1))$ the Koopman representation for $\Gamma \actson D \ominus \C 1$. Denote by $\cK^\gamma$ the associated $M$-$M$-bimodule on the Hilbert space $\cK^\gamma := \rL^2(D \ominus \C 1) \ot \rL^2(M)$ as in \eqref{eq.Krhobis}.

We first prove in Lemma \ref{lem.case1} below that in case 1, the $M$-$M$-bimodule $\cK^\gamma$ is left $P$-amenable and that this implies the left $P$-amenability of the $M$-$M$-bimodule $\cK^\eta$ associated with the original orthogonal representation $\eta : \Gamma \recht \cO(K_\R)$.
We next prove in Lemma \ref{lem.case2} below that in case 2 there exist $t,\delta > 0$ such that $\|\psi_t(a)\|_2 \geq \delta$ for all $a \in \cU(A)$.

\emph{So once both Lemmas \ref{lem.case1} and \ref{lem.case2} are proven, also Theorem \ref{thm.source} is proven.}

\begin{lemma}\label{lem.case1}
In case 1 the $M$-$M$-bimodule $\cK^\eta$ is left $P$-amenable.
\end{lemma}
\begin{proof}
Throughout the proof we write $\Ktil := \rL^2(D \ominus \C 1)$.

The main part of the proof consists in showing the left $P$-amenability of the $M$-$M$-bimodule $\cK^\gamma$.
From the definition of the $M$-$M$-bimodule $\cK^\gamma$ in \eqref{eq.Krhobis} we see that $\B(\cK^\gamma) \cap (M\op)'$ can be identified with $\B(\Ktil) \ovt M$ in such a way that the left $M$-module action on $\cK^\gamma$ corresponds to the embedding
$$\Delta_\gamma : M \recht \B(\Ktil) \ovt M : \Delta_{\gamma}(b \ot u_g) = \gamma(g) \ot b \ot u_g \; .$$
So to prove the left $P$-amenability of $\cK^\gamma$, we have to produce a $\Delta_{\gamma}(P)$-central state $\Om$ on $\B(\Ktil) \ovt M$ satisfying $\Om(\Delta_{\gamma}(x)) = \tau(x)$ for all $x \in M$.

Since $P$ is the normalizer of $A$ inside $M$, we have $P' \cap M = \cZ(P)$. We apply Lemma \ref{lem.make-Om} to the von Neumann algebra $\B(\Ktil) \ovt M$ with von Neumann subalgebra $\Delta_{\gamma}(M)$ and groups of unitaries $\cG_1 = \{1\}$, $\cG_2 = \Delta_\gamma(\cU(P))$. To prove the left $P$-amenability of $\cK^\gamma$, by Lemma \ref{lem.make-Om} it suffices to find for every nonzero central projection $p \in \cZ(P)$ a $\Delta_{\gamma}(P)$-central positive functional on $\B(\Ktil) \ovt M$ whose restriction to $\Delta_{\gamma}(M)$ is normal and nonzero on $\Delta_\gamma(p)$. Fix a nonzero central projection $p \in \cZ(P)$.

Consider the unitary operator
$$U : \Ktil \ot H \ot \ell^2(\Gamma) \recht \cHtil \ominus \cH = H \ot \rL^2(D \ominus \C 1) \ot \ell^2(\Gamma) : U(d \ot \xi \ot \delta_g) = \xi \ot d \ot \delta_g$$
for all $d \in D \ominus \C 1$, $\xi \in H$ and $g \in \Gamma$. Consider $\id \ot \pi : \B(\Ktil) \ovt M \recht \B(\Ktil) \ovt \cN$ and then define
$$\Psi : \B(\Ktil) \ovt M \recht \B(\cHtil \ominus \cH) : \Psi(S) = U (\id \ot \pi)(S) U^* \; .$$
For $x \in M$ we can view $\pi(x)$ as an element of $\cNtil$. As such $\pi(x)$ acts on $\cHtil \ominus \cH$ and with this point of view we have $\Psi(\Delta_{\gamma}(x)) = \pi(x)$ for all $x \in M$. Further note that
$$\Psi(\B(\Ktil) \ovt M) = B \, \ovt \, \B(\rL^2(D \ominus \C 1)) \, \ovt \, \{\lambda_g \mid g \in \Gamma\}\dpr \; .$$
Using formulae \eqref{eq.formrep} it follows that
$$\theta(P\op) \vee \cJtil \pi(M) \cJtil \vee \cJtil \theta(P\op) \cJtil = (P\op \vee JBJ \vee JP\op J) \,\ovt\, 1 \,\ovt\, \{\rho_g \mid g \in \Gamma\}\dpr \; .$$
Hence,
\begin{equation}\label{eq.commute}
\Psi(\B(\Ktil) \ovt M) \quad\text{commutes with}\quad \theta(P\op) \vee \cJtil \pi(M) \cJtil \vee \cJtil \theta(P\op) \cJtil \; .
\end{equation}

We claim that there exists a net of vectors $\mu_i \in \cHtil \ominus \cH$ such that $\|\mu_i\| \leq 1$ for all $i$ and
\begin{align}
& \lim_i \| \pi(u) \, \theta(\ubar) \, \cJtil \pi(u) \, \theta(\ubar) \cJtil \, \mu_i - \mu_i \| = 0 \quad\text{for all $u \in \cN_M(A)$,} \label{eq.1}\\
& \limsup_i \|\pi(x) \mu_i\| \leq \|x\|_2 \quad\text{for all $x \in M$,} \label{eq.2} \\
& \liminf_i \| \pi(p) \mu_i \| \geq \frac{\|p\|_2}{16} \; .\label{eq.3}
\end{align}
Once this claim is proven and after a passage to a subnet of $(\mu_i)$, we may assume that the net of positive functionals on $\B(\Ktil) \ovt M$ given by $S \mapsto \langle \Psi(S) \mu_i,\mu_i \rangle$ converges weakly$^*$ to a positive functional $\Om$ on $\B(\Ktil) \ovt M$.

We first prove that \eqref{eq.commute} and \eqref{eq.1} imply that $\Om \circ \Ad \Delta_\gamma(u) = \Om$ for all $u \in \cN_M(A)$. Fix $S \in \B(\Ktil) \ovt M$ and $u \in \cN_M(A)$. Since $\Psi(\Delta_\gamma(x)) = \pi(x)$ for all $x \in M$ and by \eqref{eq.1} and \eqref{eq.commute}, we get that
\begin{align*}
\Omega(\Delta_\gamma(u) \, S \, \Delta_\gamma(u)^*) &= \lim_i \langle \Psi(S) \, \pi(u)^* \mu_i , \pi(u)^* \mu_i \rangle \\
&= \lim_i \langle \Psi(S) \, \theta(\ubar) \, \cJtil \pi(u) \, \theta(\ubar) \cJtil \, \mu_i, \theta(\ubar) \, \cJtil \pi(u) \, \theta(\ubar) \cJtil \, \mu_i \rangle \\
&= \lim_i \langle \Psi(S) \, \mu_i , \mu_i \rangle = \Omega(S) \; .
\end{align*}

Since $\Psi(\Delta_\gamma(x)) = \pi(x)$ for all $x \in M$, the formulae \eqref{eq.2} and \eqref{eq.3} imply that $\Om(\Delta_{\gamma}(x)) \leq \tau(x)$ for all $x \in M^+$ and that $\Om(\Delta_{\gamma}(p)) \geq \tau(p) / 256$. In particular the restriction of $\Om$ to $\Delta_{\gamma}(M)$ is normal and nonzero on $\Delta_\gamma(p)$.

We finally show that $\Om$ is $\Delta_\gamma(P)$-central. Choose $x \in P$ and $S \in \B(\Ktil) \ovt M$ with $\|x\| \leq 1$ and $\|S\| \leq 1$. We need to prove that $\Om(\Delta_\gamma(x)S) = \Om(S\Delta_\gamma(x))$. To prove this formula, choose $\eps > 0$. Take a finite linear combination $y$ of unitaries $u \in \cN_M(A)$ such that $\|x-y\|_2 \leq \eps$. Since $\Om \circ \Ad \Delta_\gamma(u) = \Om$ for all $u \in \cN_M(A)$, we get that $\Om(\Delta_\gamma(y) S) = \Om(S \Delta_\gamma(y))$.
The Cauchy Schwarz inequality, the inequality $\Om(\Delta_\gamma(z)) \leq \tau(z)$ for all $z \in M^+$ and the choice of $\|S\| \leq 1$ imply that
\begin{align*}
|\Om(\Delta_\gamma(x) S) - \Om(\Delta_\gamma(y) S)|^2 & = |\Om(\Delta_\gamma(x-y) S)|^2 \\ & \leq \Om(\Delta_\gamma((x-y)(x-y)^*)) \, \Om(S^* S) \leq \|x-y\|_2^2 \leq \eps^2 \; .
\end{align*}
We similarly get that $|\Om(S \Delta_\gamma(x)) - \Om(S \Delta_\gamma(y))| \leq \eps$. So we have shown that
$$|\Om(\Delta_\gamma(x) S) - \Om(S \Delta_\gamma(x))| \leq 2 \eps$$
for all $\eps > 0$. Hence the required formula $\Om(\Delta_\gamma(x)S) = \Om(S\Delta_\gamma(x))$ follows and we have proven the $\Delta_\gamma(P)$-centrality of $\Om$. As observed in the first paragraph this concludes the proof of the left $P$-amenability of $\cK^\gamma$.

It remains to prove the claim above, i.e.\ the existence of a net of vectors $\mu_i \in \cHtil \ominus \cH$ satisfying $\|\mu_i\| \leq 1$ for all $i$ and satisfying \eqref{eq.1}, \eqref{eq.2} and \eqref{eq.3} above. Take finite subsets $\cF \subset \cN_M(A)$, $\cG \subset M$ and $\eps > 0$. It suffices to find a vector $\mu \in \cHtil \ominus \cH$ such that $\|\mu\| \leq 1$ and
\begin{align}
& \| \pi(u) \, \theta(\ubar) \, \cJtil \pi(u) \, \theta(\ubar) \cJtil \, \mu - \mu \| \leq 3\eps \quad\text{for all $u \in \cF$,} \label{eq.a}\\
& \|\pi(x) \mu\| \leq \|x\|_2 + \eps \quad\text{for all $x \in \cG$,}\label{eq.b}\\
& \| \pi(p) \mu \| \geq \frac{\|p\|_2}{16} - \eps \; .\label{eq.c}
\end{align}
We will find $\mu$ of the form $\mu := e^\perp V_t \pi(p) \xi_n$ by first choosing $t>0$ small enough and then choosing $n$ large enough.

Take $t > 0$ small enough such that
$$\|\al_{-t}(u) - u\|_2 \leq \eps \quad\text{for all $u \in \cF$ and}\quad \|\al_{-t}(p) - p \|_2 \leq \frac{\|p\|_2}{16} \; .$$
Define $\mu_n := e^\perp V_t \pi(p) \xi_n$. We prove that $\mu := \mu_n$ for certain $n$ large enough satisfies the conditions \eqref{eq.a}, \eqref{eq.b} and \eqref{eq.c} above.

The projection $e^\perp$ commutes with $\pi(M)$, $\theta(P\op)$ and with $\cJtil$. The unitary $V_t$ implements $\al_t$ on $\pi(M)$ and commutes with $\theta(P\op)$ and with $\cJtil$. So we get that
$$\pi(u) \, \theta(\ubar) \, \cJtil \pi(u) \, \theta(\ubar) \cJtil \, \mu_n = e^\perp V_t \; \theta(\ubar) \; \cJtil \theta(\ubar) \cJtil \; \pi(\al_{-t}(u)p) \; \cJtil \pi(\al_{-t}(u)) \cJtil \; \xi_n \; .$$
Since $\cJtil \xi_n = \xi_n$ and using \eqref{eq.onx} we have for all $u \in \cF$ that
$$\limsup_n \| \cJtil \pi(\al_{-t}(u)) \cJtil \, \xi_n - \cJtil \pi(u) \cJtil \, \xi_n \| = \|\al_{-t}(u) - u \|_2 \leq \eps \; .$$
We apply $\pi(\al_{-t}(u) p)$ and first observe that
$$\pi(\al_{-t}(u) p) \, \cJtil \pi(u) \cJtil \, \xi_n = \cJtil \pi(u) \cJtil \, \pi(\al_{-t}(u) p) \, \xi_n \; .$$
Again by \eqref{eq.onx} we have
$$\limsup_n \| \pi(\al_{-t}(u) p) \xi_n - \pi(up) \xi_n\| = \|\al_{-t}(u) p - u p\|_2 \leq \eps \; .$$
Altogether it follows that for all $u \in \cF$,
\begin{equation*}
\begin{split}
\limsup_n \| \pi(u) & \, \theta(\ubar) \, \cJtil \pi(u) \, \theta(\ubar) \cJtil \, \mu_n - \mu_n \| \\
& \leq 2\eps + \limsup_n \bigl\| \pi(p) \, \bigl( \pi(u) \, \theta(\ubar) \, \cJ \pi(u) \theta(\ubar) \cJ \, \xi_n - \xi_n \bigr) \bigr\| \; .
\end{split}
\end{equation*}
By \eqref{eq.onu} the $\limsup$ on the right hand side is $0$ and we conclude that \eqref{eq.a} holds for all $\mu := \mu_n$ with $n$ large enough.

Next observe that for all $x \in M$,
\begin{equation*}
\limsup_n \| \pi(x) \mu_n \| \leq \limsup_n \|\pi(\al_{-t}(x) p) \xi_n \| = \|\al_{-t}(x) p\|_2 \leq \|x\|_2 \; .
\end{equation*}
Hence also \eqref{eq.b} holds for all $\mu := \mu_n$ with $n$ large enough.

Finally, by the assumption of case 1 we know that $\limsup_n \|\mu_n\| \geq \|p\|_2 / 8$. Noticing that
$$\limsup_n \| \pi(p) \mu_n - \mu_n \| \leq \limsup_n \|\pi(\al_{-t}(p) p - p) \xi_n\| = \|\al_{-t}(p) p - p \|_2 \leq \frac{\|p\|_2}{16} \; ,$$
we conclude that
\begin{equation*}
\limsup_n \| \pi(p) \mu_n \| \geq \frac{\|p\|_2}{16} \; .
\end{equation*}
So \eqref{eq.c} holds for certain $\mu := \mu_n$ where $n$ can be chosen arbitrarily large.
Altogether there indeed exists an $n$ such that $\mu := \mu_n$ satisfies all the conditions \eqref{eq.a}, \eqref{eq.b} and \eqref{eq.c}.

So we have proven that $\cK^\gamma$ is a left $P$-amenable $M$-$M$-bimodule. It remains to prove that also $\cK^\eta$ is a left $P$-amenable $M$-$M$-bimodule.
Denote by $\epsilon$ the trivial representation of $\Gamma$ and define the unitary representation $\zeta$ of $\Gamma$ as the direct sum of $\epsilon$ and all tensor powers $\eta^{\otimes k}$, $k \geq 1$. The Koopman representation $\gamma : \Gamma \recht \cU(\rL^2(D \ominus \C 1))$ is isomorphic to the direct sum of all the $k$-fold ($k \geq 1$) symmetric tensor powers of $\eta$. Hence $\gamma$ is a subrepresentation of the tensor product representation $\eta \ot \zeta$. By Corollary \ref{cor.weakcontainment}, it follows that $\cK^{\eta \ot \zeta}$ also is a left $P$-amenable $M$-$M$-bimodule. But
$$\bim{M}{\cK^{\eta \ot \zeta}}{M} \; \cong \; \bim{M}{(\cK^\eta \ot_M \cK^\zeta)}{M} \; .$$
Condition~5 in Proposition \ref{prop.char} now implies the left $P$-amenability of $\bim{M}{\cK^\eta}{M}$.
\end{proof}

\begin{lemma}\label{lem.case2}
In case 2 there exist $t,\delta > 0$ such that $\|\psi_t(a)\|_2 \geq \delta$ for all $a \in \cU(A)$.
\end{lemma}

\begin{proof}
Fix a nonzero central projection $p \in \cZ(P)$ and a $t > 0$ such that
$$\limsup_n \| e^\perp V_t \pi(p) \xi_n \| \leq \frac{\|p\|_2}{8} \; .$$
A direct computation yields the transversality property of \cite[Lemma 2.1]{Po06b}~:
$$\|V_{\sqrt{2} t} \; \mu - \mu\| = \sqrt{2} \|e^\perp V_t \mu\| \quad\text{for all $\mu \in \cH \subset \cHtil$.}$$
Replacing $t$ by $\sqrt{2} t$, we have found a nonzero central projection $p \in \cZ(P)$ and a $t > 0$ such that
$$\limsup_n \| V_t \pi(p) \xi_n - \pi(p) \xi_n \| \leq \frac{\|p\|_2}{4} \; .$$

Recall from \eqref{eq.psitbis} the definition of the unital completely positive maps $\psi_t : M \recht M$. Also recall that $\psi_{s^2/2}(x) = E_M(\al_s(x))$ for all $x \in M$ and all $s \in \R$.
We prove that
\begin{equation}\label{eq.aim}
\|\psi_{t^2/2}(a)\|_2 \geq \frac{\|p\|_2}{2} \quad\text{for all}\;\; a \in \cU(A) \; .
\end{equation}
Once this inequality is proven, also the lemma is proven.

To prove \eqref{eq.aim} fix a unitary $a \in \cU(A)$. First notice that for all $\mu \in \cH \subset \cHtil$ and for all $x \in M$, we have
$$e \pi(\al_{-t}(x)) \mu = \pi(\psi_{t^2/2}(x)) \mu \; .$$
Using this formula we next prove that
\begin{equation}\label{eq.aim1}
\limsup_n | \langle \pi(a) \, \theta(\abar) \, V_t \pi(p) \xi_n , V_t \pi(p) \xi_n \rangle| \leq \|\psi_{t^2/2}(a)\|_2 \, \|p\|_2 \; .
\end{equation}
Indeed, since $V_t$ commutes with $\theta(\abar)$ and implements $\al_t$ on $\pi(M)$, we observe that
\begin{align*}
\langle \pi(a) \, \theta(\abar) \, V_t \pi(p) \xi_n , V_t \pi(p) \xi_n \rangle &= \langle \pi(\al_{-t}(a) p) \xi_n , \theta(a\op) \pi(p) \xi_n \rangle \\
&= \langle e \pi(\al_{-t}(a) p) \xi_n , \theta(a\op) \pi(p) \xi_n \rangle \\
&= \langle \pi(\psi_{t^2/2}(a) p) \xi_n , \theta(a\op) \pi(p) \xi_n \rangle
\end{align*}
Using \eqref{eq.onx} the $\limsup$ of the absolute value of the last expression is smaller or equal than
$$\limsup_n \|\pi(\psi_{t^2/2}(a) p) \xi_n \| \, \|\pi(p) \xi_n\| = \|\psi_{t^2/2}(a) p\|_2 \, \|p\|_2 \leq \|\psi_{t^2/2}(a)\|_2 \, \|p\|_2 \; .$$
So \eqref{eq.aim1} is proven.

Secondly, the fact that $\limsup_n \|V_t \pi(p) \xi_n - \pi(p) \xi_n \| \leq \|p\|_2 / 4$, while $\limsup_n \|V_t \pi(p) \xi_n\|_2 = \|p\|_2$, implies that
\begin{equation*}
\limsup_n \bigl| \langle \pi(a) \, \theta(\abar) \, V_t \pi(p) \xi_n , V_t \pi(p) \xi_n \rangle - \langle \pi(a) \, \theta(\abar) \, \pi(p) \xi_n , \pi(p) \xi_n \rangle \bigr| \leq \tau(p) / 2 \; .
\end{equation*}
Since moreover by \eqref{eq.onx} and \eqref{eq.ona} we have
$$\langle \pi(a) \, \theta(\abar) \, \pi(p) \xi_n , \pi(p) \xi_n \rangle \recht \tau(p) \; ,$$
we conclude that
$$\liminf_n \bigl| \langle \pi(a) \, \theta(\abar) \, V_t \pi(p) \xi_n , V_t \pi(p) \xi_n \rangle \bigr| \geq \tau(p) / 2 \; .$$
In combination with \eqref{eq.aim1} we find \eqref{eq.aim} and this ends the proof of the lemma.
\end{proof}

\begin{remark}\label{rem.generalq}
Above we only proved Theorem \ref{thm.source} in the special case where the projection $q$ in the formulation of the theorem equals $1$. Assume now that $q$ is an arbitrary nonzero projection and that $A \subset qMq$ is a von Neumann subalgebra that is amenable relative to $B$. Lemma \ref{lem.reduction} was proven for arbitrary $q$ so that we can still assume that $\Gamma$ acts trivially on $(B,\tau)$.
Denote by $P := \cN_{qMq}(A)\dpr$ the normalizer of $A$ inside $qMq$. Define $N$ as the von Neumann algebra generated by $B$ and $P\op$ on the Hilbert space $\rL^2(M)q \ot_A \rL^2(P)$. Put $\cN := N \ovt \rL(\Gamma)$ and define the tautological embeddings
$$\pi : M \recht \cN : \pi(b \ot u_g) = b \ot u_g \quad\text{and}\quad \theta : P\op \recht \cN : \theta(y\op) = y\op \ot 1$$
for all $b \in B, g \in \Gamma, y \in P$.

With literally the same proof as the one of Theorem \ref{thm.omn}, we find a net of normal positive functionals $\om_i \in (\pi(q)\cN \pi(q))_*$ satisfying the following properties.
\begin{itemize}
\item $\om_i(\pi(x)) \recht \tau(x)$ for all $x \in qMq$,
\item $\om_i(\pi(a) \theta(\abar)) \recht 1$ for all $a \in \cU(A)$,
\item $\|\om_i \circ \Ad(\pi(u) \theta(\ubar)) - \om_i\| \recht 0$ for all $u \in \cN_{qMq}(A)$.
\end{itemize}
Again we take the canonical implementation of the functionals $\om_i$ by positive vectors $(\xi_i)$ in a standard Hilbert space for $\cN$. We proceed with these vectors in exactly the same way as above.
\end{remark}

\section{Proof of Theorem \ref{thm.Crigidgroups}}

Using \cite[Theorem A.1]{Po01} Theorem \ref{thm.Crigidgroups} is an immediate consequence of the following result.

\begin{theorem}
Let $\Gamma$ be any of the groups in the formulation of Theorem \ref{thm.Crigidgroups}. Take an arbitrary trace preserving action $\Gamma \actson (B,\tau)$ and put $M = B \rtimes \Gamma$. Assume that $q \in M$ is a projection and that $A \subset qMq$ is a von Neumann subalgebra that is amenable relative to $B$ and whose normalizer $P:=\cN_{qMq}(A)\dpr$ has finite index in $qMq$. Then $A \prec_M B$.
\end{theorem}
\begin{proof}
Whenever $\eta : \Gamma \recht \cO(K_\R)$ is an orthogonal representation, we consider its complexification $\eta : \Gamma \recht \cU(K)$ and the corresponding $M$-$M$-bimodule $\cK^\eta$ given by \eqref{eq.Krhobis}. Whenever $c : \Gamma \recht K_\R$ is a $1$-cocycle into $\eta$, we consider the $1$-parameter family of completely positive maps $(\psi_t)_{t > 0}$ on $M$ given by \eqref{eq.psit}.

We first prove that if $\eta : \Gamma \recht \cU(K)$ is a unitary representation such that the $P$-$M$-bimodule $q \cK^\eta$ is left $P$-amenable, then $\eta$ is an amenable representation.

So assume that $q \cK^\eta$ is a left $P$-amenable $P$-$M$-bimodule. Since $P \subset qMq$ has finite index, it follows from Corollary \ref{cor.passage-rel-amen} that $q \cK^\eta$ is also left $qMq$-amenable. Defining
$$\Delta_\eta : M \recht \B(K) \ovt M : \Delta_\eta(b u_g) = \eta_g \ot bu_g \quad\text{for all}\;\; b \in B , g \in \Gamma \; ,$$
the left $qMq$-amenability of $q \cK^\eta$ precisely amounts to the existence of a positive functional $\Om$ on $\B(K) \ovt M$ with the following properties.
\begin{itemize}
\item $\Om(1 - \Delta_\eta(q)) = 0$ and $\Om(\Delta_\eta(x)) = \tau(x)$ for all $x \in qMq$.
\item $\Om(S \Delta_\eta(x)) = \Om(\Delta_\eta(x) S)$ for all $S \in \B(K) \ovt M$ and all $x \in qMq$.
\end{itemize}
Choose partial isometries $v_1,\ldots,v_n \in M$ such that $v_i^* v_i \leq q$ for all $i$ and such that $\sum_{i=1}^n v_i v_i^*$ is a nonzero central projection $z \in \cZ(M)$. Define the positive functional $\Omtil$ on $\B(K) \ovt M$ by the formula
$$\Omtil(S) := \sum_{i=1}^n \Om(\Delta_\eta(v_i^*) S \Delta_\eta(v_i)) \quad\text{for all}\;\; S \in \B(K) \ovt M \; .$$
A direct computation yields $\Omtil(\Delta_\eta(x)) = \tau(x)$ for all $x \in Mz$ and that $\Omtil(1 - \Delta_\eta(z)) = 0$.

We now prove that $\Omtil(S \Delta_\eta(x)) = \Omtil(\Delta_\eta(x) S)$ for all $S \in \B(K) \ovt M$ and all $x \in M$. Since $z$ is central, we have $x v_i = z x v_i$ and $v_j^* x z = v_j^* x$ for all $i,j$. Also observe that $v_j^* x v_i \in qMq$ for all $x \in M$ and all $i,j$. So we get that
\begin{align*}
\Omtil(S \Delta_\eta(x)) &= \sum_{i=1}^n \Om(\Delta_\eta(v_i^*) S \Delta_\eta(x v_i)) = \sum_{i=1}^n \Om(\Delta_\eta(v_i^*) S \Delta_\eta(z x v_i)) \\
&= \sum_{i,j = 1}^n \Om(\Delta_\eta(v_i^*) S \Delta_\eta(v_j) \Delta_\eta(v_j^* x v_i)) \\
&= \sum_{i,j = 1}^n \Om(\Delta_\eta(v_j^* x v_i v_i^*) S \Delta_\eta(v_j)) \\
&= \sum_{j=1}^n \Om(\Delta_\eta(v_j^* x z) S \Delta_\eta(v_j)) = \sum_{j=1}^n \Om(\Delta_\eta(v_j^* x ) S \Delta_\eta(v_j)) \\
&= \Omtil(\Delta_\eta(x) S) \; .
\end{align*}

Define the state $\Psi$ on $\B(K)$ by the formula $\Psi(S) = \Omtil(1)^{-1} \Omtil(S \ot 1)$. The following computation shows that $\Psi$ is $(\Ad \eta_g)_{g \in \Gamma}$-invariant and hence that $\eta$ is an amenable representation.
\begin{align*}
\Omtil(1) \, \Psi(S \eta_g) &= \Omtil(S \eta_g \ot 1) = \Omtil\bigl((S \ot u_g^*) \Delta_\eta(u_g) \bigr) \\
&= \Omtil\bigl( \Delta_\eta(u_g) (S \ot u_g^*) \bigr) = \Omtil(\eta_g S \ot 1) = \Omtil(1) \, \Psi(\eta_g S) \; .
\end{align*}

We are now ready to prove that for both families of groups $\Gamma$ in the formulation of Theorem \ref{thm.Crigidgroups}, we get that $A \prec_M B$. If $\beta_1^{(2)}(\Gamma) > 0$, we know that $\Gamma$ is nonamenable and that $\Gamma$ admits an unbounded $1$-cocycle $c$ into a multiple of the regular representation. The regular representation is mixing and is nonamenable by the nonamenability of $\Gamma$. So to cover the first family of groups in Theorem \ref{thm.Crigidgroups} it suffices to consider a weakly amenable group $\Gamma$ that admits an unbounded $1$-cocycle $c : \Gamma \recht K_\R$ into a nonamenable mixing representation $\eta : \Gamma \recht \cO(K_\R)$. From the discussion above we know that the $P$-$M$-bimodule $q \cK^\eta$ is not left $P$-amenable. So, from Theorem \ref{thm.source} we get $t,\delta > 0$ such that $\|\psi_t(a)\|_2 \geq \delta$ for all $a \in \cU(A)$.

As in \eqref{eq.alphat}, we consider the malleable deformation $(\al_t)$ of the tracial von Neumann algebra $\Mtil := (B \ovt D) \rtimes \Gamma$, where $\Gamma \actson (D,\tau)$ is the Gaussian action corresponding to $\eta : \Gamma \recht \cO(K_\R)$ and where $\Gamma \actson B \ovt D$ diagonally. Since $\eta$ is mixing and $\|\psi_t(a)\|_2 \geq \delta$ for all $a \in \cU(A)$, we get from \cite[Proposition 3.9]{Va10b} a nonzero central projection $p \in \cZ(P)$ such that $\al_t \recht \id$ uniformly in $\|\,\cdot\,\|_2$ on the unit ball of $A p$. If $A \not\prec_M B$, it follows from\footnote{We refer here to \cite{Va10b} where the notation and formulation is exactly suited for our purposes. Note however that the quoted result is due to Peterson \cite[Theorem 4.5]{Pe06} and Chifan-Peterson \cite[Theorem 2.5]{CP10}.} \cite[Theorem 3.10]{Va10b} that $\al_t \recht \id$ uniformly in $\|\,\cdot\,\|_2$ on the unit ball of $P p$. Since $P \subset qMq$ has finite index, also $P p \subset p M p$ has finite index. Using a Pimsner-Popa basis\footnote{See \cite[Proposition 1.3]{PP84} and see \cite[Proposition A.2]{Va07} for a nonfactorial version that can be readily applied here.}, it follows that $\al_t \recht \id$ uniformly in $\|\,\cdot\,\|_2$ on the unit ball of $pMp$. Denoting by $z \in \cZ(M)$ the central support of $p$, it follows that $\al_t \recht \id$ uniformly in $\|\,\cdot\,\|_2$ on the unit ball of $Mz$. This means that also $\psi_t \recht \id$ uniformly in $\|\,\cdot\,\|_2$ on the unit ball of $Mz$. If $t \recht 0$, we know that $\|\psi_t(xz) - \psi_t(x) z\|_2$ is small uniformly in $x$ belonging to the unit ball of $M$. So we can fix a $t > 0$ such that
$$\|\psi_t(x) z\|_2 \geq \|z\|_2 / 2 \quad\text{for all}\;\; x \in \cU(M) \; .$$
Since $c : \Gamma \recht K_\R$ is unbounded, we can take a sequence $g_n \in \Gamma$ such that $\|c(g_n)\| \recht \infty$. It follows that $\|\psi_t(u_{g_n})\|_2 \recht 0$ as $n \recht \infty$. Hence also $\|\psi_t(u_{g_n}) z\|_2 \recht 0$, contradicting the previous estimate. So we have shown that actually $A \prec_M B$.

Next consider the case where $\Gamma$ is a weakly amenable group that admits a proper $1$-cocycle $c : \Gamma \recht K_\R$ into a nonamenable representation $\eta : \Gamma \recht \cO(K_\R)$. From the first paragraphs of the proof we know that the $P$-$M$-bimodule $q \cK^\eta$ is not left $P$-amenable. So, from Theorem \ref{thm.source} we get $t,\delta > 0$ such that $\|\psi_t(a)\|_2 \geq \delta$ for all $a \in \cU(A)$. Whenever $x \in M$, we denote by
\begin{equation}\label{eq.fourier}
x = \sum_{g \in \Gamma} x_g u_g \quad\text{with}\;\; x_g \in B \;\;\text{for all}\;\; g \in \Gamma
\end{equation}
the Fourier decomposition of $x$. A direct computation yields
\begin{equation}\label{eq.concreteform}
\|\psi_t(x)\|_2^2 = \sum_{g \in \Gamma} \exp(-2t \|c(g)\|^2) \, \|x_g\|_2^2
\end{equation}
for all $x \in M$, $t > 0$.

If $A \not\prec_M B$, Definition \ref{def.intertwine} yields a sequence of unitaries $a_k \in \cU(A)$ such that for every fixed $g \in \Gamma$, the sequence of $g$'th Fourier coefficients $(a_k)_g \in B$, defined by \eqref{eq.fourier}, satisfies $\lim_k \|(a_k)_g\|_2 = 0$. The properness of the $1$-cocycle $c : \Gamma \recht K_\R$, together with formula \eqref{eq.concreteform}, implies that $\lim_k \|\psi_t(a_k)\|_2 = 0$. This is a contradiction with the property that $\|\psi_t(a_k)\|_2 \geq \delta$ for all $k$. So we also get $A \prec_M B$ when $\Gamma$ belongs to the second family of groups in Theorem \ref{thm.Crigidgroups}.

To finally conclude that $A \prec^f_M B$, observe that \cite[Proposition 2.5]{Va10b} provides a projection $q_0 \in \cZ(P)$ such that $Aq_0 \prec^f_M B$ and $A(q-q_0) \not\prec_M B$. Applying the above to the subalgebra $A(q-q_0) \subset (q-q_0) M (q-q_0)$ implies that $q-q_0 = 0$.
\end{proof}

\section{Proof of Theorem \ref{thm.main}}

Take $M = B \rtimes \Gamma$ as in the formulation of Theorem \ref{thm.main}. Let $A \subset M$ be a von Neumann subalgebra that is amenable relative to $B$ and denote by $P := \cN_M(A)\dpr$ its normalizer.

By our assumptions $\Gamma$ is weakly amenable and we have a proper $1$-cocycle $c : \Gamma \recht K_\R$ into an orthogonal representation $\eta : \Gamma \recht \cO(K_\R)$ that is weakly contained in the regular representation. We consider the $M$-$M$-bimodule $\cK^\eta$ associated with $\eta$ as in \eqref{eq.Krhobis} and we consider the $1$-parameter group $(\psi_t)_{t > 0}$ of completely positive maps on $M$ associated with the $1$-cocycle $c$ as in \eqref{eq.psit}. Theorem \ref{thm.source} says that
\begin{itemize}
\item either the $M$-$M$-bimodule $\cK^\eta$ is left $P$-amenable,
\item or there exist $t,\delta > 0$ such that $\|\psi_t(a)\|_2 \geq \delta$ for all $a \in \cU(A)$.
\end{itemize}

First assume that $\cK^\eta$ is a left $P$-amenable $M$-$M$-bimodule. Since $\eta$ is weakly contained in the regular representation $\lambda$, it follows that $\cK^\eta$ is weakly contained in $\cK^\lambda$ as $M$-$M$-bimodules. Corollary \ref{cor.weakcontainment} then implies that $\cK^\lambda$ is a left $P$-amenable $M$-$M$-bimodule.
As an $M$-$M$-bimodule $\cK^\lambda$ is isomorphic with the $M$-$M$-bimodule $\rL^2(M) \ot_B \rL^2(M)$. So $\bim{M}{(\rL^2(M) \ot_B \rL^2(M))}{M}$ is left $P$-amenable. By condition~5 in Proposition \ref{prop.char} also $\bim{M}{\rL^2(M)}{B}$ is left $P$-amenable. This means exactly that $P$ is amenable relative to $B$.

Finally assume that we have $t,\delta > 0$ such that $\|\psi_t(a)\|_2 \geq \delta$ for all $a \in \cU(A)$. We repeat a paragraph from the proof of Theorem \ref{thm.Crigidgroups}, using the Fourier decomposition of $x \in M$ as in \eqref{eq.fourier}. If $A \not\prec_M B$, Definition \ref{def.intertwine} yields a sequence of unitaries $a_k \in \cU(A)$ such that for every fixed $g \in \Gamma$ we have that $\lim_k \|(a_k)_g\|_2 = 0$. The properness of the $1$-cocycle $c : \Gamma \recht K_\R$, together with formula \eqref{eq.concreteform}, implies that $\lim_k \|\psi_t(a_k)\|_2 = 0$. This is a contradiction with the property that $\|\psi_t(a_k)\|_2 \geq \delta$ for all $k$. So, $A \prec_M B$ and the theorem is proven.\hfill\qedsymbol

\section{Proof of Theorem \ref{thm.main-products}}

Using e.g.\ \cite[Proposition 2.5]{Va10b}, we find projections $p_i \in \cZ(P)$ such that $A p_i \prec^f_M B \rtimes \Gammah_i$ and $A (1-p_i) \not\prec_M B \rtimes \Gammah_i$ for all $i=1,\ldots,n$. Of course, some or even all of the $p_i$ could be zero. Define $p_0 := 1 - (p_1 \vee \cdots \vee p_n)$. We consider the subalgebra $A p_0 \subset p_0 M p_0$, whose normalizer is given by $P p_0$. We need to prove that $P p_0$ is amenable relative to $B$.

By construction, for every $i$ we have that $A p_0 \not\prec_M B \rtimes \Gammah_i$. Viewing $M$ as the crossed product $M = (B \rtimes \Gammah_i) \rtimes \Gamma_i$, it then follows from Theorem \ref{thm.main} that $P p_0$ is amenable relative to $B \rtimes \Gammah_i$ for every $i=1,\ldots,n$.

All the subalgebras $B \rtimes \Gammah_i \subset M$ are regular and all the crossed products of $B$ by a certain number of the $\Gamma_i$'s are in commuting square position with respect to each other. So by Proposition \ref{prop.intersection}, we conclude that $Pp_0$ is amenable relative to $B$.\hfill\qedsymbol

\section{Proof of Theorem \ref{thm.freeproducts}}

Let $\Gamma = \Lambda_1 * \Lambda_2$ be any weakly amenable free product group and consider $M = B \rtimes \Gamma$ as in the formulation of the theorem. Let $A \subset M$ be a von Neumann subalgebra that is amenable relative to $B$. Denote by $P := \cN_M(A)\dpr$ its normalizer. Using e.g.\ \cite[Proposition 2.5]{Va10b}, we can take projections $q, p_1, p_2 \in \cZ(P)$ such that
\begin{itemize}
\item $A q \prec^f_M B$ and $A (1-q) \not\prec_M B$,
\item $P p_i \prec^f_M B \rtimes \Lambda_i$ and $P(1-p_i) \not\prec_M B \rtimes \Lambda_i$ for all $i = 1,2$.
\end{itemize}
As above, some or all of the $q,p_1,p_2$ might be zero. Denote $p_0 = 1 - (q \vee p_1 \vee p_2)$. We have to prove that $P p_0$ is amenable relative to $B$.

For $g \in \Gamma$ denote by $|g|$ the length of $g$, i.e.\ the number of elements needed to write $g$ as an alternating product of elements in $\Lambda_1 - \{e\}$ and $\Lambda_2-\{e\}$. Consider the direct sum $K_\R := \ell^2_\R(\Gamma) \oplus \ell^2_\R(\Gamma)$ of two copies of the regular representation of $\Gamma$ and denote this orthogonal representation as $\eta$. Define the unique $1$-cocycle $c : \Gamma \recht K_\R$ satisfying
$$c(g) = (\delta_g - \delta_e,0) \quad\text{for all}\;\; g \in \Lambda_1 \quad\text{and}\quad c(h) = (0,\delta_h - \delta_e) \quad\text{for all}\;\; h \in \Lambda_2 \; .$$
One computes easily that $\|c(g)\|^2 = 2 |g|$ for all $g \in \Gamma$.

We denote by $\cK^\eta$ the $M$-$M$-bimodule associated with $\eta$ as in \eqref{eq.Krhobis}. We consider the $1$-parameter group $(\psi_t)_{t > 0}$ of completely positive maps on $M$ associated with the $1$-cocycle $c$ as in \eqref{eq.psit}. We apply Theorem \ref{thm.source} to the subalgebra $A p_0 \subset p_0 M p_0$. Note that the normalizer of $A p_0$ inside $p_0 M p_0$ is precisely $P p_0$. So by Theorem \ref{thm.source}, either $p_0 \cK^\eta$ is a left $Pp_0$-amenable $p_0 M p_0$-$M$-bimodule, or there exist $t,\delta > 0$ such that $\|\psi_t(a)\|_2 \geq \delta$ for all $a \in \cU(A p_0)$.

Because by construction $A p_0 \not\prec_M B$ and $P p_0 \not\prec_M B \rtimes \Lambda_i$ for all $i = 1,2$, it follows from one of the main results in \cite{IPP05} (and actually by literally applying the version that we presented as \cite[Theorem 5.4]{PV09}) that it is impossible to have $\|\psi_t(a)\|_2 \geq \delta$ for all $a \in \cU(A p_0)$. So $p_0 \cK^\eta$ is a left $Pp_0$-amenable $p_0 M p_0$-$M$-bimodule. Since $\eta$ is a multiple of the regular representation, this implies in the same way as in the proof of Theorem \ref{thm.main}, that $P p_0$ is amenable relative to $B$.

\section{Stability under measure equivalence subgroups}

Consider the following strengthening of $\cCs$-rigidity involving measure preserving actions on potentially infinite measure spaces.

\begin{definition}\label{def.propstar}
We say that a countable group $\Gamma$ has property ($\ast$) if the following holds: for every measure preserving action $\Gamma \actson (X,\mu)$ on a standard, possibly infinite, measure space $(X,\mu)$ and for every abelian von Neumann subalgebra $A \subset qMq$ where $M = \rL^\infty(X) \rtimes \Gamma$ and $q \in \rL^\infty(X)$ a projection of finite measure, we have the dichotomy that either $A \prec_{qMq} \rL^\infty(X)q$ or that the normalizer $\cN_{qMq}(A)\dpr$ is amenable.
\end{definition}

Obviously every nonamenable group $\Gamma$ satisfying property ($\ast$) is $\cCs$-rigid.

We first prove that any weakly amenable group $\Gamma$ that admits a proper $1$-cocycle into an orthogonal representation that is weakly contained in the regular representation, has property ($\ast$). Then we will show that property ($\ast$) is preserved under the passage to measure equivalence subgroups (ME-subgroups). Also weak amenability is stable under the passage to ME-subgroups. Interestingly enough, it is not known whether having a proper $1$-cocycle into an orthogonal representation that is weakly contained in the regular representation, is stable under ME-subgroups (or even under measure equivalence). To prove such a stability result one needs an integrability condition on the associated orbit equivalence cocycle (cf.\ \cite[Theorem 5.10]{Th08}).

Recall that a countable group $\Lambda$ is said to be an \emph{ME-subgroup} of a countable group $\Gamma$ if $\Gamma \times \Lambda$ admits a measure preserving action on a, typically infinite, standard measure space $(\Omega,m)$ such that both the actions $\Gamma \actson \Om$ and $\Lambda \actson \Om$ are free and admit a fundamental domain, with the fundamental domain of $\Gamma \actson \Om$ having finite measure. If the actions can be chosen in such a way that also the fundamental domain of $\Lambda \actson \Om$ has finite measure, the groups $\Gamma$ and $\Lambda$ are called \emph{measure equivalent}.

\begin{theorem}
Let $\Gamma$ be a weakly amenable group that admits a proper $1$-cocycle into an orthogonal representation that is weakly contained in the regular representation. Then $\Gamma$ has property ($\ast$) in the sense of Definition \ref{def.propstar}.
\end{theorem}
\begin{proof}
Choose a measure preserving action $\Gamma \actson (X,\mu)$. Put $B = \rL^\infty(X)$ and let $q \in B$ be a projection of finite measure. Put $M = B \rtimes \Gamma$ and let $A \subset q M q$ be an abelian von Neumann subalgebra. Denote by $P:= \cN_{qMq}(A)\dpr$ the normalizer of $A$ inside $qMq$. Define the normal $*$-homomorphism
$$\Delta : M \recht M \ovt \rL(\Gamma) : \Delta(b u_g) = b u_g \ot u_g \quad\text{for all}\;\; b \in B, g \in \Gamma \; .$$
So $\Delta(A)$ is an abelian von Neumann subalgebra of $qMq \ovt \rL(\Gamma)$. Since $qMq$ has a finite trace, we can apply Theorem \ref{thm.main} with $B = qMq$ and $\Gamma \actson B$ the trivial action. This means that either $\Delta(A) \prec_{qMq \ovt \rL(\Gamma)} qMq \ot 1$ or that $\Delta(P)$ is amenable relative to $qMq \ot 1$. With exactly the same argument as in the proof of Lemma \ref{lem.reduction}, it follows that either $A \prec_{qMq} Bq$ or that $P$ is amenable relative to $Bq$, which implies that $P$ is plainly amenable.
\end{proof}

\begin{proposition}
If $\Gamma$ is a countable group satisfying property ($\ast$), then also all ME-subgroups of $\Gamma$ satisfy property ($\ast$).
\end{proposition}
\begin{proof}
{\bf Part 1~:} to establish property ($\ast$), it suffices to consider \emph{free} measure preserving actions $\Gamma \actson (X,\mu)$. Indeed, assume that property ($\ast$) holds for all \emph{free} measure preserving actions of $\Gamma$ and let $\Gamma \actson (X,\mu)$ be \emph{any} measure preserving action. Put $M = \rL^\infty(X) \rtimes \Gamma$. Assume that $q \in \rL^\infty(X)$ is a projection of finite measure and that $A \subset q M q$ is an abelian von Neumann subalgebra. We have to prove that either $A \prec_{qMq} \rL^\infty(X) q$ or that $\cN_{qMq}(A)\dpr$ is amenable.

Let $\Gamma \actson Y$ be any free pmp action, e.g.\ a Bernoulli action. Then the diagonal action $\Gamma \actson Y \times X$ is free. Put $\Mtil = \rL^\infty(Y \times X) \rtimes \Gamma$ and view $M \subset \Mtil$ in the obvious way. Then $\qtil = 1 \ot q$ is a projection of finite measure and we can view $A$ as a subalgebra of $\qtil \Mtil \qtil$. Since property ($\ast$) holds for free actions we have that either $A \prec_{\qtil \Mtil \qtil} \rL^\infty(Y \times X)\qtil$ or that $\cN_{\qtil \Mtil \qtil}(A)\dpr$ is amenable. In the first case, it follows that $A \prec_{q M q} \rL^\infty(X)q$, while in the second case also the subalgebra $\cN_{qMq}(A)\dpr$ of $\cN_{\qtil \Mtil \qtil}(A)\dpr$ is amenable. This ends the proof of part 1.

{\bf Part 2~:} if $\Gamma$ has property ($\ast$), then $\Gamma \times G$ has property ($\ast$) for every finite group $G$. By part 1, it suffices to consider free measure preserving actions $\Gamma \times G \actson Y$. Then $\rL^\infty(Y) \rtimes (\Gamma \times G)$ is isomorphic with $(\rL^\infty(X) \rtimes \Gamma)^n$ where $n = |G|$, $X = Y / G$ and where we use the notation $Q^n := \M_n(\C) \ot Q$. Moreover, under this isomorphism $\rL^\infty(Y)$ corresponds to $\D_n(\C) \ot \rL^\infty(X)$, where $\D_n(\C) \subset \M_n(\C)$ denotes the subalgebra of diagonal matrices. So take a free measure preserving action $\Gamma \actson (X,\mu)$, write $B = \rL^\infty(X)$ and take an integer $n$ and a projection of finite measure $q \in \D_n(\C) \ot B$. Write $M := B \rtimes \Gamma$ and assume that $A \subset q M^n q$ is an abelian von Neumann subalgebra. Assume that $A \not\prec_{qM^n q} qB^n q$. Denote $P := \cN_{q M^n q}(A)\dpr$. We must prove that $P$ is amenable.

Denote by $D \subset \rL^\infty(X)$ the subalgebra of $\Gamma$-invariant functions. Since $\Gamma \actson X$ is free, we have that $D = \cZ(M)$ and $(1 \ot D)q = \cZ(q M^n q)$. Denote $\Atil = A \vee (1 \ot D)q$. Obviously $\Atil$ is abelian and $\Atil \not\prec_{qM^n q} qB^n q$. Denote $\Ptil := \cN_{q M^n q}(\Atil)\dpr$.
Every unitary $u \in \cU(q M^n q)$ that normalizes $A$, commutes with $(1 \ot D)q$ and hence, also normalizes $\Atil$. So $P \subset \Ptil$.

Since $q \in \D_n(\C) \ot B$, we write $q = \sum_{i=1}^n e_{ii} \ot q_i$, where $q_i \in B$ are projections of finite measure. We claim that there exist orthogonal projections $p_i \in \Atil$, with sum $q$, such that, inside $q M^n q$, the projections $p_i$ and $e_{ii} \ot q_i$ are equivalent for all $i=1,\ldots,n$. To prove this claim, it suffices to show that ``$\Atil$ is diffuse over the center $(1 \ot D)q$'', i.e.\ it suffices to show that there is no nonzero projection $p \in \Atil$ such that $\Atil p = (1 \ot D) p$. This follows immediately since $(1 \ot D)q \subset q B^n q$ and since we assumed that $\Atil \not\prec q B^n q$.

By the claim in the previous paragraph, we can take partial isometries $v_1,\ldots,v_n \in \M_{1,n}(\C) \ot M$ such that $v_i v_i^* = q_i$ and such that $v_i^* v_i = p_i$ where the $p_i$ are orthogonal projections in $\Atil$ with sum $q$. Define $A_i := v_i \Atil v_i^*$ and $P_i := v_i \Ptil v_i^*$. By \cite[Lemma 3.5]{Po03}, $P_i$ is the normalizer of $A_i$ inside $q_i M q_i$. Since $\Atil \not\prec q B^n q$, we also have that $A_i \not\prec_{q_i M q_i} B q_i$. Since property ($\ast$) holds for $\Gamma$, it follows that $P_i$ is amenable for every $i$. Hence, $p_i \Ptil p_i$ is amenable for every $i$. Since $\sum_{i=1}^n p_i = q$ and $q$ is the unit of $\Ptil$, it follows that $\Ptil$ is amenable. Because $P \subset \Ptil$, this concludes the proof of part 2.

{\bf Part 3~:} property ($\ast$) is stable under ME-subgroups. Assume that $\Gamma$ satisfies property ($\ast$) and that $\Lambda$ is an ME-subgroup of $\Gamma$. Take a measure preserving action $\Gamma \times \Lambda \actson (\Om,m)$ such that the actions $\Gamma \actson \Om$ and $\Lambda \actson \Om$ are free and admit a measurable fundamental domain, with the fundamental domain of $\Gamma \actson \Om$ having finite measure. Taking the diagonal product of $\Gamma \times \Lambda \actson \Om$ with a free pmp action of $\Gamma \times \Lambda$, we may assume that $\Gamma \times \Lambda \actson \Om$ is free. Choosing an ergodic component, we may further assume that $\Gamma \times \Lambda \actson \Om$ is ergodic. Put $Z = \Om/\Lambda$, $Y = \Om / \Gamma$ and consider the natural measure preserving actions $\Gamma \actson Z$ and $\Lambda \actson Y$, with the measure on $Y$ being finite. Note that both actions are free and ergodic.

As in \cite[Lemma 3.2 and Theorem 3.3]{Fu98} the free ergodic measure preserving actions $\Gamma \actson Z$ and $\Lambda \actson Y$ are by construction stably orbit equivalent. Denote by $t = m(Z)/m(Y)$ the compression constant of this stable orbit equivalence, where by convention $t = +\infty$ if $Z$ has infinite measure. If $t < 1$, we replace $\Gamma \actson Z$ by $\Gamma \times \Z/n\Z \actson Z \times \Z/n\Z$ for $n$ large enough such that $1/n \leq t$. By part 2 of the proof, $\Gamma \times \Z/n\Z$ still has property ($\ast$). So we may assume that $t \geq 1$. This means that we can find a subset $Z_0 \subset Z$ of finite measure and a measure scaling isomorphism $\theta : Z_0 \recht Y$ such that $\theta(Z_0 \cap \Gamma \cdot z) = \Lambda \cdot \theta(z)$ for a.e.\ $z \in Z_0$.

Since $\Gamma \actson Z$ is ergodic and $Z_0 \subset Z$ is non-negligible, we can choose a measurable map $p : Z \recht Z_0$ such that $p(z) = z$ for a.e.\ $z \in Z_0$ and $p(z) \in \Gamma \cdot z$ for a.e.\ $z \in Z$. Denote by $\om : \Gamma \times Z \recht \Lambda$ the $1$-cocycle for the action $\Gamma \actson Z$ with values in $\Lambda$ determined by
$$\theta(p(g \cdot z)) = \om(g,z) \cdot \theta(p(z)) \quad\text{for all}\;\; g \in \Gamma \;\;\text{and a.e.}\;\; z \in Z \; .$$

Let $\Lambda \actson (X,\mu)$ be any measure preserving action on a standard measure space $(X,\mu)$. Put $B = \rL^\infty(X)$ and $M = B \rtimes \Lambda$. Let $q \in B$ be a projection of finite measure. Assume that $A \subset q M q$ is an abelian von Neumann subalgebra. We have to prove that either $A \prec_{qMq} Bq$ or that the normalizer $\cN_{qMq}(A)\dpr$ is amenable.

Define the free measure preserving action $\Gamma \actson Z \times X$ given by $g \cdot (z,x) = (g \cdot z, \om(g,z) \cdot x)$. Put $\Btil := \rL^\infty(Z \times X)$ and $\Mtil := \Btil \rtimes \Gamma$. We write $p = \chi_{Z_0} \in \rL^\infty(Z)$. By construction, the restriction of the orbit equivalence relation of $\Gamma \actson Z \times X$ to the subset $Z_0 \times X$ is isomorphic, through $\theta \times \id$, with the orbit equivalence relation of the diagonal action $\Lambda \actson Y \times X$. So we find an isomorphism of von Neumann algebras
$$\Psi : (p \ot 1) \Mtil (p \ot 1) \recht \rL^\infty(Y \times X) \rtimes \Lambda$$
satisfying $\Psi(F) = F \circ \theta^{-1}$ for all $F \in \rL^\infty(Z_0 \times X)$. In particular, $\Psi^{-1}(1 \ot q) = p \ot q$. Note that $p \ot q$ is a projection of finite measure in $\Btil$. Put $\Atil := \Psi^{-1}(1 \ot A)$ and note that $\Atil$ is an abelian von Neumann subalgebra of $(p \ot q)\Mtil(p\ot q)$.
Since $\Gamma$ has property ($\ast$), we conclude that either $\Atil$ embeds into $(p \ot q) \Btil (p \ot q)$ inside $(p \ot q)\Mtil (p \ot q)$, or that $\Atil$ has an amenable normalizer inside $(p \ot q)\Mtil (p \ot q)$. Transporting back with $\Psi$, we get that either $1 \ot A$ embeds into $\rL^\infty(Y \times X)(1 \ot q)$ inside $(1 \ot q)(\rL^\infty(Y \times X) \rtimes \Lambda)(1 \ot q)$, or that the normalizer of $1 \ot A$ is amenable. In the first case, it follows that $A$ embeds into $\rL^\infty(X) q$ inside $qMq$. In the second case, we get that $\cN_{qMq}(A)\dpr$ is amenable.
\end{proof}

\section{\boldmath Applications to W$^*$-superrigidity and classification results} \label{sec.superrigid}

We start this section by proving Theorem \ref{thm.class}.

\begin{proof}[Proof of Theorem \ref{thm.class}]
1.\ If $\rL^\infty(X) \rtimes \F_n \cong \rL^\infty(Y) \rtimes \F_m$, it follows from Theorem \ref{thm.Crigidgroups} that the free ergodic pmp actions $\F_n \actson X$ and $\F_m \actson Y$ are orbit equivalent. It then follows from \cite[Th\'{e}or\`{e}me 3.2]{Ga01} that $n = m$.

2.\ In one direction the isomorphism of the II$_1$ factors together with Theorem \ref{thm.Crigidgroups} implies that the actions $\F_n \actson X_0^{\F_n}$ and $\F_m \actson Y_0^{\F_m}$ are stably orbit equivalent with compression constant $s/t$. By \cite[Th\'{e}or\`{e}me 6.3]{Ga01} we get that $(n-1)/(m-1) = s/t$. Conversely assume that $(n-1)/s = (m-1)/t$. Combining \cite[Corollary 1.2]{Bo09a} and \cite[Theorem 1.1]{Bo09b} we know that the actions $\F_n \actson X_0^{\F_n}$ and $\F_m \actson Y_0^{\F_m}$ are stably orbit equivalent with compression constant $(n-1)/(m-1) = s/t$. Hence the crossed product II$_1$ factors are stably isomorphic with amplification constant $s/t$. The result applies in particular to $\rL(\Z \wr \F_n) \cong \rL^\infty\bigl([0,1]^{\F_n}\bigr) \rtimes \F_n$.

3.\ Assume that $\cR_1$ is a treeable countable ergodic pmp equivalence relation and that $\rL \cR_1 \cong \rL \cR_2$ for another pmp equivalence relation $\cR_2$. Let $c \in [1,+\infty]$ be the cost of $\cR_1$. If $c=1$, it follows that $\cR_1$ is amenable. Hence also $\rL \cR_1 \cong \rL \cR_2$ is amenable, so that $\cR_2$ is amenable. So $\cR_1 \cong \cR_2$. If $c \in (1,+\infty]$, take $s > 0$ such that $n:=(c-1)/s$ is a positive integer or $+\infty$. By \cite[Proposition 2.6]{Ga99} the amplification $\cR_1^s$ is treeable with cost $n+1$. By \cite[Corollary 1.2]{Hj05} the equivalence relation $\cR_1^s$ can be implemented by a free action of $\F_{n+1}$. This implies that $\rL(\cR_1^s) = \rL^\infty(Z) \rtimes \F_{n+1}$ for some free ergodic pmp action $\F_{n+1} \actson Z$. Since $\rL(\cR_1^s) \cong \rL(\cR_2^s)$, it follows from Theorem \ref{thm.Crigidgroups} that $\cR_1^s \cong \cR_2^s$, i.e.\ that $\cR_1 \cong \cR_2$.
\end{proof}

As in \cite[Definition 6.1]{PV09} a free ergodic pmp action $\Gamma \actson (X,\mu)$ is called \emph{W$^*$-superrigid} if the following property holds: whenever $\Lambda \actson (Y,\eta)$ is another free ergodic pmp action and $\Theta : \rL^\infty(X) \rtimes \Gamma \recht \rL^\infty(Y) \rtimes \Lambda$ is an isomorphism, the groups $\Gamma$ and $\Lambda$ must be isomorphic, their actions must be conjugate and $\Theta$ is implemented by this conjugacy. More precisely, we find an isomorphism of groups $\delta : \Gamma \recht \Lambda$ and an isomorphism of probability spaces $\Delta : X \recht Y$ such that
\begin{itemize}
\item $\Delta(g \cdot x) = \delta(g) \cdot \Delta(x)$ for all $g \in \Gamma$ and a.e.\ $x \in X$,
\item $U \Theta(a u_g) U^* = \Delta_*(a \om_g) u_{\delta(g)}$ for all $a \in \rL^\infty(X)$ and all $g \in \Gamma$, where $U \in \rL^\infty(Y) \rtimes \Lambda$ is a unitary and $(\om_g)_{g \in \Gamma}$ is a family of unitaries in $\rL^\infty(X)$ defining a $1$-cocycle for $\Gamma \actson X$ with values in $\T$.
\end{itemize}

To formulate the next theorem recall that a pmp action $\Gamma \actson (X,\mu)$ is said to be a \emph{quotient} (or factor) of the pmp action $\Gamma \actson (Y,\eta)$ if there exists a measure preserving map $p : Y \recht X$ such that $p(g \cdot y) = g \cdot p(y)$ for all $g \in \Gamma$ and a.e.\ $y \in Y$. Also recall that a group is said to be \emph{icc} if it has infinite conjugacy classes. Finally recall that a subgroup $\Lambda < \Gamma$ is called co-amenable if $\Gamma / \Lambda$ admits a $\Gamma$-invariant mean. By \cite[Proposition 6]{MP03} a subgroup $\Lambda < \Gamma$ is co-amenable if and only if the subalgebra $\rL(\Lambda) \subset \rL(\Gamma)$ is co-amenable in the sense explained in Section \ref{sec.rel-amen}.

\begin{theorem}\label{thm.superrigid-quotient-bernoulli}
Let $\Gamma_1,\Gamma_2$ be icc weakly amenable groups that admit a proper $1$-cocycle into a nonamenable representation. Put $\Gamma = \Gamma_1 \times \Gamma_2$. Let $\Gamma \actson I$ be a transitive action and $i_0 \in I$. Assume that
\begin{itemize}
\item $\Gamma_1 \cap \Stab i_0 < \Gamma_1$ is not co-amenable,
\item $\Gamma_2 \cap \Stab i_0 < \Gamma_2$ is not of finite index.
\end{itemize}
Then any free ergodic pmp action $\Gamma \actson (X,\mu)$ that arises as a quotient of the generalized Bernoulli action $\Gamma \actson [0,1]^I$ is W$^*$-superrigid.
\end{theorem}

Theorem \ref{thm.superrigid-quotient-bernoulli} will be a consequence of the following similar result for quotients of a Gaussian action $\Gamma \actson (Y_\pi,\mu_\pi)$ associated with an orthogonal representation $\pi$ of $\Gamma$.

\begin{theorem}\label{thm.superrigid-gaussian}
Let $\Gamma_1,\Gamma_2$ be icc weakly amenable groups that admit a proper $1$-cocycle into a nonamenable representation. Put $\Gamma = \Gamma_1 \times \Gamma_2$. Let $\pi : \Gamma \recht \cO(K_\R)$ be any orthogonal representation with corresponding Gaussian action $\Gamma \actson (Y_\pi,\mu_\pi)$. Assume that
\begin{itemize}
\item $\pi_{|\Gamma_1}$ is a nonamenable representation,
\item $\pi_{|\Gamma_2}$ is a weakly mixing representation, i.e.\ a representation without nonzero finite dimensional invariant subspaces.
\end{itemize}
Then any free ergodic pmp action $\Gamma \actson (X,\mu)$ that arises as a quotient of the Gaussian action $\Gamma \actson (Y_\pi,\mu_\pi)$ is W$^*$-superrigid.
\end{theorem}

\begin{remark}\label{rem.superrigid}
Theorem \ref{thm.superrigid-quotient-bernoulli} provides large new families of W$^*$-superrigid actions.
\begin{itemize}
\item In \cite[Theorem A]{Io10} it was shown that a Bernoulli action $\Gamma \actson (X,\mu)$ is W$^*$-superrigid
whenever $\Gamma$ is an
icc property (T) group. In \cite[Theorem 10.1]{IPV10} the same was established when $\Gamma = \Gamma_1 \times \Gamma_2$ is a direct product of a nonamenable icc group $\Gamma_1$ and an infinite icc group $\Gamma_2$. The conditions on $\Gamma_1$, $\Gamma_2$ in Theorem \ref{thm.superrigid-quotient-bernoulli} are of course much stricter, but we now also get W$^*$-superrigidity for generalized Bernoulli actions and their quotients.

\item The following is an interesting class of generalized Bernoulli actions covered by Theorem \ref{thm.superrigid-quotient-bernoulli}. Assume that $\Gamma$ is an icc weakly amenable group that admits a proper $1$-cocycle into a nonamenable representation. Consider the left-right action of $\Gamma \times \Gamma$ on $I = \Gamma$. Since both $\Gamma \times \{e\}$ and $\{e\} \times \Gamma$ act freely on $\Gamma$, the conditions of Theorem \ref{thm.superrigid-quotient-bernoulli} are satisfied and it follows that all free quotient actions of $\Gamma \times \Gamma \actson [0,1]^\Gamma$ are W$^*$-superrigid.

\item Generalized Bernoulli actions typically admit a lot of nonconjugate quotient actions. Indeed, whenever $K$ is a second countable compact group, consider the diagonal action of $K$ on $K^I$ which commutes with the generalized Bernoulli action $\Gamma \actson K^I$. Then $\Gamma \actson K^I / K$ is a quotient action of $\Gamma \actson K^I$.

    When $\Gamma$ and its action $\Gamma \actson I$ satisfy the conditions of Theorem \ref{thm.superrigid-quotient-bernoulli}, we will see in the proof of Theorem \ref{thm.superrigid-gaussian} that $\Gamma \actson K^I$ is cocycle superrigid. Hence it follows from \cite[Lemma 5.2]{PV06} that varying $K$, the actions $\Gamma \actson K^I / K$ are nonconjugate for nonisomorphic compact groups $K$. So by Theorem \ref{thm.superrigid-quotient-bernoulli} also their crossed product II$_1$ factors are nonisomorphic when $K$ varies.

\item As mentioned above, in \cite[Theorem A]{Io10} it was shown that the Bernoulli action $\Gamma \actson (X,\mu)$ is W$^*$-superrigid for all icc property (T) groups $\Gamma$. Theorem \ref{thm.superrigid-quotient-bernoulli} does not cover property (T) groups, but in our forthcoming paper \cite{PV12}, we will cover generalized Bernoulli actions of hyperbolic property (T) groups, as well as all their quotient actions.
\end{itemize}
\end{remark}

\begin{proof}[Proof of Theorem \ref{thm.superrigid-gaussian}]
Let $\Gamma \actson (X,\mu)$ be a free ergodic pmp action that arises as the quotient of a Gaussian action $\Gamma \actson (Y_\pi,\mu_\pi)$ satisfying the assumptions in the theorem. Note that also $\Gamma = \Gamma_1 \times \Gamma_2$ is a weakly amenable group that admits a proper $1$-cocycle into a nonamenable representation. So because of Theorem \ref{thm.Crigidgroups} any isomorphism
$\Theta : \rL^\infty(X) \rtimes \Gamma \recht \rL^\infty(Y) \rtimes \Lambda$ with another group measure space construction satisfies, after a unitary conjugacy, $\Theta(\rL^\infty(X)) = \rL^\infty(Y)$. This means that $\Theta$ is given by a scalar $1$-cocycle (i.e.\ an automorphism of $\rL^\infty(X) \rtimes \Gamma$ that is the identity on $\rL^\infty(X)$) and an isomorphism coming from an orbit equivalence between $\Gamma \actson X$ and $\Lambda \actson Y$. It therefore only remains to argue that $\Gamma \actson (X,\mu)$ is OE superrigid, i.e.\ that this orbit equivalence between $\Gamma \actson X$ and $\Lambda \actson Y$ comes from a conjugacy of the actions.

We claim that the action $\Gamma \actson Y_\pi$ satisfies the hypotheses of \cite[Theorem 1.1]{Po06b}. By \cite[Theorem 1.2]{Fu06} this Gaussian action is s-malleable. Next we have to check that $\Gamma_1 \actson Y_\pi$ has stable spectral gap, i.e.\ that the unitary representation $\Gamma_1 \actson \rL^2(Y_\pi) \ominus \C 1$ is nonamenable. This unitary representation is the direct sum of all $k$-fold ($k \geq 1$) symmetric tensor powers of $\pi_{|\Gamma_1}$. Hence it is a subrepresentation of $\pi_{|\Gamma_1} \ot \rho$, where $\rho$ is defined as the direct sum of all $k$-fold ($k \geq 0$) tensor powers of $\pi_{|\Gamma_1}$. Since $\pi_{|\Gamma_1}$ is nonamenable, also $\pi_{|\Gamma_1} \ot \rho$ is nonamenable and it follows that $\Gamma_1 \actson Y_\pi$ has stable spectral gap. Finally we have to check that $\Gamma_2 \actson Y_\pi$ is weakly mixing, i.e.\ that the unitary representation $\Gamma_2 \actson \rL^2(Y_\pi) \ominus \C 1$ has no nonzero finite dimensional invariant subspaces. This follows with a similar reasoning by using that $\pi_{|\Gamma_2}$ is weakly mixing.

So it follows from \cite[Theorem 1.1]{Po06b} that $\Gamma \actson Y_\pi$ is cocycle superrigid with countable (and even more generally, $\cU_{\text{fin}}$) target groups. Since $\Gamma$ is icc and since $\Gamma \actson Y_\pi$ is weakly mixing (because even $\Gamma_2 \actson Y_\pi$ is weakly mixing as explained above), it follows from \cite[Theorem 5.6]{Po05} that $\Gamma \actson (X,\mu)$ is OE superrigid. So the theorem is proven.
\end{proof}

\begin{proof}[Proof of Theorem \ref{thm.superrigid-quotient-bernoulli}]
The generalized Bernoulli action $\Gamma \actson [0,1]^I$ is isomorphic with the Gaussian action associated to the representation
$\Gamma \overset{\pi}{\actson} \ell^2_\R(I)$. Since $\Gamma \actson I$ is transitive, one has for any $i_0 \in I$ that $\pi_{|\Gamma_1}$ is a multiple of $\Gamma_1 \actson \ell^2(\Gamma_1 / (\Gamma_1 \cap \Stab i_0))$ and that $\pi_{|\Gamma_2}$ is a multiple of $\Gamma_2 \actson \ell^2(\Gamma_2 / (\Gamma_2 \cap \Stab i_0))$. We conclude that
\begin{itemize}
\item $\pi_{|\Gamma_1}$ is nonamenable if and only if $\Gamma_1 \cap \Stab i_0 < \Gamma_1$ is not co-amenable,
\item $\pi_{|\Gamma_2}$ is weakly mixing if and only if $\Gamma_2 \cap \Stab i_0 < \Gamma_2$ is not of finite index.
\end{itemize}
So Theorem \ref{thm.superrigid-quotient-bernoulli} is a direct consequence of Theorem \ref{thm.superrigid-gaussian}.
\end{proof}

Our unique Cartan decomposition theorem \ref{thm.Crigidgroups} can also be coupled with the work of Monod and Shalom \cite{MS02} yielding the result below. To formulate it, recall that an ergodic pmp action $\Gamma \actson (X,\mu)$ is called \emph{aperiodic} if all finite index subgroups of $\Gamma$ still act ergodically. Following \cite[Definition 1.8]{MS02} an ergodic pmp action $\Lambda \actson (Y,\eta)$ is called \emph{mildly mixing} if there are no nontrivial recurrent subsets: if $A \subset Y$ is measurable and $\liminf_{g \recht \infty} \eta(g \cdot A \, \triangle \, A) = 0$, then $\eta(A) = 0$ or $\eta(A)=1$. Note that for a mildly mixing action $\Lambda \actson (Y,\eta)$ all infinite subgroups of $\Lambda$ act ergodically on $(Y,\eta)$.

\begin{theorem}\label{thm.strong-rigidity-monod-shalom}
Let  $\Gamma = \F_n \times \F_m$, for some $2 \leq n, m \leq \infty$. Assume that $\Gamma \actson (X,\mu)$ is a free ergodic pmp action that is aperiodic and irreducible, meaning that both $\F_n$ and $\F_m$ act ergodically on $(X,\mu)$.

If $\rL^\infty(X) \rtimes \Gamma \cong \rL^\infty(Y) \rtimes \Lambda$ for any free mildly mixing pmp action $\Lambda \actson (Y,\eta)$, then $\Gamma \cong \Lambda$ and the actions $\Gamma \actson X$ and $\Lambda \actson Y$ are conjugate.
\end{theorem}
\begin{proof}
Since $\Gamma$ is a product of free groups, Theorem \ref{thm.Crigidgroups} applies. So the existence of an isomorphism $\rL^\infty(X) \rtimes \Gamma \cong \rL^\infty(Y) \rtimes \Lambda$ implies that $\Gamma \actson (X,\mu)$ and $\Lambda \actson (Y,\eta)$ are orbit equivalent. Since free groups belong to the class $\cC_{\text{reg}}$ of Monod and Shalom,
it follows from \cite[Theorem 1.10]{MS02} that the groups $\Gamma$ and $\Lambda$ must be isomorphic and that their actions must be conjugate.
\end{proof}

We finally prove Theorem \ref{thm.rigid-outer}.

\begin{proof}[Proof of Theorem \ref{thm.rigid-outer}]
Assume that $\theta : R \rtimes \Gamma \recht R \rtimes \Lambda$ is a $*$-isomorphism. As in the proof of Theorem \ref{thm.Crigidgroups} it follows that $\theta(R) \prec R$ and $R \prec \theta(R)$. By \cite[Lemma 8.4]{IPP05} the subfactors $\theta(R)$ and $R$ are unitarily conjugate. So after a unitary conjugacy we may assume that $\theta(R) = R$. This precisely means that the actions $\Gamma \actson R$ and $\Lambda \actson R$ are cocycle conjugate.
\end{proof}

\begin{remark}
Theorems \ref{thm.class} and \ref{thm.rigid-outer} say that for $n \neq m$ we have $P \rtimes \F_n \not\cong Q \rtimes \F_m$ both in the case of free ergodic pmp actions on abelian von Neumann algebras, and in the case of outer actions on the hyperfinite II$_1$ factor. As illustrated by the following natural example, the result fails for arbitrary properly outer trace preserving actions.

Let $\pi : \F_2 \recht \Z/2\Z$ be a surjective homomorphism and let $\Z/2\Z$ act nontrivially on a set with two points. Denote the composition with $\pi$ as $(\si_g)_{g \in \F_2}$. Take any outer action $(\al_g)_{g \in \F_2}$ of $\F_2$ on the hyperfinite II$_1$ factor $R$. Consider the action $\al_g \ot \si_g$ of $\F_2$ on $R \ot \C^2$. Identify $\F_3 = \Ker \pi$ and consider the action $\id \ot \al_g$ of $\F_3$ on $\M_2(\C) \ot R$. One canonically has
$$(R \ot \C^2) \rtimes \F_2 \cong (\M_2(\C) \ot R) \rtimes \F_3 \; .$$
\end{remark}

\end{document}